\documentclass[twoside,a4paper,reqno,11pt]{amsart} 
\usepackage{amsfonts, amsbsy, amsmath, amssymb, latexsym}
\usepackage{mathrsfs,array}
\usepackage[top=30mm,right=30mm,bottom=30mm,left=30mm]{geometry}
\usepackage{stmaryrd}
\usepackage{bm}

\usepackage{hyperref}

\renewcommand{\a}{\alpha}
\renewcommand{\b}{\beta}

\newcommand{\e}{\epsilon}
 
\renewcommand{\l}{\lambda} \renewcommand{\O}{\Omega} 

 \renewcommand{\to}{\rightarrow}
 \newcommand{\s}{\sigma}

\newcommand{\G}{\bar{G}}

\newcommand{\la}{\langle}
\newcommand{\ra}{\rangle}

\newcommand{\leqs}{\leqslant}
\newcommand{\geqs}{\geqslant}

 \newcommand{\vs}{\vspace{3mm}}

\makeatletter
\newcommand{\imod}[1]{\allowbreak\mkern4mu({\operator@font mod}\,\,#1)}
\makeatother

\newtheorem{theorem}{Theorem} 
\newtheorem*{conj*}{Conjecture}

\newtheorem{thm}{Theorem}[section] 
\newtheorem{prop}[thm]{Proposition} 
\newtheorem{lem}[thm]{Lemma}
\newtheorem{cor}[thm]{Corollary}

\theoremstyle{definition}
\newtheorem{rem}[thm]{Remark}
\newtheorem{ex}[thm]{Example}
\newtheorem{remk}{Remark}

\newtheorem*{def-non}{Definition}

\begin{document}

\author{Timothy C. Burness}
 \address{T.C. Burness, School of Mathematics, University of Bristol, Bristol BS8 1TW, UK}
 \email{t.burness@bristol.ac.uk}
 
 \author{Adam R. Thomas}
 \address{A.R. Thomas, School of Mathematics, University of Bristol, Bristol, BS8 1TW, UK, and The Heilbronn Institute for Mathematical Research, Bristol, UK.}
 \email{adam.thomas@bristol.ac.uk}
 
\title[The involution fixity of exceptional groups]{On the involution fixity of \\ exceptional groups of Lie type}

\begin{abstract}
The involution fixity ${\rm ifix}(G)$ of a permutation group $G$ of degree $n$ is the maximum number of fixed points of an involution. In this paper we study the involution fixity of primitive almost simple exceptional groups of Lie type. We show that if $T$ is the socle of such a group, then either ${\rm ifix}(T) > n^{1/3}$, or ${\rm ifix}(T) = 1$ and $T = {}^2B_2(q)$ is a Suzuki group in its natural $2$-transitive action of degree $n=q^2+1$. This bound is best possible and we present more detailed results for each family of exceptional groups, which allows us to determine the groups with ${\rm ifix}(T) \leqs n^{4/9}$. This extends recent work of Liebeck and Shalev, who established the bound ${\rm ifix}(T) > n^{1/6}$ for every almost simple primitive group of degree $n$ with socle $T$ (with a prescribed list of exceptions). Finally, by combining our results with the Lang-Weil estimates from algebraic geometry, we determine bounds on a natural analogue of involution fixity for primitive actions of exceptional algebraic groups over algebraically closed fields.
\end{abstract}

\subjclass[2010]{Primary 20G41, 20B15; Secondary 20E28, 20E45} 
\keywords{Finite exceptional groups; primitive actions; involution fixity}
\date{\today}
\maketitle

\section{Introduction}\label{s:intro}

Let $G \leqs {\rm Sym}(\O)$ be a permutation group on a finite set $\O$ of size $n$ and recall that the \emph{minimal degree} $\mu(G)$ of $G$ is defined to be the minimal number of points moved by a non-identity element of $G$. This notion has been extensively studied since the nineteenth century, with many classical results in the context of lower bounds for primitive groups (see Bochert \cite{Boch} and Jordan \cite{Jo}, for example). We refer the reader to \cite{Babai, GM, LS91} for more recent results and applications. 

The related notion of \emph{fixity} was first studied by Ronse in \cite{Ronse}. The fixity $f(G)$ of $G$ is the maximal number of points fixed by a non-identity element, that is, $f(G) = n - \mu(G)$. In \cite{SS}, Saxl and Shalev study primitive groups of bounded fixity. One of their main results, \cite[Theorem 1.3]{SS}, shows that if $G$ is primitive with fixity $f$, then either $G$ has a solvable subgroup whose index is bounded by a function of $f$, or $G$ is almost simple with socle ${\rm L}_{2}(q)$ or ${}^2B_2(q)$ in their $2$-transitive actions of degree $n=q + 1$ or $q^2 + 1$, respectively. More recently, this study has been extended by Liebeck and Shalev in \cite{LSh}. For example, \cite[Theorem 1]{LSh} provides detailed information on the structure of a primitive permutation group $G$ of degree $n$ with $f(G) < n^{1/6}$. To do this, they first  use the O'Nan-Scott theorem for primitive groups to establish a reduction to affine and almost simple groups, and they handle the affine groups by applying earlier work on the fixity of linear groups \cite{LSh11}. This leaves the almost simple groups. Here the main result is \cite[Theorem 4]{LSh}, which shows that the socle of an almost simple primitive group contains an \emph{involution} fixing at least $n^{1/6}$ points (with specified exceptions). This leads us naturally to the following notion.

\begin{def-non}
Let $G \leqs {\rm Sym}(\O)$ be a permutation group on a finite set $\O$.  The \emph{involution fixity} ${\rm ifix}(G)$ of $G$ is the maximal number of points fixed by an involution in $G$.
\end{def-non}

As noted in \cite{LSh}, the involution fixity of permutation groups has been studied in several papers, going back to work of Bender \cite{Bender} in the early 1970s, who classified the transitive groups $G$ with ${\rm ifix}(G) = 1$ (these are the groups for which a point stabilizer is a so-called \emph{strongly embedded} subgroup). In terms of applications, bounds on the involution fixity of almost simple primitive groups have recently been used in the context of finite geometry to study point-primitive generalized quadrangles (see \cite{BPP}).

In this paper, we investigate the involution fixity of almost simple primitive permutation groups with socle an exceptional group of Lie type, with the aim of establishing accurate lower bounds. This complements recent work of Covato \cite{Cov}, who has studied the analogous problem for almost simple groups with socle a sporadic, alternating or classical group. A simplified version of our main result is the following.

\begin{theorem}\label{t:main}
Let $G$ be an almost simple primitive permutation group of degree $n$ with socle $T$, an exceptional group of Lie type over $\mathbb{F}_{q}$. Then either
\begin{itemize}\addtolength{\itemsep}{0.2\baselineskip}
\item[{\rm (i)}] ${\rm ifix}(T) > n^{1/3}$; or
\item[{\rm (ii)}] $(T,n)=({}^2B_2(q),q^2+1)$ or $({}^2G_2(3)',9)$, and ${\rm ifix}(T)=1$.
\end{itemize}
\end{theorem}

Note that ${}^2G_2(3)' \cong {\rm L}_{2}(8)$, so the second example in part (ii) can be viewed as a classical group. Also note that ${\rm ifix}(G) \geqs {\rm ifix}(T)$, so the theorem gives a lower bound on the involution fixity of $G$. 

\begin{remk}\label{r:main}
The lower bound in part (i) is best possible in the sense that for any $\e>0$, there are infinitely many groups with $n^{1/3} < {\rm ifix}(T)< n^{1/3+\e}$. For instance, if we consider the action of $G = {}^2G_2(q)'$ on the set of cosets of a Borel subgroup, then Theorem \ref{t:parab} states that ${\rm ifix}(T)>n^{1/3}$ and 
\[
\liminf_{q \to \infty}\frac{\log {\rm ifix}(T)}{\log n} = \frac{1}{3}.
\]
The same conclusion holds when $T=G_2(q)$, $p=3$ and $G$ contains a graph or graph-field automorphism (again, with respect to the action of $G$ on the cosets of a Borel subgroup). 
\end{remk}

With a view towards applications, it is desirable to strengthen the bound in part (i) of Theorem \ref{t:main}, at the possible expense of some additional (known) exceptions. In \cite{Cov}, Covato determines the almost simple primitive groups with socle $T$ (a sporadic, alternating or classical group) such that ${\rm ifix}(T) \leqs n^{4/9}$. Our next theorem establishes the analogous result for exceptional groups of Lie type.

\begin{theorem}\label{t:main2}
Let $G$ be an almost simple primitive permutation group of degree $n$ with socle $T$, an exceptional group of Lie type over $\mathbb{F}_{q}$, and point stabilizer $H$. Set $H_0 = H \cap T$ and $q=p^f$ with $p$ a prime. Then ${\rm ifix}(T) \leqs n^{4/9}$ if and only if $(T,H_0)$ appears in Table $\ref{tab:main}$. In each of these cases, ${\rm ifix}(T) \geqs n^{\a}$ for the given constant $\a$.
\end{theorem}

{\small 
\renewcommand{\arraystretch}{1.2}
\begin{table}
\begin{center}
$$\begin{array}{llcl}\hline
T & H_0 & \a & \mbox{Conditions} \\ \hline
F_4(q) & (q^2+q+1)^2.(3 \times {\rm SL}_2(3)) & 0.427 & q = 2,\, G = {\rm Aut}(T) \\
G_2(q)' & P_1 & 2/5 & \mbox{$q \geqs 7$ odd} \\
& P_2 & 2/5 & q \geqs 7 \\
& P_{1,2} & 1/3 & \mbox{$p = 3$, $q \geqs 9$} \\
& G_2(q_0) & 3/7 & \mbox{$q=q_0^k$ odd, $k$ prime} \\
& {}^2G_2(q) & 2/5 & \mbox{$q=3^{2m+1}$, $m \geqs 0$} \\
& 2^3.{\rm L}_{3}(2) & 3/7 & q=p \geqs 11 \\
& {\rm U}_{3}(3){:}2 & 3/7 & q=p \geqs 11 \\
& {\rm L}_{2}(13) & 3/7 & \mbox{$q \geqs 17$, $p \ne 13$, $\mathbb{F}_{q} = \mathbb{F}_{p}[\sqrt{13}]$} \\
& {\rm L}_{2}(8) & 3/7 & \mbox{$q \geqs 23$, $p \geqs 5$, $\mathbb{F}_{q} = \mathbb{F}_{p}[\omega]$, $\omega^3-3\omega+1 = 0$} \\
& 3^{1+2}.8 & 0.416 & q=2 \\
{}^3D_4(q) & P_{1,3,4} & 4/11 & \mbox{$q$ odd} \\
& {}^3D_4(q_0) & 3/7 & \mbox{$q=q_0^k$ odd, $k$ prime}  \\
& G_2(q) & 3/7 & \mbox{$q$ odd} \\
& {\rm PGL}_{3}(q) & 2/5 & \mbox{$q$ odd} \\
& {\rm PGU}_{3}(q) & 0.387 & \mbox{$q$ odd} \\
& (q^4-q^2+1).4 & 0.436 & q = 2 \\
& (q^2+q+1)^2.{\rm SL}_{2}(3) & 0.401 & q = 2,3  \\
& (q^2-q+1)^2.{\rm SL}_{2}(3) & 0.351 & q = 2, 3, 4 \\
{}^2G_2(q)' & q^{1+1+1}{:}(q-1) & 1/3 & \\
& {}^2G_2(q_0) & 0.426 & \mbox{$q=q_0^k$, $k$ prime}  \\
& (q-\sqrt{3q}+1){:}6 & 0.442 & q=27 \\
& D_{18} & 0.416 & q=3  \\
& D_{14} & 0.386 & q=3  \\
& 2^3{:}7 & 0 & q=3  \\
{}^2B_2(q) & q^{1+1}{:}(q-1) & 0 & \\ 
 & {}^2B_2(q_0) & 0.397 & \mbox{$q=q_0^k$, $k$ prime}  \\
 & (q+\sqrt{2q}+1){:}4 & 0.438 & q=8 \\ 
 & (q-\sqrt{2q}+1){:}4 & 0.380 & q=8,32 \\ \hline
\end{array}$$
\caption{The cases with $n^{\a} < {\rm ifix}(T) \leqs n^{4/9}$ in Theorem \ref{t:main2}}
\label{tab:main}
\end{center}
\end{table}
\renewcommand{\arraystretch}{1}}

\begin{remk}
Let us make some comments on the statement of Theorem \ref{t:main2}.
\begin{itemize}\addtolength{\itemsep}{0.2\baselineskip}
\item[{\rm (a)}] The focus on the bound ${\rm ifix}(T) \leqs n^{4/9}$ is somewhat arbitrary, although there are infinitely many examples where it is essentially best possible. For instance, if $G = {}^3D_4(q)$ and $H=P_2$ is a maximal parabolic subgroup, 
then ${\rm ifix}(T)>n^{4/9}$ and 
\[
\liminf_{q \to \infty} \frac{\log {\rm ifix}(T)}{\log n} = \frac{4}{9}
\]
(see Theorem \ref{t:parab}). 
\item[{\rm (b)}] More generally, there are many examples with ${\rm ifix}(T)<n^{1/2}$. In the special cases where $H$ is a maximal parabolic, subfield or exotic local subgroup, the relevant groups are determined in Theorems \ref{t:parab}, \ref{t:subfield} and \ref{t:exotic}, respectively. 
\item[{\rm (c)}] For the cases in Table \ref{tab:main}, the lower bound ${\rm ifix}(T) \geqs n^{\a}$ is best possible in the sense that the given constant $\a$ is either accurate to $3$ decimal places (as a lower bound which is valid for each value of $q$ recorded in the final column of the table), or 
\[
\liminf_{q \to \infty} \frac{\log {\rm ifix}(T)}{\log n} = \a.
\]
As noted in Section \ref{s:small}, if $T = {}^2G_2(q)$ and $H_0 = {}^2G_2(q_0)$ is a subfield subgroup then 
\begin{equation}\label{e:eqqq}
\liminf_{q \to \infty} \frac{\log {\rm ifix}(T)}{\log n} = \b
\end{equation}
with $\b=3/7$. Similarly, \eqref{e:eqqq} holds with $\b=2/5$ if $T = {}^3D_4(q)$ and $H_0 = {\rm PGU}_{3}(q)$ (with $q$ odd), and also if $T = {}^2B_2(q)$ and $H_0 = {}^2B_2(q_0)$.
\end{itemize}
\end{remk}

Finally, we turn to asymptotic results. Let $G$ be an almost simple exceptional group of Lie type over $\mathbb{F}_{q}$ with socle $T = (\bar{G}_{\s})'$, where $\bar{G}$ is a simple exceptional algebraic group (of type $E_8$, $E_7$, $E_6$, $F_4$ or $G_2$) over the algebraic closure of $\mathbb{F}_{q}$ and $\s$ is an appropriate Steinberg endomorphism of $\bar{G}$. For the purposes of Theorem \ref{t:main3} below, let us say that a maximal subgroup $H$ of $G$ is \emph{algebraic} if it is of the form 
$H = N_G(\bar{H}_{\s})$, where $\bar{H}$ is a $\s$-stable closed positive dimensional
subgroup of $\bar{G}$ (equivalently, $H$ is algebraic if it is either a maximal parabolic subgroup, or a maximal rank subgroup as in \cite{LSS}, or if it is a non-maximal rank subgroup as in part (I) of Theorem \ref{t:types}). In the second column of Table \ref{tab:main2} we refer to the \emph{type} of $H_0$, which is simply the socle of $H_0$ for the non-parabolic subgroups in the table.

\begin{theorem}\label{t:main3}
Let $G$ be an almost simple primitive permutation group of degree $n$ with socle $T$, an exceptional group of Lie type over $\mathbb{F}_{q}$, and point stabilizer $H$, an algebraic subgroup of $G$. Set $H_0 = H \cap T$ and $q=p^f$ with $p$ a prime. Then 
\[
\liminf_{q \to \infty}\frac{\log {\rm ifix}(T)}{\log n} = \b 
\]
and either $\b \geqs 1/2$, or $(T,H_0,\b)$ is one of the cases listed in Table 
$\ref{tab:main2}$. 
\end{theorem}

{\small 
\renewcommand{\arraystretch}{1.2}
\begin{table}
\begin{center}
$$\begin{array}{llcl}\hline
T & \mbox{Type of $H_0$} & \b & \mbox{Conditions} \\ \hline
E_8(q) & \Omega_5(q) & 58/119 & p \geqs 5 \\
& {\rm L}_{2}(q) & 17/35 &  \mbox{$3$ classes; $p \geqs 23, 29, 31$} \\
E_7(q) & {\rm L}_{2}(q) & 31/65 & \mbox{$2$ classes; $p \geqs 17, 19$} \\
E_6^{\e}(q) & {\rm L}_{3}^{\e}(q^3) & 13/27 & \\
F_4(q) & G_2(q) & 9/19 & p = 7 \\
& {\rm L}_{2}(q) & 23/49 & p \geqs 13 \\
G_2(q) & P_1 &  2/5 & p \geqs 3 \\ 
 & P_2 &  2/5 & \\
 & P_{1,2} &  1/3 & p = 3 \\
 & {\rm L}_2(q) & 5/11 & p \geqs 7 \\
 {}^2F_4(q) & P_{2,3} & 5/11 & \\
{}^2G_2(q)' & q^{1+1+1}{:}(q-1)  & 1/3 & \\ \hline
\end{array}$$
\caption{The cases with $\b<1/2$ in Theorem \ref{t:main3}}
\label{tab:main2}
\end{center}
\end{table}
\renewcommand{\arraystretch}{1}}

The proof of Theorem \ref{t:main3} follows by combining the asymptotic statements in Theorems \ref{t:parab}, \ref{t:mr} and \ref{t:nonmr} on parabolic, maximal rank and lower rank algebraic subgroups, respectively (see Remark \ref{r:thm3}).

We can interpret Theorem \ref{t:main3} at the level of algebraic groups. Let $\bar{G}$ be a simple exceptional algebraic group over an algebraically closed field, let $\bar{H}$ be a closed subgroup of $\bar{G}$ and consider the natural action of $\bar{G}$ on the coset variety $\O = \bar{G}/\bar{H}$. For each $t \in \bar{G}$, the set of fixed points $C_{\O}(t)$ is a subvariety and  
\[
{\rm ifix}(\bar{G}) = \max\{\dim C_{\O}(t) \,:\, \mbox{$t \in \bar{G}$ is an involution}\}
\]
is a natural definition for the involution fixity of $\bar{G}$ (note that $\bar{G}$ has only finitely many conjugacy classes of involutions, so this is well-defined). The algebraic group analogue of Theorem \ref{t:main3} is the following (in Table \ref{tab:main3}, it is convenient to set $p = \infty$ when the underlying field has characteristic zero).

\begin{theorem}\label{t:main4}
Let $\bar{G}$ be a simple exceptional algebraic group over an algebraically closed field $K$ of characteristic $p \geqs 0$, let $\bar{H}$ be a maximal positive dimensional closed subgroup of $\bar{G}$ and consider the action of $\bar{G}$ on $\O = \bar{G}/\bar{H}$. Then either 
\begin{itemize}\addtolength{\itemsep}{0.2\baselineskip}
\item[{\rm (i)}] ${\rm ifix}(\bar{G}) \geqs \frac{1}{2}\dim \O$; or
\item[{\rm (ii)}] ${\rm ifix}(\bar{G}) = \gamma\dim \O$ and $(\bar{G},\bar{H},\gamma)$ is one of the cases listed in Table $\ref{tab:main3}$.
\end{itemize}
\end{theorem}

{\small 
\renewcommand{\arraystretch}{1.2}
\begin{table}
\begin{center}
$$\begin{array}{llcl}\hline
\bar{G} & \bar{H} & \gamma & \mbox{Conditions} \\ \hline
E_8 & B_2 & 58/119 & p \geqs 5 \\
& A_1 & 17/35 &  \mbox{$3$ classes; $p \geqs 23, 29, 31$} \\
E_7 & A_1 & 31/65 & \mbox{$2$ classes; $p \geqs 17, 19$} \\
F_4 & G_2 & 9/19 & p = 7 \\
& A_1 & 23/49 & p \geqs 13 \\
G_2 & P_1 &  2/5 & p \ne 2 \\ 
 & P_2 &  2/5 & \\ 
 & A_1 & 5/11 & p \geqs 7 \\ \hline
\end{array}$$
\caption{The constant $\gamma$ in Theorem \ref{t:main4}(ii)}
\label{tab:main3}
\end{center}
\end{table}
\renewcommand{\arraystretch}{1}}

The short proof combines Theorem \ref{t:main3} with the well known Lang-Weil estimate \cite{LW} from algebraic geometry (see Section \ref{s:alg} for the details). We refer the reader to \cite{Bur1} for some related results for simple algebraic groups.

\vs

Let us outline the main steps in the proof of Theorems \ref{t:main}, \ref{t:main2} and \ref{t:main3}. Let $G$ be an almost simple primitive permutation group with socle $T = (\bar{G}_{\s})'$, an exceptional group of Lie type over $\mathbb{F}_{q}$ (where $q=p^f$ with $p$ a prime), and point stabilizer $H$. By primitivity, $H$ is a maximal subgroup of $G$ with $G = HT$.
The subgroup structure of finite groups of Lie type has been intensively studied in recent years, and through the work of many authors, the maximal subgroups of $G$ can be partitioned into several naturally defined subgroup collections (in some sense, this is analogous to Aschbacher's celebrated subgroup structure theorem for classical groups \cite{Asch}). This powerful reduction theorem provides the framework for the proofs of our main results, which are organised according to the type of $H$. Roughly speaking, one of the following holds (see Theorem \ref{t:types}): 
\begin{itemize}\addtolength{\itemsep}{0.2\baselineskip}
\item[{\rm (a)}] $H = N_G(\bar{H}_{\s})$, where $\bar{H}$ is a $\s$-stable closed positive dimensional subgroup of $\bar{G}$;
\item[{\rm (b)}] $H$ is almost simple.
\end{itemize}

First we focus on the subgroups arising in (a), which we divide into three further subcollections, according to the structure of $\bar{H}$: 
\begin{itemize}\addtolength{\itemsep}{0.2\baselineskip}
\item[{\rm (i)}] Parabolic subgroups;
\item[{\rm (ii)}] Reductive subgroups of maximal rank;
\item[{\rm (iii)}] Reductive subgroups of smaller rank. 
\end{itemize}
In Section \ref{s:parab}, $H$ is a maximal parabolic subgroup and we use character-theoretic methods (following \cite{LLS2}) to compute ${\rm fix}(t)$ for each involution $t \in T$ (the specific details depend heavily on the parity of the underlying characteristic $p$). Next, in Sections \ref{s:mr} and \ref{s:algebraic}, we handle the reductive subgroups arising in (ii) and (iii) above. Typically, we identify a specific involution $t \in H_0$ (the possibilities for $H_0 = H \cap T$ are known) and then determine the $T$-class of $t$, which yields a lower bound
$${\rm ifix}(T) \geqs {\rm fix}(t) = \frac{|t^T \cap H_0|}{|t^T|}\cdot n  \geqs \frac{|t^{H_0}|}{|t^T|}\cdot n$$
(see Section \ref{s:inv}).
We can often identify the $T$-class of $t$ by computing the Jordan form of $t$ on a suitable $\bar{G}$-module $V$ (to do this, we work with the restriction 
$V{\downarrow}\bar{H}^{\circ}$). When $p=2$, earlier work of Lawther \cite{Lawunip} on unipotent classes is also very useful. This approach relies heavily on the extensive literature concerning conjugacy classes of involutions in finite groups of Lie type, some of which is briefly summarised in Section \ref{s:inv}.

Next we handle subfield subgroups and exotic local subgroups (see \cite{CLSS}) in Sections \ref{s:subfield} and \ref{s:exotic}, respectively (both cases are straightforward). Finally, to complete the proof, we turn to the remaining almost simple maximal subgroups $H$ of $G$ in Section \ref{s:as}. Let $S$ be the socle of $H$ and let ${\rm Lie}(p)$ be the set of finite simple groups of Lie type in characteristic $p$. Here the analysis naturally falls into two cases, according to whether or not $S \in {\rm Lie}(p)$. Although a complete list of maximal almost simple subgroups is not available, in general, there are strong restrictions on the possibilities that can arise (see Theorem \ref{t:simples}). For each possible socle $S$, we choose an involution $t \in S$ so that $|t^S|$ is maximal. By assuming $t$ is contained in the largest class of involutions in $T$, we obtain an upper bound on $|t^T|$ and this yields a lower bound on ${\rm ifix}(T)$ as above. Note that we may have very little information on the embedding of $S$ in $T$, so it is difficult to determine the precise $T$-class of $t$, in general. 

Finally, in Section \ref{s:small} we handle a few extra cases that arise when $T$ is a Suzuki, Ree or Steinberg triality group, and we give a proof of Theorem \ref{t:main4} in Section \ref{s:alg}.

\vs

Let us say a few words on our notation. Let $A$ and $B$ be groups and let $n$ and $p$ be positive integers, with $p$ a prime. We denote a cyclic group of order $n$ by $Z_n$ (or just $n$) and $A^n$ is the direct product of $n$ copies of $A$. In particular, $p^n$ denotes an elementary abelian $p$-group of order $p^n$. We write $[n]$ for an unspecified solvable group of order $n$. A split extension of $A$ by $B$ will be denoted by writing $A{:}B$, while $A.B$ is an arbitrary extension. We write $A \circ B$ for a central product of $A$ and $B$. If $A$ is finite, then ${\rm soc}(A)$ denotes the socle of $H$ (that is, the subgroup of $A$ generated by its minimal normal subgroups) and $i_2(A)$ is the number of involutions in $A$.
For matrices, it will be convenient to write $[A_1^{n_1}, \ldots, A_k^{n_k}]$ for a block-diagonal matrix with a block $A_i$ occurring with multiplicity $n_i$. In addition, $J_i$ will denote a standard unipotent Jordan block of size $i$. 

For finite simple groups, we adopt the notation in \cite{KL}. For example, we will write ${\rm L}_{n}^{+}(q) = {\rm L}_{n}(q)$ and ${\rm L}_{n}^{-}(q) = {\rm U}_{n}(q)$ for ${\rm PSL}_{n}(q)$ and ${\rm PSU}_{n}(q)$, respectively, and similarly $E_6^{+}(q) = {}E_6(q)$ and $E_6^{-}(q) = {}^2E_6(q)$, etc. For a semisimple algebraic group $\bar{G}$ of rank $r$ we write $\Phi(\bar{G})$, $\Phi^{+}(\bar{G})$ and $\Pi(\bar{G}) = \{\a_1, \ldots, \a_r\}$ for the set of roots, positive roots and simple roots of $\bar{G}$, with respect to a fixed Borel subgroup. We will often associate $\bar{G}$ with $\Phi(\bar{G})$ by writing $\bar{G}=A_r$, etc., and we will use $W(\bar{G})$ and $\mathcal{L}(\bar{G})$ for the Weyl group and Lie algebra of $\bar{G}$, respectively. We adopt the standard Bourbaki \cite{Bou} labelling of  simple roots and we will write $\{\l_1, \ldots, \l_r\}$ for the corresponding fundamental dominant weights of $\bar{G}$. If $\lambda$ is a dominant weight, then $V_{\bar{G}}(\lambda)$ (or simply $V(\lambda)$) is the rational irreducible $\bar{G}$-module with highest weight $\lambda$ (the trivial module will be denoted by $0$ and for $\bar{G}=A_1$ we will write $V(n)$ rather than $V(n\l_1)$). Similarly, $W(\lambda)$ will denote the corresponding Weyl module for $\bar{G}$. Finally, if $V$ is a $\bar{G}$-module then $V^n$ will denote the direct sum of $n$ copies of $V$. We will also use $V{\downarrow}\bar{H}$ for the restriction of $V$ to a subgroup $\bar{H}$ of $\bar{G}$. 

\vs

\noindent \textbf{Acknowledgments.} We thank David Craven, Martin Liebeck and Alastair Litterick for helpful discussions. We also thank an anonymous referee for several useful comments which have improved the clarity of the paper.

\section{Involutions}\label{s:inv}

We start by recalling some preliminary results which will be needed for the proofs of our main theorems. In particular, it will be convenient to record some of the information we need on involutions in simple exceptional groups of Lie type (in terms of conjugacy classes and centralizers); our main references are \cite[Section 4.5]{GLS} (in odd characteristic) and \cite{AS} (for even characteristic).

Let $G$ be an almost simple primitive permutation group of degree $n$ with socle $T$ and point stabilizer $H$ (by primitivity, $H$ is a maximal subgroup of $G$ and $G = HT$). Set $H_0 = H \cap T$. We define 
$${\rm ifix}(G) = \max\{{\rm fix}(t) \,:\, \mbox{$t \in G$ is an involution}\}$$
to be the \emph{involution fixity} of $G$, where ${\rm fix}(t)$ is the number of fixed points of $t$ on $\O$. It is easy to see that $n = |T:H_0|$ and 
\begin{equation}\label{e:fix}
{\rm fix}(t) = \frac{|t^T \cap H_0|}{|t^T|} \cdot n
\end{equation}
for all $t \in T$ (see \cite[Lemma 2.5]{LS91}, for example). We will repeatedly make use of this  observation throughout the paper, without further comment.

Let $T$ be an exceptional group of Lie type over $\mathbb{F}_{q}$, where $q=p^f$ with $p$ a prime. A convenient list of the relevant groups and their orders is given in \cite[Table 5.1.B]{KL}. Before going any further, it will be convenient to establish Theorem \ref{t:main2} in the special cases where
\begin{equation}\label{e:small}
T \in  \{G_2(2)', {}^2G_2(3)', {}^2F_4(2)'\}.
\end{equation}
Note that $G_2(2)' \cong {\rm U}_{3}(3)$ and ${}^2G_2(3)'  \cong {\rm L}_{2}(8)$.

\begin{prop}\label{p:small}
Let $G$ be an almost simple primitive permutation group of degree $n$, with socle $T$ and point stabilizer $H$. If $T \in \{G_2(2)', {}^2G_2(3)', {}^2F_4(2)'\}$ then either ${\rm ifix}(T)>n^{4/9}$, or $(T,H_0)$ is one of the cases in Table $\ref{tab:small}$ and ${\rm ifix}(T) \geqs n^{\a}$ for the given constant $\a$.
\end{prop}

\begin{proof}
This is a straightforward {\sc Magma} \cite{Magma} calculation, using a suitable permutation representation of $G$ given in the Web Atlas \cite{WebAt} (indeed, it is easy to compute the exact value of ${\rm ifix}(T)$ in every case). 
\end{proof}

{\small 
\renewcommand{\arraystretch}{1.2}
\begin{table}
\begin{center}
$$\begin{array}{llccc}\hline
T & H_0 & n & {\rm ifix}(T) & \a  \\ \hline
G_2(2)' & 3^{1+2}.8 & 28 & 4  & 0.416 \\
{}^2G_2(3)' & D_{14} & 36 & 4  & 0.386 \\
& D_{18} & 28 & 4 & 0.416 \\
& 2^3{:}7 & 9 & 1 & 0 \\ \hline
\end{array}$$
\caption{The constant $\a$ in Proposition \ref{p:small}}
\label{tab:small}
\end{center}
\end{table}
\renewcommand{\arraystretch}{1}}

For the remainder of the paper, we may assume that $T$  is not one of the groups in \eqref{e:small}. Write $T = (\bar{G}_{\sigma})'$, where $\bar{G}$ is a simple adjoint algebraic group over the algebraic closure $K = \bar{\mathbb{F}}_{q}$ and $\sigma:\bar{G} \to \bar{G}$ is an appropriate Steinberg endomorphism. Note that $\bar{G}_{\s} = {\rm Inndiag}(T)$ is the group of inner-diagonal automorphisms of $T$ (see \cite[Definition 2.5.10]{GLS}). Given the formula in \eqref{e:fix}, we require information on the conjugacy classes of involutions in $T$ and we will need to understand the fusion of $H_0$-classes in $T$. This analysis falls naturally into two cases, according to the parity of $p$ (indeed, involutions are unipotent when $p=2$, and semisimple when $p \ne 2$). 

There is an extensive literature on conjugacy classes in exceptional groups, both finite and algebraic. First assume $p=2$. In \cite[Sections 12-18]{AS}, Aschbacher and Seitz provide detailed information on the classes of involutions in simple groups of Lie type, in terms of  representatives and centralizer structure. For exceptional groups, we will use the notation of Liebeck and Seitz \cite[Tables 22.2.1--22.2.7]{LS_book} for class representatives, which is consistent with the Bala-Carter labelling of unipotent classes (since we are interested in involutions, we need only consider the classes whose Bala-Carter label has only $A_1$ or $\tilde{A}_1$ components). For instance, $F_4(q)$ has four classes of involutions, labelled as follows:  
$$A_1, \; \tilde{A}_1, \; (\tilde{A}_1)_2,\; A_1\tilde{A}_1$$ 
(see \cite[Table 22.2.4]{LS_book}). Note that the labels $A_1$ and $\tilde{A}_1$ indicate long and short root elements, respectively. The twisted groups ${}^3D_4(q)$ and ${}^2B_2(q)$ are not included in the tables in \cite[Section 22]{LS_book}, but it is well known that ${}^3D_4(q)$ has two classes of involutions and $^2{}B_2(q)$ has a unique class (see \cite{Spal, Suz}, for example). For $p \ne 2$, we refer the reader to \cite[Sections 4.3-4.6]{GLS} for an in-depth treatment of involutions in groups of Lie type (both finite and algebraic). We will use the notation of \cite[Table 4.5.1]{GLS} for class representatives. 

\begin{rem}\label{r:invcon}
Let $t \in {\rm Inndiag}(T)$ be an involution. For $p=2$ we observe that $t$ lies in $T$ and $t^T = t^{{\rm Inndiag}(T)}$. Similarly, \cite[Theorem 4.2.2(j)]{GLS}  implies that $t^T = t^{{\rm Inndiag}(T)}$ when $p \ne 2$, but note that if $T = E_7(q)$ then there are involutions in ${\rm Inndiag}(T)$ that do not lie in $T$ (see Remark \ref{r:inv} below). In all characteristics, we find that there is a $1$-$1$ correspondence between conjugacy classes of involutions in $\bar{G}$ and those in $T$. 
\end{rem}

Most of the information we need on involutions in exceptional groups is summarised in the following proposition.

\begin{prop}\label{l:excepinv}
For each finite simple exceptional group of Lie type $T$, Table $\ref{tabinv}$ lists a complete set of representatives of the $T$-classes of involutions $t \in T$, and each row of the table gives the order of the centralizer $C_{{\rm Inndiag}(T)}(t)$. 
\end{prop}

\begin{rem}\label{r:inv}
In the third column of Table \ref{tabinv} we present a representative of each $T$-class of involutions in $T$. The case $T = E_7(q)$ with $q$ odd requires special attention. Here ${\rm Inndiag}(T) = T.2$ and the ambient algebraic group $\bar{G} = E_7$ has three classes of involutions, with centralizers $A_1D_6$, $A_7.2$ and $T_1E_6.2$ (here $T_1$ denotes a $1$-dimensional torus). The latter two classes each give rise to two ${\rm Inndiag}(T)$-classes of involutions, one in $T$ and the other in ${\rm Inndiag}(T) \setminus T$, so in total there are five classes of involutions in ${\rm Inndiag}(T)$, but only three in $T$. More precisely, using the notation in Table \ref{tabinv}, the involutions in $T$ are as follows (see \cite[Table 4.5.1]{GLS}):
\[
\left\{\begin{array}{ll}
t_1, t_4, t_7 & \mbox{if $q \equiv 1 \imod{4}$} \\
t_1, t_4', t_7' & \mbox{if $q \equiv 3 \imod{4}$} 
\end{array}\right.
\]
\indent In the final column of the table, we give the order of the centralizer of $t$ in ${\rm Inndiag}(T)$. To avoid any possible confusion, note that this is always a monic polynomial in $q$, with the exception of the elements $t_4,t_4',t_7,t_7' \in E_7(q)$ with $q$ odd, where the order is twice a monic polynomial.
\end{rem}

We will also need some information on involutions in the low-rank classical groups. Let $L$ be a classical group over $\mathbb{F}_q$ with natural module $V$ of dimension $n$. First assume $p=2$. Unipotent involutions in $L$ are essentially determined up to conjugacy by their Jordan form on $V$, which is of the form $[J_2^s,J_1^{n-2s}]$ for some $s \geqs 1$ (if $s$ is even, there are two such classes when $L$ is a symplectic or orthogonal group, denoted by the labels $a_s$ and $c_s$ in \cite{AS}). We will frequently use the notation from \cite{AS} for class representatives. 

We will often be interested in the case where $x \in L$ is a long root element (that is, an element with Jordan form $[J_2^s,J_1^{n-2s}]$, where $s=2$ if $L = \Omega_{n}^{\e}(q)$ is an orthogonal group (an $a_2$ involution in the notation of \cite{AS}), and $s=1$ in the other cases); the conjugacy class size of such an element is as follows (for example, see the proof of \cite[Proposition 3.22]{Bur2}):
$$|x^L| = \left\{\begin{array}{ll}
(q^{n-1}-\e^{n-1})(q^n-\e^n)/(q-\e) & L = {\rm L}_{n}^{\e}(q) \\
q^n-1 & L = {\rm Sp}_{n}(q) \\
(q^{n/2-2}+\e)(q^{n-2}-1)(q^{n/2}-\e)/(q^2-1) & \mbox{$L = \Omega_{n}^{\e}(q)$, $n \geqs 8$ even}
\end{array}\right.$$
We will refer to long root elements as $A_1$-involutions (in the notation of \cite{AS}, these are $b_1$-involutions when $L = {\rm Sp}_{n}(q)$ and $a_2$-involutions when $L = \Omega_{n}^{\e}(q)$), which is consistent with our usage for exceptional groups.

For $p$ odd, the conjugacy class of an involution in $L$ is essentially determined by its eigenvalues (in a suitable splitting field) on the natural module $V$. We refer the reader to \cite[Table 4.5.1]{GLS} for information on class representatives and centralizers (the relevant details are conveniently summarised in \cite[Appendix B]{BG}). 

{\small 
\renewcommand{\arraystretch}{1.2}
\begin{table}
\begin{center}
\begin{tabular}{lllll}\hline
$T$ & $q$ & $t$ & $|C_{\text{Inndiag}(T)}(t)|$ \\ \hline
$E_8(q)$ & odd & $t_1$ & $|\text{SO}^+_{16}(q)|$  \\
& & $t_8$ & $|\text{SL}_2(q)||E_7(q)|$ \\
& even & $A_1$ & $q^{57}|E_7(q)|$ \\
& & $A_1^2$ & $q^{78}|\text{Sp}_{12}(q)|$ \\
& & $A_1^3$ & $q^{81}|\text{SL}_2(q)||F_4(q)|$ \\
& & $A_1^4$ & $q^{84}|\text{Sp}_8(q)|$ \\

$E_7(q)$ & odd & $t_1$ & $|\text{SL}_2(q)||{\rm SO}^+_{12}(q)|$   \\
& & $t_4$, $t_4'$ & $2|\text{SL}^{\e}_{8}(q)|^{\dagger}$  \\
& & $t_7$, $t_7'$ & $2(q-\e)|E^\e_{6}(q)|^{\dagger}$  \\
& even & $A_1$ & $q^{33}|\Omega_{12}^+(q)|$ \\
& & $A_1^2$ & $q^{42}|\text{SL}_2(q)||\text{Sp}_8(q)|$ \\
& & $(A_1^3)^{(1)}$ & $q^{27}|F_4(q)|$ \\
& & $(A_1^3)^{(2)}$ & $q^{45}|\text{SL}_2(q)||\text{Sp}_6(q)|$ \\
& & $A_1^4$ & $q^{42}|\text{Sp}_6(q)|$ \\

$E^\e_6(q)$ & odd & $t_1$ & $(q-\e)|\text{SO}^{\e}_{10}(q)|$ \\
& & $t_2$ & $|\text{SL}_2(q)||\text{SL}^\e_{6}(q)|$ \\
& even & $A_1$ & $q^{21}|\text{SL}^{\e}_6(q)|$ \\
& & $A_1^2$ & $q^{24}(q - \e)|\text{Sp}_6(q)|$ \\
& & $A_1^3$ & $q^{27}|\text{SL}_2(q)||\text{SL}^{\e}_3(q)|$ \\

$F_4(q)$ & odd & $t_1$ & $|\text{SL}_2(q)||\text{Sp}_{6}(q)|$ \\
& & $t_4$ & $|\text{SO}_9(q)|$ \\
& even & $A_1$ & $q^{15}|\text{Sp}_6(q)|$ \\
& & $\tilde{A}_1$ & $q^{15}|\text{Sp}_6(q)|$ \\
& & $(\tilde{A}_1)_{2}$ & $q^{20}|\text{Sp}_4(q)|$ \\
& & $A_1 \tilde{A}_1$ & $q^{18}|\text{SL}_2(q)|^2$ \\

$G_2(q)$, $q>2$ & odd & $t_1$ & $|\text{SL}_2(q)|^2$ \\
 & even & $A_1$ & $q^5 |\text{SL}_2(q)|$ \\
& & $\tilde{A}_1$ & $q^3 |\text{SL}_2(q)|$ \\

${}^3D_4(q)$ & odd & $t_2$ & $|\text{SL}_2(q)||\text{SL}_2(q^3)|$\\
& even & $A_1$ & $q^9 |\text{SL}_2(q^3)|$ \\
& & $A_1^3$ & $q^9 |\text{SL}_2(q)|$ \\

${}^2 F_4(q)$, $q>2$ & even & $(\tilde{A}_1)_{2}$ & $q^{10} |{}^2B_2(q)|$ \\
 & & $A_1 \tilde{A}_1$ & $q^9|\text{SL}_2(q)|$ \\

${}^2G_2(q)$, $q>3$ & odd & $t_1$ & $|\text{SL}_2(q)|$ \\

${}^2B_2(q)$ & even & $(\tilde{A}_1)_2$ & $q^2$ \\ \hline
\multicolumn{4}{l}{\mbox{{\small $^{\dagger}$ Here $q \equiv \e \imod{4}$; see Remark \ref{r:inv}}}} \\
& & & \\
\end{tabular}
\caption{Involution classes in finite simple exceptional groups of Lie type}
\label{tabinv}
\end{center}
\end{table}
\renewcommand{\arraystretch}{1.2}
}

\section{Parabolic subgroups}\label{s:parab}

Let $G$ be an almost simple primitive permutation group of degree $n$, with socle $T$ and point stabilizer $H$, where $T$ is a group of exceptional Lie type over $\mathbb{F}_{q}$ with $q=p^f$ for a prime $p$. In this section, we will use character-theoretic methods to study the involution fixity of $T$ in the special case where $H$ is a maximal parabolic subgroup of $G$ (this approach was first introduced in \cite{LLS2}). We adopt the standard notation $P_{i,j...}$ for the parabolic subgroup of $T$ obtained by removing the nodes $i,j, \ldots$ from the corresponding Dynkin diagram of $T$, in terms of the standard Bourbaki labelling in \cite{Bou}. Set  $H_0 = H \cap T$.

\begin{rem}\label{r:small}
In view of Proposition \ref{p:small}, throughout this section we will assume $T$ is not one of the groups in \eqref{e:small}.
\end{rem}

Our main result for parabolic actions is the following, which establishes a strong form of Theorem \ref{t:main2} (and thus Theorem \ref{t:main}) in this special case. In Table \ref{tabb:par}, $\tau$ denotes a graph or graph-field automorphism of $T$.

\begin{thm}\label{t:parab}
Let $G$ be an almost simple primitive permutation group of degree $n$ with socle $T$, an exceptional group of Lie type over $\mathbb{F}_{q}$, and point stabilizer $H$, a maximal parabolic subgroup of $G$. Then one of the following holds:
\begin{itemize}\addtolength{\itemsep}{0.2\baselineskip}
\item[{\rm (i)}] $T = {}^2B_2(q)$, $H_0 = q^{1+1}{:}(q-1)$ and ${\rm ifix}(T)=1$; 
\item[{\rm (ii)}] ${\rm ifix}(T)>n^{1/2}$ and 
\[
\liminf_{q\to \infty} \frac{\log {\rm ifix}(T)}{\log n } \geqs \frac{1}{2};
\]
\item[{\rm (iii)}] $(T,H_0)$ is one of the cases in Table $\ref{tabb:par}$ and we have ${\rm ifix}(T)>n^{\gamma}$ and
\[
\liminf_{q\to \infty} \frac{\log {\rm ifix}(T)}{\log n } = \gamma,
\]
for the given constant $\gamma$.   
\end{itemize}
In particular, if $(T,H_0)$ is in Table $\ref{tabb:par}$, then ${\rm ifix}(T) \leqs n^{4/9}$ unless 
\begin{itemize}\addtolength{\itemsep}{0.2\baselineskip}
\item[{\rm (a)}] $(T,H_0) = ({}^3D_4(q), P_2)$, $({}^2F_4(q),P_{2,3})$; or 
\item[{\rm (b)}] $T = G_2(q)$ and $q<7$.
\end{itemize}
\end{thm}

{\small
\renewcommand{\arraystretch}{1.2}
\begin{table}
\begin{center}
$$\begin{array}{llcl}\hline
T & H_0 & \gamma & \mbox{Conditions} \\ \hline  
G_2(q) & P_1 & 2/5 & \mbox{$q$ odd} \\
& P_2 & 2/5 & \\
& P_{1,2}  & 1/3 & \mbox{$p=3$, $q \geqs 9$, $\tau \in G$} \\
{}^3D_4(q) & P_{2} & 4/9 & \\
& P_{1,3,4} & 4/11 & \mbox{$q$ odd} \\
{}^2F_4(q) & P_{2,3} & 5/11 &  \\
{}^2G_2(q) & q^{1+1+1}{:}(q-1)  & 1/3 & \\ \hline
\end{array}$$
\caption{The parabolic actions with $1<{\rm ifix}(T) \leqs n^{1/2}$}
\label{tabb:par}
\end{center}
\end{table}
\renewcommand{\arraystretch}{1}}

Write $T = (\G_{\s})'$, where $\bar{G}$ is a simple algebraic group over $K = \bar{\mathbb{F}}_{q}$ and $\s$ is an appropriate Steinberg endomorphism of $\G$. Then $H = N_G(P_{\s})$, where $P$ is a $\s$-stable parabolic subgroup of $\G$, and we can consider the permutation character $\chi = 1^{\G_{\s}}_{P_{\s}}$, which gives $\chi(t) = {\rm fix}(t)$ for all $t \in T$. Let $W = W(\bar{G})$ be the Weyl group of $\G$. We partition the proof of Theorem \ref{t:parab} into two cases, according to the parity of $p$.

\subsection{Unipotent involutions}\label{ss:uni}

First assume $p=2$ and $T$ is untwisted, so 
$$T \in \{E_8(q),E_7(q),E_6(q),F_4(q),G_2(q)\}.$$
According to \cite[Lemma 2.4]{LLS2}, the permutation character $\chi$ can be decomposed as a sum  
\begin{equation}\label{e:dec}
\chi = \sum_{\phi \in \widehat{W}}n_{\phi}R_{\phi},
\end{equation}
where $\widehat{W}$ is the set of complex irreducible characters of $W$. Here the $R_{\phi}$ are almost characters of $\G_{\s}$ and the coefficients are given by $n_{\phi} = \la 1^W_{W_P},\phi \ra$, where $W_P$ is the corresponding parabolic subgroup of $W$. The restriction of the $R_{\phi}$ to unipotent elements yields the Green functions of $\G_{\s}$, which are polynomials in $q$ with non-negative coefficients. L\"{u}beck \cite{Lub} has implemented an algorithm of Lusztig \cite{Lus} to compute the relevant Green functions and this allows us to calculate ${\rm fix}(t)$ for every involution $t \in T$. 
We refer the reader to \cite[Section 2]{LLS2} for further details.

In the following proposition, we assume $T$ is untwisted and $H_0$ is a maximal subgroup of $T$. Note that the latter condition simply means that we exclude the cases
$$(T,H_0) \in \{(E_6(q),P_{1,6}),  (E_6(q), P_{3,5}),  (F_4(q), P_{1,4}),  (F_4(q), P_{2,3})\},$$
which can arise when $G$ contains a graph or graph-field automorphism of $T$.

\begin{prop}\label{p:parabuni}
Suppose $p=2$, $T$ is untwisted and $H$ is a maximal parabolic subgroup of $G$ such that $H_0$ is maximal in $T$. Then  
${\rm ifix}(T)>n^{\a}$ and
\begin{equation}\label{e:al}
\lim_{q\to \infty,\, \mbox{{\rm {\scriptsize $q$ even}}}} \frac{\log {\rm ifix}(T)}{\log n } = \a
\end{equation}
where $\a$ is given in Table $\ref{tab:par1}$.
In particular, ${\rm ifix}(T) \leqs n^{1/2}$ if and only if $T=G_2(q)$ and $H_0=P_2$, in which case ${\rm ifix}(T) \leqs n^{4/9}$ if $q \geqs 8$.
\end{prop}

\begin{proof} 
First observe that the decomposition in \eqref{e:dec} is given in \cite[pp.413--415]{LLS2} (with respect to Carter's labelling of irreducible characters in \cite[Section 13.2]{Carter}). Let $t \in T$ be an $A_1$-involution (see Table \ref{tabinv}). As noted above, we can use \cite{Lub} to compute $\chi(t) = {\rm fix}(t)$ and verify the lower bound on ${\rm ifix}(T)$. Alternatively, since $t$ is a long root element, we can use \cite[Proposition 2.1]{LLS2}, which does not require \cite{Lub}.

For example, suppose $T = E_6(q)$ and $H_0=P_{2}$. Here $|P_{\s}| = q^{21}|{\rm GL}_{6}(q)|$, so
$$n = |\bar{G}_{\s}: P_{\s}| = (q^6+q^3+1)(q^6+1)(q^4+1)(q^3+1)(q^2+q+1)$$
and \cite[p.414]{LLS2} gives
$$\chi = R_{\phi_{1,0}}+R_{\phi_{6,1}}+ R_{\phi_{20,2}}+R_{\phi_{30,3}}+R_{\phi_{15,4}}.$$
Note that $T$ has three classes of involutions, labelled $A_1$, $A_1^2$ and $A_1^3$ (see Table \ref{tabinv}). 

For an $A_1$-involution $t \in T$, using \cite{Lub}, we compute
$$\chi(t) = q^{15}+2q^{14}
+3q^{13}+4q^{12}+5q^{11}+6q^{10}+5q^{9}+5q^8+5q^7+4q^6+3q^5+3q^4+2q^3+q^2+q+1$$
and thus ${\rm ifix}(T)>n^{\a}$ with $\a=5/7$. Alternatively, in the notation of \cite[Proposition 2.1]{LLS2}, we have $P_0 = P$, $lht(\a_0)=11$, $L'=A_5$ and we get 
\begin{align*}
\chi(t) = & \; q^{11}(q^{-11}+q^{-10}+q^{-9}+2q^{-8}+3q^{-7}+3q^{-6}+4q^{-5}+5q^{-4}+5q^{-3}+5q^{-2}+6q^{-1}) \\
& \; + q^{10}(q^5+2q^4+3q^3+4q^2+5q),
\end{align*}
which agrees with the expression above. By repeating the calculation for involutions in the classes labelled $A_1^2$ and $A_1^3$, we deduce that ${\rm ifix}(T) = {\rm fix}(t)$ when $t$ is an $A_1$-involution and thus \eqref{e:al} holds. (In fact, one can show that if $t$ is an $A_1$-involution, then ${\rm fix}(t) \geqs {\rm fix}(u)$ for all non-trivial unipotent elements $u \in T$; see \cite[p.413]{LLS2}.)

The case $T=G_2(q)$ with $H_0 = P_2$ requires special attention. Here $q \geqs 4$ (see Remark \ref{r:small}), $n = (q^4+q^2+1)(q+1)$ and $T$ has two classes of involutions, labelled $A_1$ and $\tilde{A}_1$ in Table \ref{tabinv}, corresponding to long and short root elements, respectively. For $t \in A_1$ we get ${\rm fix}(t) = q^2+q+1$, whereas ${\rm fix}(t) = q^2+2q+1$ for $t \in \tilde{A}_1$. We conclude that  $n^{2/5}<{\rm ifix}(T)<n^{1/2}$ and $\a=2/5$ in \eqref{e:al}. Note that if $q=4$ then 
$$\frac{\log {\rm ifix}(T)}{\log n} = \frac{\log 25}{\log 1365} > 0.445,$$
whereas ${\rm ifix}(T)<n^{4/9}$ for $q \geqs 8$. 

The remaining cases are similar and we omit the details. It is worth noting that in every case, we find that ${\rm fix}(t)$ is maximal when $t$ is an $A_1$-involution, unless $T = F_4(q)$ and $H_0 = P_1$ or $P_2$, for which short root elements (that is, $\tilde{A}_1$-involutions) have the largest number of fixed points.
\end{proof}

{\small
\renewcommand{\arraystretch}{1.5}
\begin{table}
\begin{center}
$$\begin{array}{c|cccccccc}
& P_1 & P_2 & P_3 & P_4 & P_5 & P_6 & P_7 & P_8 \\ \hline
E_8(q) & \frac{10}{13},\, \frac{43}{78} & \frac{35}{46},\, \frac{51}{92} & \frac{75}{98},\, \frac{27}{49} & \frac{81}{106},\,\frac{29}{53} & \frac{79}{104},\,\frac{57}{104} & \frac{75}{97},\,\frac{53}{97} & \frac{65}{83},\,\frac{47}{83} & \frac{15}{19},\,\frac{11}{19} \\
E_7(q) & \frac{25}{33},\,\frac{7}{11} & \frac{31}{42},\, \frac{11}{21} & \frac{35}{47},\,\frac{29}{47} & \frac{39}{53},\,\frac{31}{53} & \frac{37}{50},\, \frac{13}{25} & \frac{16}{21},\,\frac{4}{7} & \frac{7}{9},\,\frac{5}{9} & \\
E_6(q) & \frac{3}{4},\,\frac{5}{8} & \frac{5}{7},\,\frac{13}{21} & \frac{18}{25},\,\frac{3}{5} & \frac{20}{29},\,\frac{17}{29} & \frac{18}{25},\,\frac{3}{5} & \frac{3}{4},\,\frac{5}{8} & & \\
F_4(q) & \frac{11}{15},\, \frac{11}{15} & \frac{7}{10},\,\frac{7}{10} & \frac{7}{10},\,\frac{13}{20} & \frac{11}{15},\,\frac{2}{3} & & & & \\
G_2(q) & \frac{3}{5},\, \frac{2}{5} & \frac{2}{5},\,\frac{2}{5} & & & & & & 
\end{array}$$
\caption{The constants $\a,\b$ in Propositions \ref{p:parabuni} and \ref{p:parabsemi}}
\label{tab:par1}
\end{center}
\end{table}
\renewcommand{\arraystretch}{1}}

Next we assume $T = E_6(q)$ or $F_4(q)$, and $H_0$ is a non-maximal subgroup of $T$. 

\begin{prop}\label{p:parab2}
Suppose $p=2$ and $(T,H_0)$ is one of the cases in Table $\ref{tab:par2}$. Then ${\rm ifix}(T)>n^{\a}$ and \eqref{e:al} holds, where $\a$ is given in the third column of the table.
\end{prop}

\begin{proof}
We proceed as in the proof of the previous proposition. First we express $\chi$ as a sum of almost characters, as in \eqref{e:dec}; see Table \ref{tab:decc}. Next we use \cite{Lub} to compute $\chi(t)$ for an $A_1$-involution $t \in T$. For example, if $T=F_4(q)$ and $H_0 = P_{2,3}$ then 
\[
n = \frac{(q^{12}-1)(q^8-1)(q^4+q^2+1)}{(q-1)^2}
\]
and
\begin{align*}
\chi(t) = & \; q^{15}+3q^{14}+6q^{13}+10q^{12}+14q^{11}+17q^{10}+19q^9+19q^8+18q^7+15q^6 \\
& \; +12q^5+9q^4+6q^3+4q^2+2q+1,
\end{align*}
which gives ${\rm ifix}(T)>n^{15/22}$. 

Alternatively, as in the proof of the previous proposition, we can use \cite[Proposition 2.1]{LLS2} to compute $\chi(t)$ for an $A_1$-involution $t$. In the notation of \cite[Section 2]{LLS2}, we have $P_0 = P_1$, $lht(\a_0)=8$, $L'=A_1\tilde{A}_1$ and we get
$$\chi(t) = \frac{|(P_0)_{\s}|}{|P_{\s}|} \cdot (q^8(q^{-8}+q^{-7}+q^{-6}+q^{-5}+2q^{-4}+2q^{-3}+2q^{-2}+2q^{-1})+q^8)$$
with 
$$\frac{|(P_0)_{\s}|}{|P_{\s}|} = (q^5+q^4+q^3+q^2+q+1)(q^2+1),$$
which is in agreement with the above expression. By carrying out a similar calculation for elements in the other three involution classes, one checks that ${\rm fix}(t)$ is maximal when $t$ is an $A_1$-involution and the result follows. 

The other cases are very similar and we omit the details. 
\end{proof}

{\small
\renewcommand{\arraystretch}{1.2}
\begin{table}
\begin{center}
$$\begin{array}{lll} \hline
T & H_0 & 1^W_{W_P} \\ \hline
E_6(q) & P_{1,6} & \phi_{1,0}+2\phi_{6,1}+\phi_{15,5}+3\phi_{20,2}+\phi_{24,6}+\phi_{30,3}+2\phi_{64,4} \\
 & P_{3,5} & \phi_{1,0}+2\phi_{6,1}+\phi_{10,9}+3\phi_{15,4}+\phi_{15,5}+5\phi_{20,2}+2\phi_{24,6} +4\phi_{30,3}+\phi_{60,11}+6\phi_{60,5} \\
& & +3\phi_{60,8}+6\phi_{64,4}+3\phi_{80,7} +4\phi_{81,6}+\phi_{81,10}+2\phi_{90,8} \\
F_4(q) & P_{1,4} & \phi_{1,0}+\phi_{2,4}'+\phi_{2,4}''+\phi_{4,8}+2\phi_{4,1}+\phi_{6,6}' + 2\phi_{8,3}' + 2\phi_{8,3}''+\phi_{9,6}'+3\phi_{9,2}+\phi_{9,6}'' \\
& & +\phi_{12,4}+2\phi_{16,5} \\
 & P_{2,3} & \phi_{1,0}+\phi_{2,4}'+\phi_{2,4}''+\phi_{4,7}'+\phi_{4,7}''+\phi_{4,8} + 2\phi_{4,1}+ \phi_{6,6}''+2\phi_{6,6}'+3\phi_{8,3}' +\phi_{8,9}' \\
& & +3\phi_{8,3}''+\phi_{8,9}'' + 2\phi_{9,6}'+4\phi_{9,2}+2\phi_{9,6}''+\phi_{9,10}+3\phi_{12,4}+4\phi_{16,5} \\ \hline
\end{array}$$
\caption{The decomposition of $\chi$ in \eqref{e:dec}}
\label{tab:decc}
\end{center}
\end{table}
\renewcommand{\arraystretch}{1}}

{\small 
\renewcommand{\arraystretch}{1.2}
\begin{table}
\begin{center}
$$\begin{array}{llcc} \hline
T & H_0 & \a & \b \\ \hline
E_6(q) & P_{1,6} & 3/4 & 7/12\\
& P_{3,5} & 22/31 & 18/31 \\
F_4(q) & P_{1,4} & 7/10 & - \\
& P_{2,3} & 15/22 & - \\ 
G_2(q) & P_{1,2} & - & 1/3 \\ \hline
\end{array}$$
\caption{The constants $\a,\b$ in Propositions \ref{p:parab2} and \ref{p:parabsemi}}
\label{tab:par2}
\end{center}
\end{table}
\renewcommand{\arraystretch}{1}}

To complete the proof of Theorem \ref{t:parab} for unipotent involutions, we handle the twisted groups $T \in \{{}^2B_2(q), {}^2G_2(q), {}^2F_4(q), {}^2E_6(q)\}$. 

\begin{prop}\label{p:parab3}
Suppose $p=2$, $T$ is twisted and $H$ is a maximal parabolic subgroup of $G$. Then one of the following holds:
\begin{itemize}\addtolength{\itemsep}{0.2\baselineskip}
\item[{\rm (i)}] $T ={}^2B_2(q)$ and ${\rm ifix}(T) = 1$;
\item[{\rm (ii)}] $(T,H_0,\a)=({}^3D_4(q), P_{2}, 4/9)$ or $({}^2F_4(q),P_{2,3},5/11)$, where $n^{\a}<{\rm ifix}(T) \leqs n^{1/2}$ and $\a$ is the constant in \eqref{e:al};
\item[{\rm (iii)}] ${\rm ifix}(T)>n^{1/2}$ and \eqref{e:al} holds with $\a$ given in Table $\ref{tab:par3}$.
\end{itemize}
\end{prop}

\begin{proof}
For $T = {}^2E_6(q)$ we proceed as in the proofs of the previous two propositions, using \cite{Lub} and the decomposition of $\chi$ in \eqref{e:dec} (see \cite[p.125]{BLS}). We omit the details. If $T = {}^2B_2(q)$ then $\chi = 1+\psi$, where $\psi$ is the Steinberg character of $T$. Since $\psi(x)=0$ for all non-identity unipotent elements $x \in T$, we deduce that ${\rm ifix}(T)=1$. 

Next assume that $T = {}^2F_4(q)$ with $q \geqs 8$ (recall that we are excluding the case $T = {}^2F_4(2)'$, which was handled in Proposition \ref{p:small}), so $H_0 = P_{1,4}$ or $P_{2,3}$, and we have
\[
P_{1,4} = [q^{10}]{:}({}^2B_2(q) \times (q-1)),\;\; P_{2,3} = [q^{11}]{:}{\rm GL}_{2}(q).
\]
The decomposition of $\chi$ as a sum of unipotent characters is given in \cite[p.416]{LLS2}, so we can use \cite{Mal0} to compute ${\rm fix}(t)$ for any involution $t \in T$. For example, if $H_0 = P_{1,4}$ then 
\[
n = (q^6+1)(q^3+1)(q+1)
\]
and
\[
\chi = \chi_1+\chi_{2}+\chi_9+\chi_{10}+\chi_{11}
\]
in the notation of \cite[Tabelle 1]{Mal0}. There are two $T$-classes of involutions, with representatives labelled $u_1$ and $u_2$ in \cite{Mal0} (in terms of Table \ref{tabinv}, $u_1 = (\tilde{A}_1)_2$ and $u_2 = A_1\tilde{A}_1$). We calculate  
\begin{align*}
\chi(u_1) = & \; 1+q(q^3-q+1)+\frac{1}{4}q^2(q-\sqrt{2q}+1)(2q^3+\sqrt{2q^5}+\sqrt{2q^3}+q+1) \\
& \; +\frac{1}{4}q^2(q+\sqrt{2q}+1)(2q^3-\sqrt{2q^5}-\sqrt{2q^3}+q+1)+\frac{1}{2}q^2(q^2+1) \\
= & \; q^6+q^4+q^3+q+1
\end{align*}
and $\chi(u_2) = q^4+q^3+q+1$. Therefore 
\[
\frac{\log {\rm ifix}(T)}{\log n} \geqs \frac{\log 266761}{\log 1210323465} > 0.597
\]
and $\a=3/5$ in \eqref{e:al}. Similarly, if $H_0 = P_{2,3}$ then ${\rm ifix}(T) = q^5+q^4+q^3+q^2+1 > n^{5/11}$ and \eqref{e:al} holds with $\a=5/11$.

Finally, let us turn to the case $T = {}^3D_4(q)$, so $H_{0} = P_{1,3,4}$ or $P_2$, where
\[
P_{1,3,4} = [q^{11}]{:}((q^3-1) \times {\rm SL}_{2}(q)),\;\; P_2 = q^{1+8}{:}((q-1) \times {\rm SL}_{2}(q^3)).
\]
There are two classes of involutions in $T$, with representatives labelled $u_1$ and $u_2$ (these correspond to the classes labelled $A_1$ and $A_1^3$ in Table \ref{tabinv}). First assume $H_0 = P_{1,3,4}$, so 
\[
n = (q^8+q^4+1)(q^3+1)
\]
In the notation of \cite{Spal}, we have 
$\chi = 1+\rho_1+\rho_2+\e_2$
(see \cite[p.416]{LLS2}) and by inspecting \cite[Table 2]{Spal} we calculate
\[
\chi(u_1) = q^7+q^4+q^3+1,\;\; \chi(u_2) = q^4+q^3+1
\]
Therefore ${\rm ifix}(T)>n^{7/11}$ and $\a=7/11$ in \eqref{e:al}. Similarly, if $H_0=P_2$ then $\chi = 1+\rho_1+\rho_2+\e_1$, ${\rm ifix}(T)=q^4+q^3+q+1>n^{4/9}$ and we have $\a=4/9$ in \eqref{e:al}.
\end{proof}

{\small 
\renewcommand{\arraystretch}{1.2}
\begin{table}
\begin{center}
$$\begin{array}{llcc} \hline
T & H_0 & \a & \b \\ \hline
{}^2E_6(q) & P_{1,6} & 3/4 & 7/12\\
& P_2 & 13/21 & 13/21 \\
& P_{3,5} & 17/31 & 17/31 \\
& P_4 & 17/29 & 17/29 \\
{}^2F_4(q) & P_{1,4} & 3/5 & - \\
& P_{2,3} & 5/11 & - \\
{}^3D_4(q) & P_{1,3,4} & 7/11 & 4/11 \\
& P_2 & 4/9 & 4/9 \\ 
{}^2G_2(q) & P_{1,2} & - & 1/3 \\ \hline
\end{array}$$
\caption{The constants $\a,\b$ in Propositions \ref{p:parab3} and \ref{p:parabsemi}}
\label{tab:par3}
\end{center}
\end{table}
\renewcommand{\arraystretch}{1}}

\subsection{Semisimple involutions}\label{ss:semi}

Now assume $p$ is odd. Once again, we can use character-theoretic methods to compute ${\rm fix}(t)$ for every involution $t \in T$. The key tool to do this is \cite[Corollary 3.2]{LLS2}. The following result completes the proof of Theorem \ref{t:parab}.

\begin{prop}\label{p:parabsemi}
Suppose $p \ne 2$ and $H$ is a maximal parabolic subgroup of $G$. Then 
\begin{equation}\label{e:al2}
\liminf_{q\to \infty,\, \mbox{{\rm {\scriptsize $q$ odd}}}} \frac{\log {\rm ifix}(T)}{\log n } = \b
\end{equation}
with $\b$ as in Tables $\ref{tab:par1}$, $\ref{tab:par2}$ and $\ref{tab:par3}$, and either
\begin{itemize}\addtolength{\itemsep}{0.2\baselineskip}
\item[{\rm (i)}] ${\rm ifix}(T)>n^{1/2}$; or
\item[{\rm (ii)}] $(T,H_0)$ is one of the cases in Table $\ref{tabb:par}$ and ${\rm ifix}(T)>n^{\gamma}$ for the given constant $\gamma$.
\end{itemize}
Moreover, if $(T,H_0)$ is in Table $\ref{tabb:par}$, then ${\rm ifix}(T) \leqs n^{4/9}$ unless 
$(T,H_0) = ({}^3D_4(q), P_2)$, or $T = G_2(q)$ and $q \in \{3,5\}$.
\end{prop}

\begin{proof}
First assume $T = {}^2G_2(q)$, so $H_0 = q^{1+1+1}{:}(q-1)$, $T$ has a unique class of involutions (with centralizer $2 \times {\rm L}_{2}(q)$) and $\chi = 1+\psi$, where $\psi$ is the Steinberg character. Therefore, for any involution $t \in T$ we have $\chi(t) = 1+|C_{T}(t)|_{p} = 1+q$ and we conclude that $n^{1/3}<{\rm ifix}(T)<n^{4/9}$ and $\b=1/3$ in \eqref{e:al2}.

Next assume $T = {}^3D_4(q)$. Here $T$ has $q^8(q^8+q^4+1)$ involutions, which form a single conjugacy class $t^T$. By arguing as in the proof of Proposition \ref{p:parab3}, using \cite{Spal} (see the second column in the table on p.689) with the given decomposition of $\chi$ into unipotent characters, we calculate that 
\[
{\rm ifix}(T) = \left\{\begin{array}{ll}
q^4+2q^3+q+2 & \mbox{if $H_0 = P_{1,3,4}$}  \\
q^4+q^3+2q+2 & \mbox{if $H_0 = P_{2}$}
\end{array}\right.
\]
It follows that $n^{4/9}<{\rm ifix}(T)<n^{1/2}$ and $\b=4/9$ if $H_0 = P_2$. Similarly, if $H_0 = P_{1,3,4}$ then $n^{4/11}<{\rm ifix}(T)<n^{4/9}$ and $\b=4/11$.

For the remainder, we may assume that either $T$ is untwisted or $T = {}^2E_6(q)$. Here we use the formula for $\chi(t)$ in \cite[Corollary 3.2]{LLS2}, which can be evaluated for any semisimple element $t \in T$ with the aid of  {\sc Magma}. This allows us to compute ${\rm ifix}(T)$ precisely. 

As before, let $W$ be the Weyl group of $\G$ and let $W_{P} \leqs W$ be the Weyl group of the corresponding parabolic subgroup $P$ of $\G$. Let $w_0$ be the longest word in $W$ and set $w^* = w_0$ if $T = {}^2E_6(q)$, otherwise $w^*=1$. Let $\Pi(\bar{G}) = \{\a_1, \ldots, \a_r\}$ be the simple roots of $\G$ and let $\a_0 \in \Phi(\bar{G})$ be the highest root. We recall that the semisimple classes in $\bar{G}_{\s}$ are parametrised by pairs $(J,[w])$, where $J$ is a proper subset of $\Pi (\bar{G})\cup \{\a_0\}$ (determined up to $W$-conjugacy), $W_J$ is the subgroup of $W$ generated by the reflections in the roots in $J$, and $[w] = W_Jw$ is a conjugacy class representative of $N_W(W_J)/W_J$. Let us also recall that the $\bar{G}_{\s}$-classes of $\s$-stable maximal tori of $\bar{G}$ are parametrised by the $\s$-conjugacy classes in $W$ (two elements $x,y \in W$ are $\s$-conjugate if there exists $z \in W$ such that $y = z^{-1}x\s(z)$). 
Let $C_1, \ldots, C_k$ be the conjugacy classes of $W$, so $C_iw^*$ is the $i$-th $\s$-conjugacy class and let $T_i$ be a representative of the corresponding $\bar{G}_{\s}$-class of maximal tori of $\bar{G}_{\s}$. In addition, let $r_i$ be the relative rank of $T_i$ (that is, $r_i$ is the multiplicity of $q-\e$ as a divisor of $|T_i|$, where $\e=-1$ if $T={}^2E_6(q)$, otherwise $\e=1$). 

To illustrate our approach, we will focus on the following three cases (the other cases are very similar and we omit the details):
\begin{itemize}\addtolength{\itemsep}{0.2\baselineskip}
\item[{\rm (a)}] $T = F_4(q)$ and $H_0 = P_1$.
\item[{\rm (b)}] $T = {}^2E_6(q)$ and $H_0 = P_{1,6}$.
\item[{\rm (c)}] $T = E_7(q)$ and $H_0 = P_2$.
\end{itemize}

First consider (a). Let $t \in T$ be a $t_4$-involution, so $|(C_{\bar{G}}(t)^{\circ})_{\s}| = |{\rm SO}_{9}(q)|$ (see 
Table \ref{tabinv}). 
In terms of the above notation, we have $J = \{\a_1, \a_2, \a_3, \a_0\}$ and $w=1$. By applying 
\cite[Corollary 3.2]{LLS2} we get
\begin{equation}\label{e:c32}
\chi(t) = \sum_{i=1}^{k}\frac{|W|}{|C_i|}\cdot \frac{|W_P \cap C_iw^*|}{|W_P|}\cdot \frac{|W_J \cap C_i|}{|W_J|}\cdot \frac{(-1)^{r_i}|(C_{\bar{G}}(t)^{\circ})_{\s}|_{p'}}{|T_i|} 
\end{equation}
with $w^*=1$. With the aid of {\sc Magma}, we calculate   
\[
\chi(t) = (q^4+q^2+1)(q^4+1)(q^2+1)(q+1)
\]
and we conclude that ${\rm ifix}(T)>n^{11/15}$. Moreover, one checks that ${\rm ifix}(t_4)>{\rm ifix}(t_1)$, so $\b=11/15$ in \eqref{e:al2}. Similarly, in (b) we take a $t_1$-involution $t \in T$, so $|(C_{\bar{G}}(t)^{\circ})_{\s}| = (q+1)|{\rm SO}_{10}^{-}(q)|$, $J = \{\a_2,\a_3,\a_4,\a_5,\a_0\}$ and $w=1$. Then \eqref{e:c32} holds (with $w^* = w_0$) and we get
$$\chi(t) = (q^7+q^4+q^3+q+2)(q^5+1)(q^2+1),$$
which gives the desired result (with $\b=7/12$). 

Finally, let us consider case (c). First assume $q \equiv 1 \imod{4}$ and let $t \in T$ be a $t_7$ involution, so $|(C_{\bar{G}}(t)^{\circ})_{\s}| = (q-1)|E_6(q)|$, $J = \{a_1, \ldots, \a_6\}$, $w=1$ and \cite[Corollary 3.2]{LLS2} gives
\begin{align*}
\chi(t)  = &  \; \sum_{i=1}^{k}\frac{|W|}{|C_i|}\cdot \frac{|W_P \cap C_i|}{|W_P|}\cdot \frac{|W_J \cap C_i|}{|W_J|}\cdot \frac{(-1)^{r_i+1}|(C_{\bar{G}}(t)^{\circ})_{\s}|_{p'}}{|T_i|} \\
= & \; 2(q^6+q^3+1)(q^6+1)(q^4+q^2+2)(q^4+1)(q^3+1)(q^2+q+1)
\end{align*}
This implies that ${\rm ifix}(T)>n^{25/42}$. Now assume $q \equiv 3 \imod{4}$, so the involutions in $T$ are of type $t_1, t_4'$ and $t_7'$ (see Table \ref{tabinv}). One checks that ${\rm fix}(t)=0$ when $t=t_4'$ or $t_7'$, so let $t$ be a $t_1$-involution. Then $J = \{\a_2, \a_3, \a_4,\a_5,\a_6,\a_7,\a_0\}$, $w=1$ and using \cite[Corollary 3.2]{LLS2} we calculate
\begin{align*}
\chi(t)  = (q^6+2q^4+q^3+2q^2+3)(q^4-q^3+q^2-q+1)(q^4+1)(q^3+1)(q^2+1)(q+1)^3
\end{align*}
and this yields ${\rm ifix}(T)>n^{11/21}$. In particular, $\b=11/21$ in \eqref{e:al2}.
\end{proof}

\section{Maximal rank subgroups}\label{s:mr}

As in the previous section, let $T = (\bar{G}_{\s})'$ be a simple exceptional group of Lie type over $\mathbb{F}_{q}$. Here we are interested in the case where $H = N_G(\bar{H}_{\s})$ and $\bar{H}$ is a $\s$-stable non-parabolic maximal rank subgroup of $\bar{G}$ (in particular, $\bar{H}^{\circ}$ contains a $\s$-stable maximal torus of $\bar{G}$). The possibilities for $H$ are determined in \cite{LSS} and the analysis falls naturally into two cases, according to whether or not $H$ is the normalizer of a maximal torus. Our main result is the following (note that throughout this section, we continue to exclude the groups in \eqref{e:small}). 

\begin{thm}\label{t:mr}
Let $G$ be an almost simple primitive permutation group of degree $n$ with socle $T$, an exceptional group of Lie type over $\mathbb{F}_{q}$, and point stabilizer $H$, a maximal rank non-parabolic subgroup of $G$. Set $H_0 = H \cap T$. Then one of the following holds:
\begin{itemize}\addtolength{\itemsep}{0.2\baselineskip}
\item[{\rm (i)}] ${\rm ifix}(T)>n^{4/9}$;
\item[{\rm (ii)}] $(T,H_0)$ is one of the cases in Table $\ref{tab:mr}$ and ${\rm ifix}(T)>n^{\a}$ for the given constant $\a$.
\end{itemize}
Moreover,   
\[
\liminf_{q \to \infty} \frac{\log {\rm ifix}(T)}{\log n} = \b
\]
and either $\b \geqs 1/2$, or $T = E_6^{\e}(q)$, ${\rm soc}(H_0) = {\rm L}_{3}^{\e}(q^3)$ and $\b=13/27$.
\end{thm}

{\small
\renewcommand{\arraystretch}{1.2}
\begin{table}
\begin{center}
$$\begin{array}{llcl}\hline
T & H_0 & \a & \mbox{Conditions} \\ \hline   
F_4(q) & (q^2+q+1)^2.(3 \times {\rm SL}_{2}(3)) & 0.427 & \mbox{$q=2$, $G={\rm Aut}(T)$} \\
{}^3D_4(q) & (q^4-q^2+1).4 & 0.436 & q = 2 \\
& (q^2+q+1)^2.{\rm SL}_{2}(3) & 0.401 & q = 2  \\
& & 0.442  & q=3 \\
& (q^2-q+1)^2.{\rm SL}_{2}(3) & 0.351 & q = 2 \\ 
& & 0.419 & q=3 \\
& & 0.439 & q=4 \\
{}^2G_2(q) & (q-\sqrt{3q}+1){:}6 & 0.442 & q = 27 \\
{}^2B_2(q) &  (q+\sqrt{2q}+1){:}4 & 0.438 & q = 8 \\
& (q-\sqrt{2q}+1){:}4 & 0.380 & q = 8 \\
& &  0.436 & q = 32  \\ \hline 
\end{array}$$
\caption{The maximal rank actions with $n^{\a}<{\rm ifix}(T) \leqs n^{4/9}$}
\label{tab:mr}
\end{center}
\end{table}
\renewcommand{\arraystretch}{1}}

Our basic strategy is to identify an involution $t \in H_0$ and then derive an upper bound $|C_{H_0}(t)| \leqs f(q)$ for some function $f$. Then by \eqref{e:fix} we have 
\begin{equation}\label{e:mrbd}
\text{ifix}(T) \geqs \frac{|T|}{f(q)|t^T|} \; \text{ and } \; \b \geqs \liminf_{q \to \infty} \frac{\log |T| - \log (f(q) |t^T|)}{\log n}
\end{equation}
and we need to identify the $T$-class of $t$ (recall that the possibilities are listed in Table \ref{tabinv}). For the purposes of obtaining lower bounds on $\text{ifix}(T)$ and $\b$, we can assume that $t$ belongs to the largest class of involutions in $T$ and we find that this crude approach is effective in some cases. However, in general we need to do more to establish the desired bounds. Here it can be helpful to consider the restriction $V{\downarrow}\bar{H}^{\circ}$ of a 
$\bar{G}$-module $V$, such as the Lie algebra $\mathcal{L}(\bar{G})$ or the minimal module for $\bar{G}$, noting that in these cases the composition factors of $V{\downarrow}\bar{H}^{\circ}$ are readily available in the literature (for example, comprehensive tables are presented in \cite{Tho}). If $p \ne 2$ then we can use the decomposition of $V{\downarrow}\bar{H}^{\circ}$ to compute the dimension of the $1$-eigenspace of $t$ on $V$, which is equal to $\dim C_{\bar{G}}(t)$ when $V =  \mathcal{L}(\bar{G})$ (see \cite[Section 1.14]{Carter}). This usually allows us to identify the $T$-class of $t$. Similarly, when $p=2$ we can work out the Jordan form of $t$ on $V$ and then use \cite{Lawunip} to determine the $T$-class of $t$ (alternatively, we can appeal to Lawther's results in \cite{Law09} on the fusion of unipotent classes in connected reductive subgroups of $\bar{G}$).  

\subsection{Normalizers of maximal tori}\label{ss:tori}

To begin with, let us assume that $H = N_G(\bar{S}_{\s})$ where $\bar{S}$ is a $\s$-stable maximal torus of $\bar{G}$ (the remaining maximal rank subgroups will be handled in Section \ref{ss:ntori}). Set $H_1 = N_{\bar{G}_{\sigma}}(\bar{S}_{\sigma})$. As explained in \cite[Section 1]{LSS}, we have
\begin{equation}\label{e:h00}
H_1 = N_{\bar{G}_{\sigma}}(\bar{S}_\sigma) = \bar{S}_{\sigma}.N
\end{equation}
and $H_0 = H_1 \cap T$, where $N$ is isomorphic to a subgroup of the Weyl group $W = N_{\bar{G}}(\bar{S})/\bar{S}$. By the main theorem of \cite{LSS} we have $H_1 = N_{\bar{G}_{\sigma}}(\bar{S})$ and the possibilities for $H_1$ are presented in \cite[Table 5.2]{LSS}. It will be convenient to set $S = \bar{S}_{\sigma} \cap T$.

\begin{prop}\label{p:mr1}
Let $G,T,H,H_0$ and $n$ be as in the statement of Theorem \ref{t:mr} and assume $H$ is the normalizer of a maximal torus of $G$. Then 
\[
\liminf_{q \to \infty} \frac{\log {\rm ifix}(T)}{\log n} \geqs \frac{1}{2}
\]
and either 
\begin{itemize}\addtolength{\itemsep}{0.2\baselineskip}
\item[{\rm (i)}] ${\rm ifix}(T)>n^{4/9}$; or
\item[{\rm (ii)}] $(T,H_0)$ is one of the cases in Table $\ref{tab:mr}$ and ${\rm ifix}(T)>n^{\a}$ for the given constant $\a$.
\end{itemize}
\end{prop}

We will prove Proposition \ref{p:mr1} in Lemmas \ref{l:mr1l1}--\ref{l:mr1l3}. 
We start by establishing an elementary result on the longest word in $W$.

\begin{lem} \label{l:wzero}
Let $\bar{G}$ be a simple adjoint algebraic group of type $B_2, G_2, D_4, F_4, E_7$ or $E_8$ over an algebraically closed field $K$ of characteristic $p \geqs 0$. Let $\bar{S}$ be a maximal torus of $\bar{G}$ and let $w_0$ be the longest word in the Weyl group $W = N_{\bar{G}}(\bar{S}) / \bar{S}$.
\begin{itemize}\addtolength{\itemsep}{0.2\baselineskip}
\item[{\rm (i)}] $w_0$ is an involution acting by inversion on $\bar{S}$. 
\item[{\rm (ii)}] $w_0$ is the image of an involution $\hat{w}_0 \in N_{\bar{G}}(\bar{S})$ such that $\bar{S}\hat{w}_0 = \hat{w}_0^{\bar{S}}$.
\item[{\rm (iii)}] $\dim C_{\bar{G}}(\hat{w}_0) = \frac{1}{2}(\dim \bar{G} - r)$, where $r$ is the rank of $\bar{G}$.
\end{itemize}
\end{lem}

\begin{proof}
Part (i) is well-known (for example, see \cite[p.192]{KL}). By \cite[Proposition 8.22]{RicSpr}, there is an involution $\hat{w}_0 \in N_{\bar{G}}(\bar{S})$ such that $w_0 = \bar{S}\hat{w}_0$, so $\la \bar{S}, \hat{w}_0 \ra = \bar{S}{:}\la \hat{w}_0 \ra$ is a split extension and every element in the coset $\bar{S}\hat{w}_0$ is an involution. Since $\hat{w}_0$ inverts $\bar{S}$, it follows that $s\hat{w}_0$ and $\hat{w}_0$ are $\bar{S}$-conjugate if and only if $s = t^2$ for some $t \in \bar{S}$. But every element in $\bar{S}$ is a square since $K$ is algebraically closed, so (ii) follows. Finally, part (iii) follows from \cite[Lemma 3.6]{BGS}.
\end{proof}

\begin{cor}\label{c:new}
Let $\bar{G}$ be as in the statement of Lemma \ref{l:wzero} and suppose $T = (\bar{G}_{\s})'$ and 
\[
H_1 = N_{\bar{G}_{\sigma}}(\bar{S}_\sigma) = \bar{S}_{\sigma}.N
\]
for some $\sigma$-stable maximal torus $\bar{S}$ of $\bar{G}$. Recall that $S = \bar{S}_{\sigma} \cap T$. Then there exists an involution $t \in H_0$ such that 
\begin{itemize}\addtolength{\itemsep}{0.2\baselineskip}
\item[{\rm (i)}] $t$ inverts $S$, 
\item[{\rm (ii)}] every element in the coset $St$ is $T$-conjugate to $t$, and
\item[{\rm (iii)}] $t^{T}$ is the largest class of involutions in $T$. 
\end{itemize}
\end{cor}

\begin{proof}
Let $w_0$ be the longest word in the Weyl group $W = N_{\bar{G}}(\bar{S})/\bar{S}$ and let 
$\hat{w}_0 \in N_{\bar{G}}(\bar{S})$ be the involution in Lemma \ref{l:wzero}(ii). Since $w_0$ is the only non-trivial element of $Z(W)$, it follows that $\sigma$ fixes $w_0$, so $\bar{S}{:}\langle \hat{w}_0 \rangle$ is $\sigma$-stable and $(\bar{S}{:}\langle \hat{w}_0 \rangle)_\sigma \leqs H_1$. Moreover, the conjugacy class $\hat{w}_0^{\bar{S}}$ is also stabilized by $\sigma$, so Lang's theorem implies that $H_1$ contains an involution $t \in \hat{w}_0^{\bar{S}}$. 

We claim that $t \in H_0$. Apart from the case where $\bar{G} = E_7$ and $p \ne 2$, we have $\bar{G}_{\s}=T$, so $H_0 = H_1$ and the claim is clear. Now assume $\bar{G} = E_7$ and $p \ne 2$. Seeking a contradiction, let us assume that $t \not\in H_0$, in which case 
$(\bar{S}_{\sigma}{:}\la t \ra) \cap T = \bar{S}_{\sigma}$. Let $\hat{G}$ be the simply connected algebraic group and by an abuse of notation we will write $\sigma$ for the corresponding Steinberg endomorphism of $\hat{G}$. Note that $\bar{G} = \hat{G}/Z$ and $\bar{S} = \hat{S}/Z$, where $Z = Z(\hat{G}) = Z(\hat{G}_{\s})$ (see \cite[Table 2.2]{GLS}). Now $Z \leqs \hat{S}_\sigma$ and so $\hat{S}_\sigma / Z = \bar{S}_\sigma \cap T$. But $\hat{S}$ and $\bar{S}$ are connected isogenous groups, so $|\hat{S}_\sigma| = |\bar{S}_\sigma|$ and thus $|\bar{S}_\sigma \cap T| = |\bar{S}_\sigma|/2$, which implies that $\bar{S}_\sigma \not\leqs T$. We conclude that $t \in H_0$ as claimed.

Now $t$ inverts $S$ (since it inverts the algebraic torus $\bar{S}$), so (i) follows. Part (iii) is an immediate corollary of Lemma \ref{l:wzero}(iii). Finally, observe that $St \subset \bar{S} \hat{w}_0$ so all elements in $St$ are $\bar{G}$-conjugate by Lemma \ref{l:wzero}(ii). Part (ii) now follows since $x^{\bar{G}} \cap T = x^T$ for all involutions $x \in T$ (see Remark \ref{r:invcon}).
\end{proof}

We are now ready to begin the proof of Proposition \ref{p:mr1}. Set
\[
\b = \liminf_{q \to \infty} \frac{\log {\rm ifix}(T)}{\log n}.
\]
Recall that if $X$ is a finite group, then $i_2(X)$ denotes the number of involutions in $X$.

\begin{lem} \label{l:mr1l1}
Proposition \ref{p:mr1} holds if $T \in \{{}^2B_2(q),{}^2G_2(q), G_2(q)\}$. 
\end{lem}

\begin{proof}
Write $H_1 = N_{\bar{G}_{\s}}(\bar{S}_{\s}) = S.N$ as in \eqref{e:h00} and note that $H_0 = H_1$. Since $p=3$ when $T = G_2(q)$ (see \cite[Table 5.2]{LSS}), in each case $T$ has a unique class of involutions (see Table \ref{tabinv}). Therefore  
\begin{equation}\label{e:i2}
{\rm ifix}(T) = \frac{i_2(H_0)}{i_2(T)}\cdot n
\end{equation}
and it suffices to compute $i_2(H_0)$. In the next paragraph, we write $D_m$ for the dihedral group of order $m$.

First assume $T = {}^2G_2(q)$ and $H_0 = S.N$, where $|S|=q+1$ and $N = Z_6$. By \cite{K88}, we have $H_0 = (2^2 \times D_{(q+1)/2}){:}3$, so $i_2(H_0) = q+4$ and the result follows. Next assume $T = G_2(q)$ and $H_0 = (q-\e)^2.D_{12}$, where $p=3$ and $q \geqs 9$ by maximality. Here Corollary \ref{c:new} implies that $i_2(H_0) \geqs (q-\e)^2$ and we deduce that $\b \geqs 1/2$ and $\text{ifix}(T) > n^{4/9}$ for $q > 9$. If $q=9$, then using {\sc Magma} we calculate that $i_2(H_0) = 115$ when $\e=1$, and $i_2(H_0) = 163$ when $\e=-1$, whence $\text{ifix}(T) > n^{4/9}$ as required. 

In each of the remaining cases (see \cite[Table 5.2]{LSS}), we have $H_0 = S.N$ where $|S|$ is odd and $N$ has a unique involution. By applying Corollary \ref{c:new}, we deduce that $i_2(H_0)=|S|$ and the result follows.
\end{proof}

\begin{lem} \label{l:mr1l2}
Proposition \ref{p:mr1} holds if $T \in \{{}^3D_4(q), {}^2F_4(q), F_4(q), E_7(q), E_8(q)\}$.  
\end{lem}

\begin{proof}
This is similar to the previous lemma. Fix an involution $t \in H_0$ such that $|t^T \cap H| \geqs |S|$ and $t$ belongs to the largest class of involutions in $T$ (see Corollary \ref{c:new}). This gives lower bounds on $\text{ifix}(T)$ and $\b$ as in \eqref{e:mrbd}, and we deduce that $\b \geqs 1/2$ in all cases. Moreover, the same bounds imply that $\text{ifix}(T) > n^{4/9}$, with the exception of the cases in Table \ref{tab:exmr}.

{\small
\renewcommand{\arraystretch}{1.2}
\begin{table}
\begin{center}
$$\begin{array}{lll} \hline
T & H_0 & q \\ \hline
E_8(q) & (q+1)^8.W & 2 \\
E_7(q) & (q+1)^7.W & 2 \\
F_4(q) & (q^2+q+1)^2.(3 \times {\rm SL}_{2}(3)) & 2 \\
{}^3D_4(q) & (q^4-q^2+1).4 & 2 \\
& (q^2+q+1)^2.{\rm SL}_{2}(3) & 2,3 \\
& (q^2-q+1)^2.{\rm SL}_{2}(3) & 2,3,4 \\ \hline 
\end{array}$$
\caption{Some normalizers of maximal tori}
\label{tab:exmr}
\end{center}
\end{table}
\renewcommand{\arraystretch}{1}}

Suppose $T = {}^3 D_4(q)$ or $F_4(q)$, as in Table \ref{tab:exmr}, so $H_1 = H_0 = S.N$. Here $|S|$ is odd and $N$ has a unique involution, so $i_2(H_0) = |t^T \cap H| = |S|$ and we can compute $\text{ifix}(T)$ precisely. One can check that $\text{ifix}(T) \leqs n^{4/9}$ in each case, so these exceptions are recorded in Table \ref{tab:mr}. Finally, we claim that $\text{ifix}(T) > n^{4/9}$ in the two remaining cases in Table \ref{tab:exmr}. To see this, we use {\sc Magma} to construct $H_0 = S.W$ as a subgroup of $E_7(4)$ or $E_8(4)$, respectively, and we find a class of involutions in $H_0$ that is contained in the $A_1$-class of $T$. For instance, if $T = E_8(2)$ and $H_0 = 3^8.W$, then there is an involution $y \in H_0$ in the $A_1$-class of $T$ with $|y^{H_0}| = 360$, and we deduce that ${\rm fix}(y)>n^{4/9}$ as required.  
\end{proof}

Finally, we complete the proof of Proposition \ref{p:mr1} by dealing with the case $T = E_6^\e(q)$. Here $\bar{G}=E_6$ and every involution in $\G_{\s}$ is contained in $T$, so for the purposes of computing ${\rm ifix}(T)$ we may assume that $\bar{G}_{\s} \leqs G$ and $H_1 \leqs H$, where $H_1 = N_{\bar{G}_{\s}}(\bar{S}_{\s})$ as above. 

For the proof of Lemma \ref{l:mr1l3} below, we need a more detailed understanding of the construction of $H_1$ (see \cite[Sections 1--3]{LSS} and \cite[Section 25.1]{MT} for further details). Fix a $\s$-stable maximal torus $\bar{T}$ of $\bar{G}$ and set $W = N_{\G}(\bar{T})/\bar{T}$.  Let 
\begin{equation}\label{e:rootg}
U_{\a} = \{x_{\a}(c) \,:\, c \in K\}
\end{equation} 
be the root subgroup of $\bar{G}$ corresponding to $\a \in \Phi(\bar{G})$.
For $\e=+$ we may assume that $\sigma$ is the standard Frobenius morphism $\s_q$ (acting as  $x_{\a}(c) \mapsto x_{\a}(c^q)$ on root group elements), and for $\e=-$ we take $\s=\s_q\tau$, where $\tau$ is the standard graph automorphism of $\bar{G}$ induced from the order two symmetry $\rho$ of the Dynkin diagram (so $\s:x_{\a}(c) \mapsto x_{\rho(\a)}(c^q)$ for $\a \in \Pi(\bar{G})$).  
Without loss of generality, we may assume that 
$$\bar{T} = \la h_{\a_i}(c_i) \mid c_i \in K^*,\; i = 1, \ldots, 6\ra, $$
where
$$h_{\alpha}(c) = x_{\a}(c) x_{-\a}(-c^{-1}) x_{\a}(c-1) x_{-\a}(1) x_{\a}(-1)$$
as in \cite[Lemma 6.4.4]{Carter72}. As before, let $w_0$ be the longest word of $W$, so that $w_0 \tau$ acts as $-1$ on $\Phi(\bar{G})$ (see \cite[p.192]{KL}). Set $w^*=w_0$ if $\e=-$, otherwise $w^*=1$.

Recall that there is a $1$-$1$ correspondence between the $\bar{G}_{\s}$-classes of $\s$-stable maximal tori in $\bar{G}$ and the $\sigma$-conjugacy classes of $W$. More precisely, if we write $\bar{S} = \bar{T}^g$ for some $g \in \bar{G}$, then $w = \s(g)g^{-1} \in N_{\bar{G}}(\bar{T})$ and its image in $W$ (which we also write as $w$) represents the $\s$-class corresponding to $\bar{S}^{\bar{G}_{\s}}$. 
As in \cite[Section 1]{LSS}, it is convenient to replace $\s$ by $\s w$ in order to simplify the description of $H_1$, where $(\s w)^g = \s$ and $(\bar{G}_{\s w})^g = \bar{G}_{\s}$. Then 
$T = (\bar{G}_{\s w})'$ and 
\begin{equation}\label{e:h11}
H_1 = N_{\bar{G}}(\bar{T})_{\sigma w} = \bar{T}_{\sigma w}.N
\end{equation}
with $N = \{x \in W \,:\, xw = w\s(x) \}$. It will be convenient to set $S = \bar{T}_{\sigma w}$.

\begin{lem} \label{l:mr1l3}
Proposition \ref{p:mr1} holds if $T = E_6^\e(q)$.
\end{lem}

\begin{proof}
According to \cite[Table 5.2]{LSS}, we have $H_1 = S.N$ as in \eqref{e:h11}, where 
\begin{itemize}\addtolength{\itemsep}{0.2\baselineskip}
\item[{\rm (a)}] $S = (q-\e)^6$ and $N=W$, with $q \geqs 5$ if $\e=1$; or 
\item[{\rm (b)}] $S = (q^2 +\epsilon q +1)^3$ and $N=3^{1+2}.\text{SL}_2(3)$, with $(\e,q) \ne (-1,2)$.
\end{itemize}
We will use the notation for involutions presented in Table \ref{tabinv}. 

First consider (a). In the notation of \eqref{e:h11}, note that $w=w^*$. First assume $p$ is odd. We claim that $S$ contains a $t_1$-involution $t$, in which case the trivial bound $|C_{H_0}(t)| \leqs |H_0|$ yields $\text{ifix}(T) > n^{4/9}$ and $\b \geqs 5/9$. To prove the claim, let $\bar{M} = D_5T_1$ be a $\s w$-stable maximal subgroup of $\bar{G}$ containing $\bar{T}$, so $Z(D_5T_1) < \bar{T} < \bar{M}$ is a chain of $\s w$-stable subgroups. Now $Z(D_5T_1)$ contains a unique involution (a $t_1$-involution), which must be fixed by $\sigma w$, hence 
$S=\bar{T}_{\sigma w}$ contains a $t_1$-involution, as claimed.  

Now assume $p=2$, so $|S|$ is odd. We claim that $H_1$ contains an $A_1$-involution, which yields $\text{ifix}(T) > n^{4/9}$ and $\b \geqs 1/2$. To justify the claim, let 
\[
t = x_{\a_2}(1) x_{-\a_2}(1) x_{\a_2}(1) \in N_{\bar{G}}(\bar{T}),
\]
where $x_{\a}(1) \in U_{\a}$, as in \eqref{e:rootg}. Then $t$ is an involution that is fixed by 
$\s w$, so $t \in H_1$. Moreover, $t$ is contained in the $A_1$-type subgroup
$\la U_{\a_2}, U_{-\a_2} \ra$, which has a unique conjugacy class of involutions. Thus $t$ is $\bar{G}$-conjugate to $x_{\alpha_2}(1)$ and the claim follows. 

Now consider (b).
If $H_1$ contains a $t_2$-involution when $p \neq 2$, or an involution in one of the classes labelled $A_1$ or $A_1^2$ when $p=2$, then we immediately deduce that $\b \geqs 1/2$ and ${\rm ifix}(T)>n^{4/9}$. Therefore, we may as well assume that every involution in $H_1$ is contained in the largest $T$-class of involutions (which can be read off from Table \ref{tabinv}) and we need to determine an appropriate lower bound on $i_2(H_1)$.  

In terms of the above correspondence between the $\bar{G}_{\s}$-classes of $\s$-stable maximal tori in $\bar{G}$ and the $\sigma$-classes of $W$, the given maximal torus $S$ corresponds to the class of $w$ with Carter diagram $3A_2$ (see \cite[Table C]{LSS}). By considering the classes in $W$, we can write $w = xw^*$, where
\[
x = (1, 3, 43)(2, 72, 35)(4, 22, 62)(5, 6, 47) \ldots \in W  
\] 
as a permutation of $\Phi(\bar{G})$, where our ordering of the $72$ roots in $\Phi(\bar{G})$ is consistent with {\sc Magma} (note that we only give part of the permutation $x$, but this is enough to uniquely determine it). It follows that  
$$S = \langle h_{\alpha_1}(c_1^{-\e q})h_{\alpha_3}(c_1), h_{\alpha_2}(c_2^{-\e q}) h_{-\alpha_0}(c_2), h_{\alpha_5}(c_3^{-\e q})h_{\alpha_6}(c_3) \mid c_i^{q^2 + \e q + 1}=1 \rangle,$$
where $\alpha_0$ is the highest root in $\Phi(\bar{G})$. Now $N = 3^{1+2}.\text{SL}_2(3)$ has $9$ involutions, which form a single conjugacy class, say $\{m_1, \dots, m_9\}$. Since $|S|$ is odd, each subgroup $N_i = S.\langle m_i \rangle \leqs H_1$ is a split extension and we note that $i_2(H_1) \geqs \sum_i i_2(N_i)$ since 
$S\langle m_i \rangle \cap S\langle m_j \rangle = S$ for all $i \neq j$. We claim that $i_2(N_i) \geqs (q^2 + \e q + 1)^2$ for each $i$, so $i_2(H_1) \geqs 9 (q^2 + \e q + 1)^2$. This implies that $\b \geqs 1/2$ and ${\rm ifix}(T)>n^{4/9}$, so it remains to justify the lower bound on $i_2(N_i)$.

The number of involutions in $N_i$ is at least the number of elements of $S$ inverted by $m_i$. Thus for each $i$, it suffices to find a subtorus of $S$ of size $(q^2 + \e q + 1)^2$ inverted by $m_i$. With the aid of {\sc Magma}, it is easy to determine the action of $m_i$ on $S$. 
For example, one of the $m_i$, say $m_1$, acts as follows
\begin{align*}
h_{\alpha_1}(c) & \longleftrightarrow h_{\alpha_2}(c) \\
h_{\alpha_3}(c) & \longleftrightarrow h_{-\alpha_0}(c) \\
h_{\alpha_5}(c) & \longleftrightarrow h_{\alpha_5}(c^{-1}) = h_{\alpha_5}(c)^{-1} \\
h_{\alpha_6}(c) & \longleftrightarrow h_{\alpha_6}(c^{-1}) = h_{\alpha_6}(c)^{-1}
\end{align*}
which implies that the subtorus 
$$\langle h_{\alpha_1}(c_1^{-\e q})h_{\alpha_3}(c_1)h_{\alpha_2}(c_1^{\e q}) h_{-\alpha_0}(c_1^{-1}), h_{\alpha_5}(c_2^{-\e q})h_{\alpha_6}(c_2) \mid c_i^{q^2 + \e q + 1}=1 \rangle < S$$
is inverted by $m_1$. Similarly, each of the remaining $m_i$ inverts a torus of size $(q^2 + \e q + 1)^2$, which proves the claim. 
\end{proof}

\subsection{The remaining maximal rank subgroups}\label{ss:ntori}

To complete the proof of Theorem \ref{t:mr}, we may assume that $H = N_G(\bar{H}_{\s})$, where $\bar{H}$ is a $\s$-stable maximal rank subgroup and the connected component $\bar{H}^\circ$ is not a maximal torus of $\bar{G}$. The possibilities for $H$ are given in \cite[Table 5.1]{LSS}.

\begin{prop}\label{p:mr2}
Let $G, T, H, H_0$ and $n$ be as in the statement of Theorem \ref{t:mr} and assume $H$ is not the normalizer of a maximal torus of $G$. Then ${\rm ifix}(T)>n^{4/9}$. Moreover,  
\[
\liminf_{q \to \infty} \frac{\log {\rm ifix}(T)}{\log n} = \b
\]
and either $\b \geqs 1/2$, or $T = E_6^{\e}(q)$, ${\rm soc}(H_0) = {\rm L}_{3}^{\e}(q^3)$ and $\b=13/27$. 
\end{prop}

As before, we will identify an involution $t \in H_0$, compute an upper bound $|C_{H_0}(t)| \leqs f(q)$ and determine the $T$-class of $t$, which allows us to evaluate the bounds in \eqref{e:mrbd}. We present our results in Tables \ref{tab:mr20}--\ref{tab:mr22}, which essentially reduces the proof of Proposition \ref{p:mr2} to verifying the information in these tables. 

{\small 
\renewcommand{\arraystretch}{1.2}
\begin{table}
\begin{center}
$$\begin{array}{lllllc}\hline
T & \mbox{Type of $H_0$} & t & f(q)  & t^T & \gamma \\ \hline
{}^2G_2(q)' & {\rm L}_{2}(q) & t'_1 & 2(q+1) & t_1 & 1/2 \\
& & & & & \\
{}^3D_4(q) & {\rm L}_{2}(q) \times {\rm L}_{2}(q^3) & (t_1,t_1) & 2(q-1)(q^3-1) & t_2 & 1/2 \\
& & (A_1,1) & q|\text{SL}_2(q^3)| & A_1 & 1/2 \\
& {\rm L}_{3}^{\e}(q) & \gamma_1 & 2|{\rm SO}_3(q)| & t_2 & 1/2 \\
& & A_1 & 2q^3(q^3-\e) & A_1 & 2/3 \\
& & & & & \\
G_2(q) & {\rm L}_{2}(q)^2 & (t_1,t_1) & 2(q-1)^2 & t_1 & 1/2 \\
 & & (A_1,1) & q|{\rm SL}_2(q)| & A_1 & 1/2 \\
& {\rm L}_{3}^{\e}(q) & \mbox{$t_1$ and $\gamma_1$} & \mbox{$2|{\rm GL}^\e_2(q)|$ and $2|{\rm SO}_{3}(q)|$} & t_1 & 1/2 \\
& & A_1 & 2q^3(q-\e) & A_1 & 2/3 \\
& & & & & \\
{}^2F_4(q) & {\rm U}_{3}(q) & \gamma_1 & 2|{\rm SL}_2(q)| & A_1 \tilde{A}_1 & 1/2 \\
& {}^2B_2(q)^2 & ((\tilde{A}_1)_2,1) & q^2|{}^2B_2(q)| & (\tilde{A}_1)_2 & 1/2 \\
& {\rm Sp}_{4}(q) & c_2 & 2q^4 & A_1 \tilde{A}_1 & 1/2 \\
& & & & & \\
F_4(q) & {\rm L}_2(q) \times {\rm PSp}_6(q) & (1,t_1'')^\dagger & 2|\text{SL}_2(q)|^2|\text{Sp}_4(q)| & t_4 & 5/7 \\

& \Omega_9(q) & t_4 & |\text{SO}^+_8(q)| &  t_4 & 1/2 \\

& & A_1 & q^{11} |\text{SL}_2(q)||\text{Sp}_4(q)| &  A_1 & 3/4 \\

& \text{P}\Omega^+_8(q) & t_2 & 6|\text{SL}_2(q)|^4 &  t_1 & 1/2 \\

& & A_1 & 6q^9|\text{SL}_2(q)|^3 &  A_1 & 3/4 \\

& {}^3D_4(q) & t_2 & 3|\text{SL}_2(q)||\text{SL}_2(q^3)| & t_1 & 1/2 \\

& & A_1 & 3q^9|\text{SL}_2(q^3)| & A_1 & 3/4 \\

& \text{L}_3^\epsilon(q)^2  & (1,t_1) & 2|\text{SL}_2(q)||\text{GL}^\epsilon_3(q)| & t_4 & 2/3 \\

& & (A_1,1) & 2q^3|\text{GL}^\epsilon_3(q)| & A_1 & 2/3 \\

& \text{Sp}_4(q)^2 & (A_1,1) & 2q^3|\text{SL}_2(q)||\text{Sp}_4(q)| & A_1 & 5/8 \\

& \text{Sp}_4(q^2) & \varphi & 2|\text{Sp}_4(q)| & (\tilde{A}_1)_2 & 5/8 \\

& & & & & \\

E_6^{\e}(q) & \text{L}_2(q) \times \text{L}^\epsilon_6(q) & (1,t_2'')^\dagger & 2|\text{SL}_2(q)|^2|\text{GL}^\epsilon_4(q)| & t_1 & 3/5 \\

& & (1,A_1) & q^9|\text{SL}_2(q)||\text{GL}^\epsilon_4(q)|& A_1 & 7/10 \\

& \text{L}^\epsilon_3(q)^3 & (t_1,t_1,1) & 6(q-\epsilon)|\text{SL}_2(q)|^2|\text{GL}_3^\epsilon(q)| & t_1 & 5/9 \\

& & (A_1,1,1) & 6q^3(q-\epsilon)|\text{SL}_3^\epsilon(q)|^2 & A_1 & 2/3 \\

& \text{L}_3(q^2) \times \text{L}_3^{-\epsilon}(q)& (t_1,1) & 2|\text{GL}_2^{\e}(q^2)||\text{SL}_3^{-\epsilon}(q)| & t_1 & 5/9 \\

& & (1,A_1) & 2q^3(q+\epsilon)|\text{SL}_3(q^2)| & A_1 & 2/3 \\

& \text{L}_3^\epsilon(q^3) & t_1 & 3|\text{GL}_2^{\e}(q^3)| & t_2 & 13/27 \\

& & A_1 & 3q^9(q^3-\epsilon) & A_1^3 & 13/27 \\

& \text{P}\Omega^+_8(q) & t_2 & 6(q-\epsilon)^2|\text{SL}_2(q)|^4 & t_2 & 1/2 \\ 

& & A_1 & 6q^9(q-\epsilon)^2|\text{SL}_2(q)|^3  & A_1 & 3/4 \\ 

& {}^3D_4(q) & t_2 & 3(q^2 + \epsilon q +1)|\text{SL}_2(q)||\text{SL}_2(q^3)| & t_2 & 1/2 \\

& & A_1 & 3q^9(q^2 + \epsilon q +1)|\text{SL}_2(q^3)| & A_1 & 3/4 \\

& \text{P}\Omega^\epsilon_{10}(q) & t_2 & |\text{SL}_2(q)|^2|\text{GL}_4^\epsilon(q)| & t_2 & 1/2 \\

& & A_1 & q^{13}|\text{SL}_2(q)||\text{SL}^\epsilon_4(q)| & A_1 & 3/4 \\ \hline
\multicolumn{5}{l}{\mbox{{\small $^{\dagger}$ The notation $t_1''$, $t_2''$ is explained in Remark \ref{r:dagger}.}}} 
\end{array}$$
\caption{Maximal rank subgroups, $T \ne E_7(q), E_8(q)$}
\label{tab:mr20}
\end{center}
\end{table}
\renewcommand{\arraystretch}{1}
}

{\small
\renewcommand{\arraystretch}{1.2}
\begin{table}
\begin{center}
$$\begin{array}{llllc}\hline
\mbox{Type of $H_0$} & t & f(q) & t^T & \gamma \\ \hline

\text{L}_2(q) \times \text{P}\Omega^+_{12}(q) & (1,t_2) & 2|\text{SL}_2(q)|^3 |\text{SO}^+_8(q)| & t_1 & 1/2 \\

& (1,A_1) & q^{17}|\text{SL}_2(q)|^2 |\Omega^+_8(q)| & A_1 & 3/4 \\

\text{L}^\epsilon_8(q) & t_4 & 4(q-\epsilon)|\text{SL}^\epsilon_4(q)|^2 & t_1 & 19/35 \\

& A_1 & 2q^{13} |\text{GL}_6^{\e}(q)|  & A_1 & 5/7 \\

\text{L}_3^\epsilon(q) \times \text{L}_6^{\epsilon}(q) & (1,t_3) \text{ or } (1,t'_3) & 2(q+\epsilon)|\text{SL}^\epsilon_3(q)| |\text{SL}_3(q^2)| & t_7 \text{ or } t'_7 & 3/5 \\

& (1,A_1) & 2q^9 |\text{SL}^\epsilon_3(q)||\text{GL}^\epsilon_4(q)|^2  & A_1 & 11/15 \\

\text{L}_2(q)^3 \times \text{P}\Omega^+_8(q) & (1,1,1,t_2) & 24|\text{SL}_2(q)|^7  & t_1 & 1/2 \\ 

& (1,1,1,A_1) & 6q^9|\text{SL}_2(q)|^6 & A_1 & 3/4 \\ 

\text{L}_2(q^3) \times {}^3D_4(q) & (1,t_2) & 3|\text{SL}_2(q)|^4|\text{SL}_2(q^3)| & t_1  & 1/2 \\ 

& (1,A_1) & 3q^9|\text{SL}_2(q)|^3|\text{SL}_2(q^3)| & A_1  & 3/4 \\ 

\text{L}_2(q)^7 & (t_1,t_1,t_1,-) \text{ or } (t'_1,t'_1,t'_1,-) & 

1344(q+1)^3|\text{SL}_2(q)|^4 & t_7 \text{ or } t'_7 & 4/7 \\ 

& (A_1,- ) & 168q|\text{SL}_2(q)|^6 & A_1 & 5/7 \\

\text{L}_2(q^7) & t_1 \text{ or } t'_1 & 

7(q^7+1) & t_4 \text{ or } t'_4 & 1/2 \\ 

& A_1 & 7q^7 & A_1^4 & 1/2 \\

E_6^\epsilon(q) & t_2 & 2|\text{SL}_2(q)||\text{GL}^\epsilon_6(q)|  & t_1 & 5/9 \\ 

& A_1 & 2q^{21} |\text{GL}^\epsilon_6(q)| & A_1 & 7/9 \\ \hline
\end{array}$$
\caption{Maximal rank subgroups, $T =E_7(q)$}
\label{tab:mr21}
\end{center}
\end{table}
\renewcommand{\arraystretch}{1}
}

{\small
\renewcommand{\arraystretch}{1.2}
\begin{table}
\begin{center}
$$\begin{array}{llllc}\hline
\mbox{Type of $H_0$} & t & f(q)  & t^T & \gamma \\ \hline

\text{P}\Omega_{16}^+(q) & t_4 & 2|\text{SO}^+_8(q)|^2  & t_1 & 1/2 \\

& A_1 & q^{23} |\text{SL}_2(q)||\Omega^+_{12}(q)| & A_1 & 49/64 \\

\text{L}_2(q) \times E_7(q) & (1,t''_1)^{\dagger} & 2|\text{SL}_2(q)|^2|\text{SO}^+_{12}(q)| & t_8 & 4/7 \\ 

& (1,A_1) & q^{33}|\text{SL}_2(q)||\Omega^+_{12}(q)| & A_1 & 13/16 \\ 

\text{L}^\epsilon_9(q) & t_3 & 2|\text{GL}^\epsilon_3(q)||\text{SL}^\epsilon_6(q)|  & t_8 & 23/42 \\ 

& A_1 & 2q^{15}|\text{GL}^\epsilon_7(q)|  & A_1 & 3/4 \\ 

\text{L}_3^\epsilon(q) \times E_6^\epsilon (q) & (1,t_2) & 2|\text{SL}_2(q)||\text{SL}^\e_3(q)||\text{SL}^\e_6(q)| & t_8 & 5/9 \\ 

& (1,A_1) & 2q^{21}|\text{SL}^\e_3(q)||\text{SL}^\e_6(q)| & A_1 & 7/9 \\ 

\text{L}_5^\epsilon(q)^2 & (t_1,t_2) & 4|\text{SL}_2(q)||\text{GL}_3^\epsilon(q)||\text{GL}^\epsilon_4(q)| & t_8 & 27/50 \\ 

& (A_1,1) & 4q^7|\text{GL}_3^\epsilon(q)||\text{SL}_5^\epsilon(q)| & A_1 & 3/4 \\ 

\text{U}_5(q^2) & \gamma_1 & 4|\text{SO}_5(q^2)| & t_1 & 1/2 \\ 

& A_1 & 4q^{14}(q^2+1)|\text{SU}_3(q^2)| & A_1^2 & 31/50 \\

\text{P}\Omega^+_8(q)^2 & (t_2,t_2) & 48|\text{SL}_2(q)|^8  & t_1 & 1/2 \\ 

& (A_1,1) & 12q^{6}|\text{SL}_4(q)||\Omega^+_8(q)| & A_1 & 47/64 \\ 

\text{P}\Omega^+_8(q^2) & t_2 & 12|\text{SL}_2(q^2)|^4 & t_1 & 1/2 \\ 

& A_1 & 12q^{12} |\text{SL}_4(q^2)| & A_1^2 & 19/32 \\

^3D_4(q)^2 & (t_2,t_2) & 6|\text{SL}_2(q)|^2 |\text{SL}_2(q^3)|^2 & t_1 & 1/2 \\ 

& (A_1,1) & 6q^9|\text{SL}_2(q^3)||{}^3D_4(q)| & A_1 & 3/4 \\

^3D_4(q^2) & t_2 & 6|\text{SL}_2(q^2)| |\text{SL}_2(q^6)| & t_1 & 1/2 \\ 

& A_1 & 6q^{18}|\text{SL}_2(q^6)| & A_1^2 & 5/8 \\

\text{L}_3^\epsilon(q)^4 & (t_1,t_1,t_1,1) & 48|\text{GL}_2(q)|^3 |\text{SL}^\epsilon_3(q)|  & t_8 & 29/54 \\ 

& (A_1,1,1,1) & 48q^3(q-\epsilon)|\text{SL}^\epsilon_3(q)|^3  & A_1 & 3/4 \\ 

\text{U}_3(q^2)^2 & (\gamma_1,\gamma_1) & 8|\text{SO}_3(q^2)|^2 & t_1 & 1/2 \\ 

& (A_1,1) & 8q^6|\text{GU}_3(q^2)| & A_1^2 & 11/18 \\

\text{U}_3(q^4) & \gamma_1 & 8|\text{SO}_3(q^4)| & t_1 & 1/2 \\ 

& \gamma_1 & 8|\text{SL}_2(q^4)| & A_1^4 & 1/2 \\ 

\text{L}_2(q)^8 & (t_1,t_1,t_1,t_1,-) & 21504(q-1)^4|\text{SL}_2(q)|^4  & t_8 & 15/28 \\ 

& (A_1,-) & 1344q|\text{SL}_2(q)|^7 & A_1 & 3/4 \\ \hline
\multicolumn{5}{l}{\mbox{{\small $^{\dagger}$ The notation $t_1''$ is explained in Remark \ref{r:dagger}.}}} \\
\end{array}$$
\caption{Maximal rank subgroups, $T =E_8(q)$}
\label{tab:mr22}
\end{center}
\end{table}
\renewcommand{\arraystretch}{1}
}

We start by describing the structure of $\bar{H}_\sigma$, which is recorded in the second column of \cite[Table 5.1]{LSS} (see \cite[Sections 1 and 2]{LSS} for further details and some instructive examples, as well as \cite[Sections 1--3]{Carter78}). Write $\bar{H} = \bar{H}^\circ.F$, where $F$ is identified with a subgroup of the Weyl group $W = W(\bar{G})$. 
Then $\bar{H}^\circ = M R$ is a commuting product of $\s$-stable subgroups, where $M = (\bar{H}^\circ)'$ is semisimple and $R$ is a torus (possibly trivial). Then $M_\sigma$ is a subgroup of $\bar{H}_{\s}$ and we have  
\begin{equation}\label{e:struc}
M_\sigma = (\hat{X}_1 \circ \hat{X}_2 \circ \cdots \circ \hat{X}_m).D = Z.(X_1 \times X_2 \times \cdots \times X_m).D
\end{equation}
for some positive integer $m$, where each $X_i$ is a group of Lie type (simple, or isomorphic to ${\rm L}_{2}(2)$, ${\rm L}_{2}(3)$ or ${\rm U}_{3}(2)$), $\hat{X}_i$ is a cover of $X_i$, $Z$ is a subgroup of the centre of the central product of the $\hat{X}_i$, and $D$ is a group inducing diagonal automorphisms on the $X_i$ (possibly $Z=1$ and $D=1$). 
To simplify the notation, we will refer to 
\[
 X_1 \times X_2 \times \cdots \times X_m
 \] 
 as the \emph{type} of $H_0 = \bar{H}_\sigma \cap T$.  

\begin{ex}
Suppose $T = E_6(q)$ and set $d = (2,q-1)$ and $e = (3,q-1)$, so 
$\bar{G}_\sigma = T.e$. Consider a $\s$-stable maximal rank subgroup $\bar{H} = A_1 A_5$. Here $\bar{H}$ is connected and semisimple so $F = 1$ and $\bar{H} = M$. Moreover,
\[
N_{\bar{G}_{\s}}(\bar{H}_{\s}) = M_\sigma = d.(\text{L}_2(q) \times \text{L}_6(q)).de = H_0.e
\]
so in \eqref{e:struc} we have $Z = Z_d$, $D = Z_{d} \times Z_{e}$ and $H_0$ has type 
$\text{L}_2(q) \times \text{L}_6(q)$. Similarly, if $\bar{H} = D_5 T_1$ then $M = D_5$ and 
$M_\sigma = h.\text{P}\Omega^+_{10}(q)$ with $h = (4,q-1)$, so $Z=Z_h$, $D=1$ and $H_0$ has type $\text{P}\Omega^+_{10}(q)$.
\end{ex}   

\begin{ex}\label{ex1}
For a slightly more complicated example, suppose $T = E_7(q)$ with $p \neq 2$, so $\bar{G}_\sigma = T.2$. Let $\bar{H} = A_1^7.\text{L}_2(3)$, so $M = A_1^7$ and $F = \text{L}_2(3)$. According to \cite[Table 5.1]{LSS}, there are precisely two $\G_\sigma$-classes of $\s$-stable conjugates of $\bar{H}$ that give rise to maximal subgroups of $\G_\sigma$. One of these classes comprises subgroups of the form  
\[
\bar{H}_{\s} = \text{PGL}_2(q^7).7,
\]
in which case $M_\sigma = \text{PGL}_2(q^7)$ and $H_0$ has type $\text{L}_2(q^7)$. To determine the precise structure of $H_0 = \bar{H}_{\s} \cap T$, write $\bar{G} = \hat{G}/Z$, where $\hat{G}$ is the simply connected version of $\bar{G} = E_7$ and $Z = Z(\hat{G})$. Consider the subsystem subgroup $\hat{H} \leqs \hat{G}$ corresponding to $\bar{H}$. Then $Z \leqs \hat{H}^\circ$ and $\hat{H}^\circ / Z = \bar{H}^\circ$. Since $|(\hat{H}^\circ)_\sigma| = |(\bar{H}^\circ)_\sigma|$ it follows that $(\hat{H}^\circ)_\sigma / Z_\sigma = T \cap (\bar{H}^\circ)_\sigma$ has index two in $(\bar{H}^\circ)_\sigma$. Therefore
$(\bar{H}^\circ)_\sigma = \text{PGL}_2(q^7)$ and $T \cap (\bar{H}^\circ)_\sigma = \text{L}_2(q^7)$, so $H_0 = \text{L}_2(q^7).7$. 
\end{ex}

\begin{rem}
By inspecting \cite[Table 5.1]{LSS}, we see that there are just two instances where the type of $H_0$ does not uniquely determine $H_0$ up to conjugacy; namely, the subgroups of type ${\rm U}_{3}(q)$ and ${\rm U}_{5}(q^2)$ in ${}^2F_4(q)$ and $E_8(q)$, respectively. In both cases, there are two classes of subgroups of the same type, but the information given in Tables \ref{tab:mr20} and \ref{tab:mr22} is identical for subgroups in both classes, so this potential ambiguity is not an issue.   
\end{rem}

Next we describe the information presented in Tables \ref{tab:mr20}-\ref{tab:mr22}. First notice that the cases are listed in the same order as they appear in \cite[Table 5.1]{LSS} and we implicitly assume all the conditions on $G$ and $q$ given in this table. The fourth column in Table \ref{tab:mr20} (and the third in Tables \ref{tab:mr21} and \ref{tab:mr22}) provides an upper bound 
$|C_{H_0}(t)| \leqs f(q)$
and the next column records the $T$-class of $t$ (in terms of the notation of Table \ref{tabinv}), where $t \in H_0$ is the involution specified in the third column. The final column gives a lower bound $\gamma$ for $\beta$. Notice that in most cases there are two rows for each possibility for $H_0$; here $p \ne 2$ in the first row and $p = 2$ in the second.

Let us explain our notation for specifying an involution $t \in H_0$. Typically, we choose 
\[
t  = (x_1, \ldots, x_m) \in X,
\]
where
\begin{equation}\label{e:xx}
X = \hat{X}_1 \circ \hat{X}_2 \circ \cdots \circ \hat{X}_m \leqs M_\sigma
\end{equation}
is a central product of quasisimple groups of Lie type. We extend the notation from Section \ref{s:inv} to involution classes in $\hat{X}_i$ (recall that this notation is consistent with \cite{GLS} and \cite{LS_book} for $p \ne 2$ and $p=2$, respectively). 
Sometimes it will be convenient to write
\[
t = (x_1, \ldots, x_l, -)
\]
if $x_{i}=1$ for $i>l$. 

To obtain the upper bounds $|C_{H_0}(t)| \leqs f(q)$, we will often use the following easy lemma. Write $X = \hat{X}/Y$, where 
$\hat{X} = \hat{X}_1 \times \hat{X}_2 \times \cdots \times \hat{X}_m$ and $Y \leqs Z(\hat{X})$. 
Since $X \leqs H_0$ we have
\begin{equation}\label{e:cenn}
 |C_{H_0}(t)| \leqs |C_{X}(t)| |H_0 : X|.
 \end{equation}

\begin{lem}\label{l:centbound}
Let $t = (x_1, \ldots, x_m) \in X$ be an involution and write $t = Y\hat{t}$ with $\hat{t} = (\hat{x}_1, \ldots, \hat{x}_m) \in \hat{X}$. Then 
$$|C_{X}(t)| \leqs |Y|_{2'}^{-1}\prod_i |C_{\hat{X}_i}(\hat{x}_i)|.$$
\end{lem}

\begin{proof}
Without loss of generality, we may assume that $\hat{t}$ is a $2$-element.
Let $\pi:\hat{X}\to \hat{X}/Y$ be the quotient map and note that 
$$\hat{t}^{\hat{X}} \subseteq \pi^{-1}(t^X) = \bigsqcup_{i}s_iY$$
for some $s_i \in \hat{X}$. Each coset $s_iY$ contains at most $|Y|_2$ $2$-elements, so $|\hat{t}^{\hat{X}}| \leqs |Y|_2|t^X|$ and the result follows.
\end{proof}

\begin{rem}\label{r:mr1}
As indicated in the tables, we sometimes choose an involution
$t \in H_0$ which induces a graph automorphism $\gamma_1$ on one or more of the $\hat{X}_i$ factors. In the special case where $T = F_4(q)$ and $H_0 = \text{Sp}_4(q^2){:}2$ with $p=2$, we take $t$ to be an involutory field automorphism $\varphi$ of $\text{Sp}_4(q^2)$.
\end{rem}

\begin{rem}\label{r:mr2}
In the cases $(\bar{G},\bar{H}^{\circ}) = (F_4,A_2 \tilde{A}_2)$, $(E_7, A_1^7)$ and $(E_8,A_1^8)$, the ordering of the $X_i$ in \eqref{e:struc} is important. In the first case, the notation indicates that $X_1$ and $X_2$ are generated by long and short root subgroups of $T$, respectively. In the latter two cases, we assume that the final four $A_1$ factors of $\bar{H}^{\circ}$ are contained in a $D_4$ subsystem subgroup of $\bar{G}$ (this is consistent with \cite{Tho}, where the composition factors of $\mathcal{L}(\bar{G}){\downarrow}\bar{H}^{\circ}$ are given in Tables 59 and 60).  
\end{rem}

\begin{rem}\label{r:dagger}
Suppose $p \ne 2$, $T = F_4(q)$, $E_6^\e(q)$ or $E_8(q)$ and $H_0$ is of type 
$\text{L}_2(q) \times \text{PSp}_6(q)$, $\text{L}_2(q) \times \text{L}^\e_6(q)$ or $\text{L}_2(q) \times E_7(q)$, respectively. In these cases we need to explain the notation $(1,t''_i)$ appearing in Tables \ref{tab:mr20} and \ref{tab:mr22}. This denotes the image in $X = \hat{X}_1 \circ \hat{X}_2$ of an involution in the quasisimple group $\hat{X}_2$. Now $Z(\hat{X}_2)$ has order $2$ and $t'_i$, $t''_i$ are the preimages of a fixed involution of type $t_i$ in the simple group 
$\hat{X}_2 / Z(\hat{X}_2)$ (in the notation of \cite[Table 4.5.1]{GLS}). One can show that $t'_i$ and $t''_i$ are not $T$-conjugate and we take $t_i''$ to be in the smaller $T$-class.
For example, if $T = F_4(q)$ and $H_0 = (\text{SL}_2(q) \circ \text{Sp}_6(q)).2$, then $t_1',t_1'' \in {\rm Sp}_{6}(q)$ have respective Jordan forms $[-I_4,I_2]$ and $[-I_2,I_4]$. See the proof of Lemma \ref{l:pmr2odd} for further details in the other two cases of interest.
\end{rem}

We begin the proof of Proposition \ref{p:mr2} by handling the special case where $T = E_6^{\e}(q)$ and $H_0$ is of type ${\rm L}_{3}^{\e}(q^3)$. This is the only case with $\b<1/2$.

\begin{lem} \label{l:mrexception}
Suppose $T = E_6^{\e}(q)$ and $H_0$ is of type ${\rm L}_{3}^{\e}(q^3)$. Then ${\rm ifix}(T) > n^{4/9}$ and $\b=13/27$.
\end{lem} 

\begin{proof}
Here $H_0 = \text{L}_3^\e(q^3).\la \varphi \ra = \text{L}_3^\e(q^3).3$, where 
$\varphi$ is a field automorphism of order $3$. Let $t \in \text{L}_3^\e(q^3)$ be a representative of the unique class of involutions in $H_0$ (as indicated in Table \ref{tab:mr20}, $t$ is a $t_1$-involution if $p \ne 2$ and an $A_1$-involution when $p=2$). In particular,
\[
|C_{H_0}(t)|  = \left\{\begin{array}{ll}
3|\text{GL}_2^{\e}(q^3)| & p \ne 2 \\
3(q^3-\e)q^9 & p = 2
\end{array}\right.
\]
and the bound $\text{ifix}(T) > n^{4/9}$ follows. In fact, we claim that $t$ belongs to the largest class of involutions in $T$ (that is, the $T$-class labelled $t_2$ for $p \ne 2$, and $A_1^3$ for $p=2$), which implies that $\b=13/27$. To see this, first observe that we may assume
\[
t \in A = \{(x,x,x) \,:\, x \in A_2\} < A_2^3 = \bar{H}^{\circ}
\]
at the level of algebraic groups (this follows from the construction of $H_0$). Now $t$ is a representative of the unique class of involutions in $A \cong A_2$ and it is sufficient to  determine the $\bar{G}$-class of $t$. To do this, we consider the action of $\bar{H}^{\circ} = A_2^3$ on the $27$-dimensional module $V = V_{\bar{G}}(\lambda_1)$:
\[
V{\downarrow}A_2^3 = (V_1 \otimes V_2^* \otimes 0) \oplus (0 \otimes V_2 \otimes V_3^*) \oplus (V_1^* \otimes 0 \otimes V_3),
\] 
where $V_i$ denotes the natural module for the $i$-th factor in $\bar{H}^{\circ}$ and $0$ is the trivial module. For $p \ne 2$ we calculate that $t$ has Jordan form $[-I_4,I_{5}]$ on each of three summands, so $t$ has Jordan form $[-I_{12},I_{15}]$ on $V$ and by inspecting \cite[Proposition 1.2]{LS99}, we conclude that $t$ is a $t_2$-involution. Similarly, if $p=2$ then $t$ has Jordan form $[J_2^{12},J_1^3]$ on $V$ and \cite[Table 5]{Lawunip} implies that $t$ is in the $A_1^3$-class (the same conclusion also follows from \cite[Section 4.9]{Law09}). 
\end{proof}

We partition the analysis of the remaining subgroups into two cases, according to the parity of $p$. The next two lemmas settle the case $p=2$.    

\begin{lem} \label{l:pmr2even0}
Proposition \ref{p:mr2} holds if $p=2$ and $T \in \{{}^3D_4(q), {}^2F_4(q)\}$.  
\end{lem}

\begin{proof}
First assume $T = {}^2F_4(q)$ and recall that $q > 2$ (see Proposition \ref{p:small} for the case $q=2$). The maximal subgroups of $G$ have been determined by Malle \cite{Mal}. 

Suppose $H_0 = \text{SU}_3(q){:}2$ or $\text{PGU}_3(q){:}2$. In both cases, choose an involutory graph automorphism $t \in H_0$, so $|C_{H_0}(t)| = 2|\text{SL}_2(q)|$. According to Table \ref{tabinv}, we have $|C_{T}(t)| = q^{10}|{}^2B_2(q)|$ or $q^{9}|{\rm SL}_{2}(q)|$. But Lagrange's theorem implies that $|C_{H_0}(t)|$ divides $|C_T(t)|$, so $|C_{T}(t)| = q^{9}|{\rm SL}_{2}(q)|$ is the only option and thus $t$ is in the $T$-class $A_1 \tilde{A}_1$.
The case $H_0 = {}^2B_2(q)^2.2$ is similar. Here we take $t$ to be an involution in one of the ${}^2B_2(q)$ factors, so $|C_{H_0}(t)| = q^2 |{}^2B_2(q)|$ and using Lagrange's theorem we see that $t$ is in the $(\tilde{A}_1)_2$-class of $T$. Now assume $H_0 = \text{Sp}_4(q){:}2$. Let $t \in \text{Sp}_4(q)$ be a $c_2$-involution in the notation of \cite{AS}. By considering the construction of $H_0$, we may assume
\[
t \in \{(x,x) \,:\, x \in C_2 \} < C_2^2 < C_4 < F_4 = \bar{G}
\] 
and $t$ is a $c_4$-involution in $C_4$. This is the class labelled $2A_1^{(2)}$ in \cite[Table 6]{Law09}, so \cite[Table 14]{Law09} implies that $t$ is an $A_1\tilde{A}_1$-involution, as recorded in Table \ref{tab:mr20}.

Now assume $T = {}^3D_4(q)$ and note that the maximal subgroups of $G$ are given in \cite{Kl3}. For  $H_0 = \text{L}_2(q) \times \text{L}_2(q^3)$ we choose $t = (A_1,1) \in H_0$, which is a long root element in the first factor, so $|C_{H_0}(t)| = q|\text{SL}_2(q^3)|$ and using Lagrange's theorem we see that $t$ is in the $A_1$-class of $T$. Finally, let us assume $H_0$ is of type 
$\text{L}^\e_3(q)$, so $\bar{H} = A_2 T_2.2$ and 
$X = {\rm SL}_{3}^{\e}(q)$ (see \eqref{e:xx}). Let $t \in X$ be an involution. Then 
$|C_{X}(t)| = q^3(q-\e)$ and by applying the bound in \eqref{e:cenn} we deduce that
$$|C_{H_0}(t)| \leqs q^3(q-\e) \cdot 2(q^2+\e q+1) = 2q^3(q^3-\e)$$
as recorded in Table \ref{tab:mr20}. Finally, Lagrange's theorem implies that $t$ is an $A_1$-involution.
\end{proof}

\begin{lem} \label{l:pmr2even}
Proposition \ref{p:mr2} holds if $p=2$. 
\end{lem}

\begin{proof}
We may assume that $T \in \{G_2(q), F_4(q), E^\e_6(q), E_7(q), E_8(q)\}$. We will postpone the analysis of the following two cases to the end of the proof:
\begin{itemize}\addtolength{\itemsep}{0.2\baselineskip}
\item[{\rm (a)}] $T = F_4(q)$ and $H_0 = \text{Sp}_4(q^2){:}2$; 
\item[{\rm (b)}] $T = E_8(q)$ and $H_0 = \text{U}_3(q^4).8$.
\end{itemize}

In many cases, if $t \in H_0$ is a long root element in one of the $\hat{X}_i$ factors then $t$ is an $A_1$-involution in $T$ and the result quickly follows. For example, suppose $T = E_7(q)$ and $H_0 = \text{L}_2(q) \times \Omega_{12}^+(q)$. Let $x \in \Omega_{12}^+(q)$ be a long root element and set $t = (1,x) \in H_0$. Then $t$  is an $A_1$-involution in $T$ and
$$|C_{H_0}(t)| = |\text{SL}_2(q)||C_{\Omega_{12}^+(q)}(x)| = q^{17}|{\rm SL}_{2}(q)|^2|\Omega_8^{+}(q)|.$$
This yields $\text{ifix}(T) > n^{4/9}$ and $\b \geqs 3/4$.

In the remaining cases (excluding (a) and (b) above), we choose $t$ to be a long root element in one of the quasisimple factors as above and we use the embedding of $H_0$ in $\bar{G}$ to determine the $T$-class of $t$ and verify the required bounds (the difference here is that $t$ is no longer an $A_1$-involution in $T$). For example, suppose $T = E_7(q)$ and $H_0 = \text{L}_2(q^7).\text{L}_3(2)$, so $X = \text{L}_2(q^7)$ and by considering the construction of $H_0$, we may assume that 
\begin{equation}\label{e:diagg}
t \in A = \{(x,x,x,x,x,x,x) \,:\, x \in A_1 \} < A_1^7 = \bar{H}^\circ.
\end{equation}
From the restriction of $V = V_{\bar{G}}(\lambda_7)$ to $A_1^7$ we deduce that 
\[
V{\downarrow}A = (V(1) \otimes V(1) \otimes V(1))^7.
\]
Now $A \cong A_1$ has a unique class of involutions and we calculate that $t$ has Jordan form $[J_2^{28}]$ on $V$. By inspecting \cite[Table 7]{Lawunip}, we conclude that $t$ is in the $A_1^4$-class of $T$.

Finally, let us turn to the cases (a) and (b) above. First consider (a). Let $t \in H_0$ be an involutory field automorphism of ${\rm Sp}_{4}(q^2)$, so $|C_{H_0}(t)| = 2|{\rm Sp}_{4}(q)|$. We claim that $t$ is in the $(\tilde{A}_1)_2$-class of $T$. To see this, we may assume that  
\[
t \in C_2^2.2 < C_4 < F_4
\]
with $t$ interchanging the two $C_2$ factors (this follows from the construction of $H_0$). Then $t$ is an $a_4$-involution in $C_4$, which is labelled $2A_1$ in \cite[Table 6]{Law09}, and by inspecting \cite[Table 14]{Law09} we see that $t$ is in the $T$-class $(\tilde{A}_1)_2$. The result follows. 

Now consider (b). Let $t \in H_0$ be an involutory graph automorphism of ${\rm U}_{3}(q^4)$, so we have $|C_{H_0}(t)| =  8|{\rm SL}_{2}(q^4)|$. At the level of algebraic groups, $t$ acts as a graph automorphism on each $A_2$ factor of $\bar{H}^{\circ} = A_2^4$, so $t$ inverts a maximal torus of $\bar{G} = E_8$ and thus $t$ is in the $A_1^4$-class by Lemma \ref{l:wzero}(iii).
\end{proof}

\begin{lem} \label{l:pmr2odd}
Proposition \ref{p:mr2} holds if $p \ne 2$. 
\end{lem}

\begin{proof}
As before, it is sufficient to justify the relevant information presented in Tables \ref{tab:mr20}-\ref{tab:mr22}. In particular, we need to explain the choice of involution $t \in H_0$, establish the upper bound on $|C_{H_0}(t)|$ and determine the $T$-class of $t$. 

First suppose $T = G_2(q)$, so $H_0$ is of type $\text{L}_2(q)^2$ or $\text{L}_3^\e(q)$ (see \cite[Table 5.1]{LSS}). Note that $T$ has $q^4(q^4+q^2+1)$ involutions, which form a single conjugacy class, so \eqref{e:i2} holds. In the first case,  
\[
H_0 = ({\rm SL}_{2}(q) \circ {\rm SL}_{2}(q)).2
\]
and we take $t \in H_0$ to be the image of $(A,A) \in {\rm GL}_{2}(q)^2$, where 
$A = \left(\begin{smallmatrix} -1 & 0 \\ \phantom{-}0 & 1 \end{smallmatrix}\right) \in {\rm GL}_{2}(q)$ is a $t_1$-type involution (note that $t \in {\rm SL}_{2}(q) \circ {\rm SL}_{2}(q)$ if and only if $q \equiv 1 \imod{4}$). In Table \ref{tab:mr20}, we denote this choice of $t$ by writing $t = (t_1,t_1)$. Note that $|C_{\text{SL}_2(q)}(A)| = q-1$ and thus $|C_{H_0}(t)| \leqs 2(q-1)^2$ by Lemma \ref{l:centbound}. It is now easy to check that the desired bounds hold. 

Finally, if $H_0 = \text{SL}_3^\e(q){:}2$ and $t \in H_0$ is an involutory graph automorphism, then $|C_{H_0}(t)| =2|{\rm SO}_{3}(q)|$ and the result follows if $q \geqs 5$. For $q=3$ we need to compute $i_2(H_0)$ in order to get ${\rm ifix}(T)>n^{4/9}$, so we write ``$t_1$ and $\gamma_1$" in Table \ref{tab:mr20}, which are representatives for the two classes of involutions in $H_0$. Note that $i_2(H_0) = 351$ if $\e=+$, otherwise $i_2(H_0) = 315$.

To complete the proof, we will focus on the following cases to illustrate the main ideas (the other cases are similar and we omit the details):
\begin{itemize}\addtolength{\itemsep}{0.2\baselineskip}
\item[{\rm (a)}] $T = F_4(q)$ and $H_0$ is of type $\text{P}\Omega^+_8(q)$ or $\text{L}^\e_3(q)^2$. 
\item[{\rm (b)}] $T = E^\e_6(q)$ and $H_0$ is of type $\text{L}_2(q) \times \text{L}^\e_6(q)$.
\item[{\rm (c)}] $T = E_7(q)$ and $H_0$ is of type $\text{L}_2(q^7)$.
\item[{\rm (d)}] $T = E_8(q)$ and $H_0$ is of type $\text{P}\Omega_{16}^+(q)$, $\text{L}_{2}(q)\times E_7(q)$ or $\text{L}_2(q)^8$.
\end{itemize}

First assume $T = F_4(q)$ and $H_0$ is of type $\text{P}\Omega^+_8(q)$, so
$$H_0 = \text{Spin}^+_8(q).{\rm Sym}_3 = 2^2.\text{P}\Omega^+_8(q).{\rm Sym}_3.$$
Let  $t \in \text{Spin}^+_8(q)$ be a $t_2$-involution in the notation of \cite[Table 4.5.2]{GLS}. This element has centralizer $\text{SL}_2(q)^4$ in $\text{Spin}^+_8(q)$, so $|C_{H_0}(t)| \leqs 6|{\rm SL}_{2}(q)|^4$ and we claim that $t$ is a $t_1$-involution in $T$. To see this, we consider the action of $\bar{H}^\circ = D_4$ on the Lie algebra $V = \mathcal{L}(\bar{G})$: 
\[
V{\downarrow}D_4 = V(\lambda_1) \oplus V(\lambda_2) \oplus V(\lambda_3) \oplus V(\lambda_4).
\] 
We calculate that $t$ has Jordan form $[-I_{16},I_{12}]$ on $V(\lambda_2)$, and $[-I_{4},I_{4}]$ on each of the other summands, so $t$ has Jordan form $[-I_{28},I_{24}]$ on $V$. Therefore, $\dim C_{\bar{G}}(t) = 24$, so $C_{\bar{G}}(t) = A_1C_3$ and the claim follows.  

Now suppose $H_0$ is of type $\text{L}^\e_3(q)^2$, so $\bar{H} = (A_2 \tilde{A}_2).2$ and $H_0 = M_\sigma.2$ with
\[
M_\sigma = e.\text{L}^\e_3(q)^2.e = ({\rm SL}_{3}^{\e}(q) \circ {\rm SL}_{3}^{\e}(q)).e
\]
and $e = (3,q-\e)$. Here $\bar{H}^{\circ} = A_2 \tilde{A}_2$ and we may assume that the first ${\rm SL}_{3}^{\e}(q)$ factor contains long root elements of $T$ and the second short root elements. Let $A \in {\rm SL}_{3}^{\e}(q)$ be an involution (labelled $t_1$ in \cite[Table 4.5.2]{GLS}) and let $t \in {\rm SL}_{3}^{\e}(q) \circ {\rm SL}_{3}^{\e}(q)$ be  the image of $(1,A) \in {\rm SL}_{3}^{\e}(q) \times {\rm SL}_{3}^{\e}(q)$, which is denoted $t = (1,t_1)$ in Table \ref{tab:mr20}. By applying Lemma \ref{l:centbound} (with $|Y|_{2'}=e$) we get
\[
|C_{H_0}(t)| \leqs 2(q-\e)|\text{SL}_2(q)||\text{SL}_3^\e(q)|
\]
and it remains to show that $t$ is a $t_4$-involution of $T$. Once again, to do this we consider the action of $\bar{H}^{\circ}$ on $V = \mathcal{L}(\bar{G})$:
\[
V{\downarrow}A_2\tilde{A}_2 = \mathcal{L}(A_2\tilde{A}_2) \oplus (V(\l_1) \otimes V(2\l_2)) \oplus (V(\l_2) \otimes V(2\l_1)).
\] 
We calculate that $t$ has Jordan form $[-I_{4},I_{12}]$ on $\mathcal{L}(A_2\tilde{A}_2)$, 
and $[-I_6,I_{12}]$ on the two other summands, whence $t$ has Jordan form $[-I_{16},I_{36}]$ on $V$. Therefore, $C_{\bar{G}}(t)=B_4$ and this justifies the claim. 

Next assume $T = E^\e_6(q)$ and $H_0$ is of type $\text{L}_2(q) \times \text{L}^\e_6(q)$. Here $\bar{H} = A_1 A_5$ and $H_0 = (\text{SL}_2(q) \circ \text{SL}^\e_6(q)).2$.
As explained in Remark \ref{r:dagger}, an involution in the $t_2$-class of $\text{L}^\e_6(q)$ has preimages $t_2', t_2'' \in \text{SL}^\e_6(q)$, with respective Jordan forms $[-I_2,I_4]$ and $[-I_4,I_2]$. Consider the images of $(1,t_2')$ and $(1,t_2'')$ in ${\rm SL}_{2}(q) \circ \text{SL}^\e_6(q)$. We claim that $(1,t_2'')$ is a $t_1$-involution in $T$, while $(1,t_2')$ is of type $t_2$. To see this, let $V =  \mathcal{L}(\bar{G})$ and observe that
\[
V{\downarrow}A_1A_5 = \mathcal{L}(A_1A_5) \oplus (V(1) \otimes V(\lambda_3)).
\] 
Both $(1,t_2')$ and $(1,t_2'')$ have Jordan form $[-I_{16},I_{22}]$ on $\mathcal{L}(A_1A_5)$ (indeed, both elements have centralizer $A_1^2A_3T_1$ in $A_1A_5$, which is $22$-dimensional), whereas the Jordan forms on $V(1) \otimes V(\lambda_3)$ are $[-I_{24},I_{16}]$ and $[-I_{16},I_{24}]$, respectively. This justifies the claim, and we choose $t = (1,t_2'')$ as in Table \ref{tab:mr20}. Finally, we note that $|C_{H_0}(t)| \leqs 2|\text{SL}_2(q)|^2|\text{GL}_{4}^\e(q)|$ since $|C_{\text{SL}^\e_6(q)}(t_2')| = |\text{SL}_2(q)||\text{GL}_{4}^\e(q)|$.

Now assume $T = E_7(q)$ and $H_0$ is of type $\text{L}_2(q^7)$. Here $\bar{H} = A_1^7.{\rm L}_{3}(2)$, $\bar{H}_{\s} = \text{PGL}_2(q^7).7$ and $H_0 = \text{L}_2(q^7).7$ (see Example \ref{ex1}). Now $H_0$ has a unique class of involutions $t^{H_0}$, labelled by $t_1$ if $q \equiv 1 \imod{4}$ and $t_1'$ if $q \equiv 3 \imod{4}$ (see \cite[Table 4.5.1]{GLS}); this explains why we write ``$t_1$ or $t_1'$" in Table \ref{tab:mr21}. Note that $|C_{H_0}(t)| \leqs 7(q^7+1)$. To identify the $T$-class of $t$, we may assume that $t$ is contained in a diagonal $A_1$-type subgroup $A <\bar{G}$ as in the proof of Lemma \ref{l:pmr2even} (see \eqref{e:diagg}), which has a unique class of involutions. If we set $V = \mathcal{L}(\bar{G})$, then using the restriction of $V$ to $A_1^7$, we calculate that
\[ V {\downarrow} A = (V(2) \otimes V(2))^7 \oplus V(2)^{21} \oplus 0^7\]
and thus $t$ has Jordan form $[-I_{70},I_{63}]$. Therefore $\dim C_{\bar{G}}(t) = 63$, so $C_{\bar{G}}(t) = A_7$ and we deduce that $t$ is a $t_4$-involution if $q \equiv 1 \imod{4}$, otherwise $t$ is in the $t'_4$-class.

To complete the proof, let us assume $T = E_8(q)$ and $H_0$ is one of the subgroups in (d) above. First assume $H_0$ is of type $\text{P}\Omega_{16}^+(q)$. Here $\bar{H}$ is a half-spin group of type $D_8$ and we have
\[
H_0 = \text{HSpin}^+_{16}(q).2 = 2.{\rm P\Omega}_{16}^{+}(q).2.
\]
Let $t \in \text{HSpin}^+_{16}(q)$ be the image of an involution in the $t_4$-class of $\text{Spin}^+_{16}(q) = 2.\text{HSpin}^+_{16}(q)$ (in the notation of \cite[Table 4.5.2]{GLS}), so $|C_{H_0}(t)| \leqs 2|{\rm SO}_{8}^{+}(q)|^2$. Then $C_{\bar{H}}(t)^{\circ} = D_4D_4$ and thus $t$ is a $t_1$-involution in $T$.

Next assume $H_0$ is of type $\text{L}_{2}(q) \times E_7(q)$, so 
\[
H_0 = 2.(\text{L}_{2}(q) \times E_7(q)).2 = ({\rm SL}_{2}(q) \circ \hat{E}_7(q)).2
\]
where $\hat{E}_7(q) = 2.E_7(q)$ is the simply connected group. An involution in the $t_1$-class of $E_7(q)$ has preimages $t_1'$ and $t_1''$ in $\hat{E}_7(q)$, with trace $-8$ and $8$ on $V_{E_7}(\lambda_1)$, respectively (see \cite[Proposition 1.2]{LS99}). Consider the images of $(1,t_1')$ and $(1,t_1'')$ in ${\rm SL}_{2}(q) \circ \hat{E}_7(q)$. We claim that $(1,t_1')$ is a $t_1$-involution in $T$, while $(1,t_1'')$ is of type $t_8$. To see this, let $V =  \mathcal{L}(\bar{G})$ and observe that
\[
V{\downarrow}A_1E_7 = \mathcal{L}(A_1E_7) \oplus (V(1) \otimes V(\lambda_7)).
\] 
Both $(1,t_1')$ and $(1,t_1'')$ have Jordan form $[-I_{64},I_{72}]$ on $\mathcal{L}(A_1E_7)$ (indeed, both elements have centralizer $A_1^2D_6$ in $A_1E_7$, which is $72$-dimensional), whereas the Jordan forms on $V(1) \otimes V(\lambda_7)$ are $[-I_{64},I_{48}]$ and $[-I_{48},I_{64}]$, respectively. This justifies the claim, and we choose $t = (1,t_1'')$ as in Table \ref{tab:mr22}. Finally, we note that $|C_{H_0}(t)| \leqs 2|\text{SL}_2(q)|^2|\text{SO}_{12}^+(q)|$ since $|C_{\hat{E}_7(q)}(t_1'')| = |\text{SL}_2(q)||\text{SO}_{12}^+(q)|$.

Finally, let us assume that $H_0$ is of type $\text{L}_2(q)^8$, so 
\[
H_0 = M_\sigma.\text{AGL}_3(2) = 2^4.\text{L}_2(q)^8.2^4.\text{AGL}_3(2).
\]
As in \cite[Lemma 2.14]{CLSS}, write $\bar{H}^\circ = A_1^8 = J_1\cdots J_8$ and $Z(J_i) = \la e_i \ra$. Recall that we may assume $J_5J_6J_7J_8$ is contained in a $D_4$ subsystem subgroup of $\bar{G}$ (see Remark \ref{r:mr2}), which implies that $J_1J_2J_3J_4$ is also contained in such a subsystem subgroup
(for instance, we can take the $A_1$ subgroups $\la X_{\a},X_{-\a}\ra$ with $\a \in \{\alpha_2, \alpha_3, \alpha_5, -\a_2-\a_3-2\a_4-\a_5\}$). 
As before, let $A = \left(\begin{smallmatrix} -1 & 0 \\ \phantom{-}0 & 1 \end{smallmatrix}\right) \in {\rm GL}_{2}(q)$, a $t_1$-type involution, and let $t$ be the image in $H_0$ of $(A,A,A,A,1,1,1,1) \in {\rm GL}_{2}(q)^8$. Our choice of ordering for the $J_i$ implies that $e_1 e_2 e_3 e_4 = 1$ by \cite[Lemma 2.14(ii)]{CLSS} and hence $t$ is indeed an involution. Moreover, we calculate that
$$|C_{H_0}(t)| \leqs 2^4|\text{AGL}_3(2)|(q-1)^4|\text{SL}_2(q)|^4 = 21504(q-1)^4|\text{SL}_2(q)|^4.$$ 
By considering the restriction of $\mathcal{L}(\bar{G})$ to $J_1J_2J_3J_4=A_1^4$ (via the restriction to $A_1^8$), or by inspecting Step 2 in the proof of \cite[Lemma 2.17]{CLSS}, we deduce that $t$ is a $t_8$-involution in $T$. The result follows.
\end{proof}

\vs
This completes the proof of Theorem \ref{t:mr}.

\section{Proof of Theorem \ref{t:main2}, Part I}\label{s:algebraic}

By our earlier work in Sections \ref{s:parab} and \ref{s:mr}, we have established Theorem \ref{t:main2} in the case where $H_0$ contains a maximal torus of $T$ and we now move on to consider the remaining possibilities for $H$. It will be convenient to postpone the analysis of the twisted groups ${}^2B_2(q)$, ${}^2G_2(q)$, ${}^2F_4(q)$ and ${}^3D_4(q)$ to Section \ref{s:small}, so in the next four sections we will assume 
\begin{equation}\label{e:T}
T \in \{E_8(q), E_7(q), E_6^{\e}(q), F_4(q), G_2(q)\},
\end{equation}
where $q=p^f$ with $p$ a prime. 
The following fundamental result (see \cite[Theorem 2]{LS90}) partitions the remaining subgroups into various types (this allows us to refer to \emph{type {\rm (I)} actions} and \emph{type {\rm (II)} subgroups} of $G$, and so on).

\begin{thm}\label{t:types}
Let $G$ be an almost simple group with socle $T = (\bar{G}_{\s})'$ as in \eqref{e:T}. Let $H$ be a maximal subgroup of $G$ with $G = HT$ and assume $H$ does not contain a maximal torus of $T$. Set $H_0 = H \cap T$. Then one of the following holds:
\begin{itemize}\addtolength{\itemsep}{0.2\baselineskip}
\item[{\rm (I)}] $H = N_G(\bar{H}_{\s})$, where $\bar{H}$ is a maximal closed $\sigma$-stable positive dimensional subgroup of $G$ (not parabolic or of maximal rank);
\item[{\rm (II)}] $H$ is of the same type as $G$ (possibly twisted) over a subfield of $\mathbb{F}_{q}$;
\item[{\rm (III)}] $H$ is an exotic local subgroup (determined in \cite{CLSS});
\item[{\rm (IV)}] $H$ is almost simple, and not of type (I) or (II);
\item[{\rm (V)}] $T=E_8(q)$, $p \geqs 7$ and $H_0 = ({\rm Alt}_5 \times {\rm Alt}_6).2^2$.
\end{itemize}
\end{thm}

It is easy to check that ${\rm ifix}(T)>n^{4/9}$ in (V), so  
it remains to handle the subgroups of type (I) - (IV). In this section, we will focus on the cases arising in part (I) of Theorem \ref{t:types} (the subgroups in (II), (III) and (IV) will be handled in Sections \ref{s:subfield}, \ref{s:exotic} and \ref{s:as}, respectively). 

Following \cite[Theorem 8]{LS03}, we partition the type (I) subgroups into three cases:
\begin{itemize}\addtolength{\itemsep}{0.2\baselineskip}
\item[{\rm (a)}] $T = E_7(q)$, $p \geqs 3$ and $\bar{H}_{\s} = (2^2 \times {\rm P\Omega}_{8}^{+}(q).2^2).{\rm Sym}_3$ or ${}^3D_4(q).3$;
\item[{\rm (b)}] $T = E_8(q)$, $p \geqs 7$ and $H_0 = {\rm PGL}_{2}(q) \times {\rm Sym}_{5}$; \hfill\refstepcounter{equation}(\theequation)\label{e:list}
\item[{\rm (c)}] $(T, {\rm soc}(H_0))$ is one of the cases listed in Table \ref{tabmaxfin} (see \cite[Table 3]{LS03}).
\end{itemize}

\begin{rem}\label{r:cor}
We take the opportunity to clarify a couple of potential ambiguities in \cite[Table 3]{LS03} that arise when $T = E_6^{\e}(q)$. Set $\bar{G} = E_6$. 
\begin{itemize}\addtolength{\itemsep}{0.2\baselineskip}
\item[{\rm (i)}] First we observe that the maximal subgroups $A_2.2$ and $G_2$ of $\bar{G}$ are not $\sigma$-stable when $\sigma$ induces a graph automorphism of $\bar{G}$. Indeed, $\sigma$-stability would imply that such a subgroup is centralized by a graph automorphism of $\bar{G}$ and is therefore contained in either $F_4$ or $C_4$, contradicting maximality. This explains the condition $\e=+$ in Table \ref{tabmaxfin} for $\text{soc}(H_0) = \text{L}_3^\pm(q)$ and $G_2(q)$. 
\item[{\rm (ii)}] We also note that $\bar{H} = A_2 G_2$ is a maximal subgroup of $\bar{G}$, with  the property that a graph automorphism of $\bar{G}$ induces a graph automorphism on the $A_2$ factor. In particular, $N_{\bar{G}}(\bar{H}) = \bar{H}$, which means that the value $t=2$ given in the fourth column of \cite[Table 10.1]{LS04} should be $t=1$. Indeed, the restriction of $V_{\bar{G}}(\lambda_1)$ to $\bar{H}$ given in the third column of \cite[Table 10.2]{LS04} shows that no element in $\bar{G}$ induces a graph automorphism on the $A_2$ factor. In terms of Table \ref{tabmaxfin}, this means that $\text{soc}(H_0) = \text{L}_3^\e(q) \times G_2(q)$ for the corresponding subgroup of $T$. 
\end{itemize}
\end{rem}

{\small
\renewcommand{\arraystretch}{1.2}
\begin{table}
\begin{center}
$$\begin{array}{ll}\hline
T & {\rm soc}(H_0) \\ \hline
E_8(q) & {\rm L}_{2}(q) \, (\mbox{$3$ classes, $p \geqs 23,29,31$}), \Omega_5(q) \, (p \geqs 5), {\rm L}_{2}(q) \times {\rm L}_{3}^{\e}(q) \, (p \geqs 5), \\
& G_2(q) \times F_4(q),  {\rm L}_{2}(q) \times G_2(q)^2 \, (p \geqs 3, q \geqs 5),  {\rm L}_{2}(q) \times G_2(q^2)\, (p \geqs 3, q \geqs 5) \\

E_7(q) &  {\rm L}_{2}(q) \, (\mbox{$2$ classes, $p \geqs 17,19$}),  {\rm L}_{3}^{\e}(q) \, (p \geqs 5), {\rm L}_{2}(q)^2 \, (p \geqs 5), \\
& {\rm L}_{2}(q) \times G_2(q)\, (p \geqs 3, q \geqs 5),  {\rm L}_{2}(q) \times F_4(q)\, (q \geqs 4),  G_2(q) \times {\rm PSp}_{6}(q) \\

E_6^{\e}(q) & {\rm L}_{3}^{\pm}(q) \, (\e=+, p \geqs 5),  G_2(q) \, (\e=+, p \ne 7), {\rm PSp}_{8}(q) \, (p \geqs 3), F_4(q),\,  \\
& {\rm L}_{3}^{\e}(q) \times G_2(q)\, ((q,\e) \ne (2,-)) \\

F_4(q) &  {\rm L}_{2}(q) \, (p \geqs 13),  G_2(q) \, (p=7), {\rm L}_{2}(q) \times G_2(q)\, (p \geqs 3, q \geqs 5) \\

G_2(q) &  {\rm L}_{2}(q) \, (p \geqs 7) \\ \hline
\end{array}$$
\caption{Some possibilities for ${\rm soc}(H_0)$ in Theorem \ref{t:types}(I)}
\label{tabmaxfin}
\end{center}
\end{table}
\renewcommand{\arraystretch}{1}}

\begin{thm}\label{t:nonmr}
Let $G$ be an almost simple primitive permutation group of degree $n$, with socle $T$ and point stabilizer $H$. Assume $T$ is one of the groups in \eqref{e:T} and $H$ is of type {\rm (I)} in Theorem \ref{t:types}. Then ${\rm ifix}(T)>n^{4/9}$. Moreover,   
\begin{equation}\label{e:logg}
\liminf_{q\to \infty} \frac{\log {\rm ifix}(T)}{\log n} = \b
\end{equation}
and either $\b \geqs 1/2$, or $(T,H_0,\b)$ is one of the cases in Table $\ref{tabb:nonmr}$.
\end{thm}

\begin{rem}\label{r:thm3}
By combining the asymptotic statement in Theorem \ref{t:nonmr} with the analogous statements in Theorems \ref{t:parab} and \ref{t:mr}, this completes the proof of Theorem \ref{t:main3}.
\end{rem}

{\small
\renewcommand{\arraystretch}{1.2}
\begin{table}
\begin{center}
$$\begin{array}{llcl}\hline
T & {\rm soc}(H_0) & \b & \mbox{Conditions} \\ \hline   
E_8(q) & \Omega_5(q) & 58/119 & p \geqs 5 \\
& {\rm L}_{2}(q) & 17/35 &  \mbox{$3$ classes; $p \geqs 23, 29, 31$} \\
E_7(q) & {\rm L}_{2}(q) & 31/65 & \mbox{$2$ classes; $p \geqs 17, 19$} \\
F_4(q) & G_2(q) & 9/19 & p = 7 \\
& {\rm L}_{2}(q) & 23/49 & p \geqs 13 \\  
G_2(q) & {\rm L}_2(q) & 5/11 & p \geqs 7 \\ \hline
\end{array}$$
\caption{Type (I) actions with $\b < 1/2$ in \eqref{e:logg}}
\label{tabb:nonmr}
\end{center}
\end{table}
\renewcommand{\arraystretch}{1}}

We start by handling the special cases in (a) and (b) above (see \eqref{e:list}). 

\begin{lem}\label{l:pp1}
Theorem \ref{t:nonmr} holds if $T = E_7(q)$, $p \geqs 3$ and $\bar{H}_{\s} = (2^2 \times {\rm P\Omega}_{8}^{+}(q).2^2).{\rm Sym}_3$ or ${}^3D_4(q).3$.
\end{lem}

\begin{proof}
In both cases, $\bar{H} = (2^2 \times D_4).{\rm Sym}_3$. Set $V = \mathcal{L}(\bar{G})$ and note that 
\[
V{\downarrow}D_4 = \mathcal{L}(D_4) \oplus V(2\lambda_1) \oplus V(2\lambda_3) \oplus V(2\lambda_4).
\]
Let $t \in \bar{H}^{\circ} = D_4$ be a $\bar{t}_2$-involution (in the notation of \cite[Table 4.3.1]{GLS}). Now $t$ has Jordan form $[-I_{16},I_{12}]$ on $\mathcal{L}(D_4)$, and it acts as $[-I_{16},I_{19}]$ on each of the other three summands. Therefore, $t$ has Jordan form $[-I_{64},I_{69}]$ on $V$ and thus $t$ is a $\bar{t}_1$-involution in $\bar{G}$. 

First assume $\bar{H}_{\s} = (2^2 \times {\rm P\Omega}_{8}^{+}(q).2^2).{\rm Sym}_3$. Let $t$ be a $t_2$-involution (in the notation of \cite[Table 4.5.1]{GLS}) in the ${\rm P\Omega}_{8}^{+}(q)$ subgroup of $H_0$. Then $|C_{H_0}(t)| \leqs 96|\text{L}_2(q)|^4$ and the above calculation shows that $t$ is in the $t_1$-class of $T$. This yields $\text{ifix}(T) > n^{4/9}$ and $\b \geqs 1/2$. 

Finally, suppose $\bar{H}_{\s} = {}^3D_4(q).3$. Now $\bar{H}_{\s}$ contains a unique class of involutions, say $t^{\bar{H}_{\s}}$, which is labelled $t_2$ in Table \ref{tabinv}. Then  
$|C_{H_0}(t)| = 3|\text{SL}_2(q)||\text{SL}_2(q^3)|$ and once again the above argument shows that $t$ is a $t_1$-involution of $T$. The desired result follows. 
\end{proof}

\begin{lem}\label{l:pp2}
Theorem \ref{t:nonmr} holds if $T = E_8(q)$, $p \geqs 7$ and $H_0 = {\rm PGL}_{2}(q) \times {\rm Sym}_{5}$.
\end{lem}

\begin{proof}
Here $\bar{H} = X \times \text{Sym}_5$ and $X$ is an adjoint group of type $A_1$. We claim that $X$ is contained in a maximal rank subgroup $M = A_1E_7$ of $\bar{G}$. Given the claim, it follows that $Z(M) \leqs C_{\bar{G}}(X) = \text{Sym}_5$ and thus the ${\rm Sym}_{5}$ factor of $H_0$ contains a $t_8$-involution of $T$. This immediately implies that $\text{ifix}(T) > n^{4/9}$ and $\b \geqs 1/2$. 

The claim can be deduced by carefully inspecting \cite[Tables 13 and 13A]{Tho}. To do this, 
first note that $X$ is constructed in \cite[Lemma 1.5]{LS90} as a diagonal subgroup of $A_1 A_1 < A_4 A_4$, with each $A_1$ acting as $V(4)$ on $V_{A_4}(\lambda_1)$. In the notation of \cite{Tho}, we have $X = E_8(\#17^{\{\underline{0}\}})$, which is a subgroup of $D_8$ with
$$V_{D_8}(\lambda_1){\downarrow} X = V(6) \oplus V(4) \oplus V(2)\oplus 0.$$ 
It follows that $X$ is contained in a subgroup $A_1^2 D_6 < D_8$ (this can also be seen by inspecting \cite[Table 13A]{Tho}) and we conclude by noting that $A_1^2D_6 < A_1E_7$. This justifies the claim and the result follows. 
\end{proof}

To complete the proof of Theorem \ref{t:nonmr}, we may assume that $(T,{\rm soc}(H_0))$ is one of the cases listed in Table \ref{tabmaxfin}. Our approach is very similar to the proof of Proposition \ref{p:mr2}. As recorded in Table \ref{tab:nonmr}, we identify an involution $t \in H_0$, we compute an upper bound $|C_{H_0}(t)| \leqs f(q)$ and we determine the $T$-class of $t$, which allows us to use the bounds in \eqref{e:mrbd}. In this way, we deduce that ${\rm ifix}(T)>n^{4/9}$ and we establish the bound $\b \geqs \gamma$, where $\gamma$ is given in the final column of Table \ref{tab:nonmr}. In the table, we list the socle of $H_0$ -- in some cases, this does not uniquely determine $H_0$ (up to conjugacy), but this potential ambiguity has no effect on the analysis of these cases. 

\begin{rem} \label{r:nmrorder}
Consider the case in Table \ref{tabmaxfin} with $T = E_7(q)$ and ${\rm soc}(H_0) = {\rm L}_{2}(q)^2$, so $\bar{G} = E_7$, $p \geqs 5$ and $\bar{H} = J_1J_2$ is a direct product of two adjoint groups of type $A_1$. Write $\bar{G} = \hat{G} / Z$ and $\bar{H} = \hat{H}/Z$, where $\hat{G}$ is the simply connected group and $Z = Z(\hat{G})$. To be consistent with \cite{Tho}, we will order the factors of $\bar{H}$ so that $\hat{H} = \hat{J}_1\hat{J}_2$ and $\hat{J}_1$ is simply connected (note that the composition factors of $\mathcal{L}(\bar{G}) {\downarrow} \bar{H}$ are given in \cite[Table 59]{Tho}). It follows that 
$$H_0 = (\hat{J}_1\hat{J}_2)_\sigma / Z = (\text{SL}_2(q) \times \text{PGL}_2(q)) / Z = \text{L}_2(q) \times \text{PGL}_2(q)$$ 
and this explains why we can choose an involution of the form $(1,t_1)$ in $H_0$. 
\end{rem}

\begin{rem} \label{r:nmr}
Let us explain the entries $\gamma_1$ and $(1,\gamma_1)$ appearing in the third column of Table \ref{tab:nonmr}. In both cases, $\gamma_1$ denotes an involutory graph automorphism on the $\text{L}_3^{\delta}(q)$ factor of $\text{soc}(H_0)$. To justify the existence of these involutions in $H_0$, let $(\bar{G},\bar{H}) = (E_6, A_2.2)$ and $(E_8,A_1A_2.2)$ be the corresponding algebraic groups, where $\bar{H}$ is $\s$-stable.
From the structure and maximality of $\bar{H}$, it follows that there exists an involution in $\bar{H} \setminus \bar{H}^\circ$ inducing a graph automorphism on the $A_2$ factor of $\bar{H}^\circ$. Moreover, there is a unique conjugacy class of such involutions in $\bar{H} \setminus \bar{H}^\circ$ (see \cite[Table 4.3.1]{GLS}), so this class is $\sigma$-stable and thus Lang's theorem implies that $\bar{H}_\sigma$ (and thus $H_0$) contains an involution inducing a graph automorphism on the $\text{L}_3^{\delta}(q)$ factor.  
\end{rem}

{\small 
\renewcommand{\arraystretch}{1.2}
\begin{table}
\begin{center}
$$\begin{array}{lllllc}\hline
T & {\rm soc}(H_0) & t & f(q)  & t^T & \gamma \\ \hline
G_2(q) & {\rm L}_{2}(q) & \mbox{$t_1$ and $t'_1$}  & \mbox{$2(q-1)$ and $2(q+1)$}  & t_1 & 5/11 \\
& & & & & \\
F_4(q) & {\rm L}_{2}(q) \times G_2(q) & (t_1,1) & 2(q-1)|G_2(q)| & t_4 & 3/5 \\
& G_2(q) & t_1 & |\text{SL}_2(q)|^2 & t_1 & 9/19 \\
& {\rm L}_{2}(q) & t_1 & 2(q-1) & t_1 & 23/49 \\
& & & & \\
E_6^{\e}(q) & F_4(q) & t_1 & |\text{SL}_2(q)||\text{Sp}_6(q)| & t_2 & 7/13 \\
& & A_1 & q^{15}|\text{Sp}_6(q)| & A_1 &  10/13 \\
& \text{PSp}_8(q) & t_2 & 2|\text{Sp}_4(q)|^2 & t_1 & 13/21 \\
& \text{L}_3^\epsilon(q) \times G_2(q) & (t_1,1)  & |\text{GL}^\epsilon_2(q)||G_2(q)| & t_1 & 1/2 \\
& & (1,A_1) & q^5|\text{SL}_2(q)||\text{SL}_3^\epsilon(q)| & A_1 & 5/7 \\
& \text{L}^{\pm}_3(q), \, \e = + & \gamma_1  & 2|\text{SO}_3(q)| & t_2 &  1/2 \\
& G_2(q), \, \e = + & t_1  & |\text{SL}_2(q)|^2 & t_1 & 1/2 \\
& & \tilde{A}_1  & q^3|\text{SL}_2(q)| & A_1^3 & 1/2 \\
& & & & & \\
E_7(q) & G_2(q) \times \text{PSp}_6(q) & (t_1,t_1) & |\text{SL}_2(q)|^3|\text{Sp}_4(q)| & t_1 & 25/49 \\
& & (A_1,1)  & q^5|\text{SL}_2(q)||\text{Sp}_6(q)| & A_1 & 40/49 \\ 
& \text{L}_2(q) \times F_4(q) & (1,t_1) & |\text{SL}_2(q)|^2|\text{Sp}_6(q)| & t_1 & 7/13 \\
& & (1,A_1) & q^{15}|\text{SL}_2(q)||\text{Sp}_6(q)| & A_1 & 10/13 \\
& \text{L}_2(q) \times G_2(q)  & (1,t_1) & |\text{SL}_2(q)|^3 & t_1 & 15/29 \\
& \text{L}_2(q)^2   & (1,t_1) & 2|\text{GL}_2(q)| & t_1 & 65/127 \\
& \text{L}_3^\epsilon(q)  & t_1 & 2|\text{GL}^\e_2(q)| & t_1 & 13/25 \\
& \text{L}_2(q)  & \mbox{$t_1$ or $t'_1$} & 2(q+1) & \mbox{$t_4$ or $t'_4$} & 31/65 \\
& & & & & \\
E_8(q) & G_2(q) \times F_4(q) & (t_1,t_4) & |\text{SL}_2(q)|^2|\text{SO}_9(q)| & t_8 & 47/91 \\
& & (1,A_1) & q^{15}|G_2(q)||\text{Sp}_6(q)| & A_1 & 10/13 \\
& \text{L}_2(q) \times G_2(q)^2  & (1,t_1,t_1) & 2|\text{SL}_2(q)|^5 & t_8 & 121/217 \\   
& \text{L}_2(q) \times G_2(q^2)  & (1,t_1) & 2|\text{SL}_2(q)||\text{SL}_2(q^2)|^2 & t_8 & 121/217 \\
& \text{L}_2(q) \times \text{L}_3^\epsilon(q)  & (1,\gamma_1) & 2|\text{SL}_2(q)||\text{SO}_3(q)| & t_8 & 130/237 \\ 
& \Omega_5(q)  &  t_1 & 2|\text{GL}_2(q)| & t_1 & 58/119 \\
& \text{L}_2(q)  & t_1 & 2(q-1) & t_1 & 17/35 \\ \hline
\end{array}$$
\caption{The algebraic subgroups in Theorem \ref{t:nonmr}}
\label{tab:nonmr}
\end{center}
\end{table}
\renewcommand{\arraystretch}{1}}

\begin{proof}[Proof of Theorem \ref{t:nonmr}]
We need to justify the information presented in Table \ref{tab:nonmr}. The arguments are similar (and easier) than those in the previous section for maximal rank subgroups, so we only provide a brief sketch. Let us write 
\[
{\rm soc}(H_0) = X_1 \times \cdots \times X_m,
\]
where each $X_i$ is a simple group of Lie type (note that $m \leqs 3$). The possibilities for  $\bar{H}$ are given in \cite[Table 1]{LS04} and we note that 
$N_{\bar{G}_{\s}}(\bar{H}_{\s}) = \bar{H}_{\s}$, so either $H_0 = \bar{H}_\sigma$, or $T = E_6^{\e}(q)$ or $E_7(q)$ and $|\bar{H}_{\s}:H_0| \leqs (3,q-\e)$ or $(2,q-1)$, respectively. In fact, it is not too difficult to determine the precise structure of $H_0$, but the above information will be sufficient for our purposes.

First assume $p=2$. Here the only possibilities for $T$ and ${\rm soc}(H_0)$ are as follows:
\[
\begin{array}{ll}
T = E_8(q): & G_2(q) \times F_4(q) \\
T = E_7(q): & {\rm L}_{2}(q) \times F_4(q)\, (q \geqs 4),\;  G_2(q) \times {\rm Sp}_{6}(q) \\
T = E_6^{\e}(q): & G_2(q) \, (\e=+),\;  F_4(q),\;  {\rm L}_{3}^{\e}(q) \times G_2(q) \, ((q,\e) \ne (2,-))
\end{array}
\]
If $T = E_6(q)$ and ${\rm soc}(H_0) = G_2(q)$ then let $t \in H_0$ be an involution in the class $\tilde{A}_1$, so $|C_{H_0}(t)| = q^3|\text{SL}_2(q)|$ (see Table \ref{tabinv}). By inspecting \cite[Table 31]{Law09}, we see that $t$ is in the $T$-class $A_1^3$ and it is easy to check that the desired bounds hold. In each of the remaining cases, we can choose an involution $t \in X_i$ that is in the $A_1$-class of $T$ (we choose $i$ so that $|t^{X_i}|$ is maximal) and the result quickly follows.  

For the remainder, let us assume $p$ is odd. First consider the cases in Table \ref{tabb:nonmr}, where we claim that $\beta < 1/2$. Here it is straightforward to find a representative of each class of involutions in $H_0$ (or in $\bar{H}_\sigma$ if $T = E_7(q)$), using \cite[Table 4.5.1]{GLS}, and we can determine the corresponding $T$-class by studying the restriction of $\mathcal{L}(\bar{G})$ to $\bar{H}$ (see \cite[Table 10.1]{LS04}). It turns out that every involution in $H_0$ is contained in the largest $T$-class and the result follows. 

For example, suppose $T = F_4(q)$ and $\text{soc}(H_0) = \text{L}_2(q)$. Here $H_0 = \text{PGL}_2(q)$ has two classes of involutions, say $x_1$ and $x_2$, with labels $t_1$ and $t'_1$, respectively. Now $\bar{H} = A_1$ has a unique class of involutions, say $y^{\bar{H}}$, so $x_1$ and $x_2$ are 
$\bar{H}$-conjugate and thus $\bar{G}$-conjugate. In view of Remark \ref{r:invcon}, all the involutions in $H_0$ are $T$-conjugate. To identify the $T$-class, we calculate the dimension of the $1$-eigenspace of $y$ on $V = \mathcal{L}(\bar{G})$. By \cite[Table 10.1]{LS04}, the restriction $V{\downarrow}\bar{H}$ has the same composition factors as the $\bar{H}$-module 
$$W = W(22) \oplus W(14) \oplus W(10) \oplus W(2),$$
where $W(m)$ is the Weyl module for $\bar{H}$ of highest weight $m \lambda_1$. Since $V{\downarrow}\la y \ra$ is completely reducible, we just need to calculate the Jordan form of $y$ on $W$, which is easily seen to be $[-I_{28},I_{24}]$ (note that $y$ has Jordan form 
$[-I_{m+1},I_{m}]$ on $W(2m)$ with $m$ odd). We conclude that the involutions in $T$ are in the $t_1$-class of $T$. The result follows.

Similar arguments apply in each of the remaining cases. We highlight the following three cases:
\begin{itemize}\addtolength{\itemsep}{0.2\baselineskip}
\item[{\rm (a)}] $T = G_2(q)$ and ${\rm soc}(H_0) = {\rm L}_{2}(q)$;
\item[{\rm (b)}] $T = E_6^{\e}(q)$ and ${\rm soc}(H_0) = \text{L}_3^{\e}(q)$;
\item[{\rm (c)}] $T=E_8(q)$ and ${\rm soc}(H_0) = \text{L}_2(q) \times \text{L}_3^{\e}(q)$.
\end{itemize}

Consider (a). This is the only case where we need to compute $|t^T \cap H_0|$ in order to verify the bound $\text{ifix}(T) > n^{4/9}$. Here $H_0 = \text{PGL}_2(q)$ and $t^T \cap H_0 = (t_1)^{H_0} \cup (t_1')^{H_0}$, so $|t^T \cap H_0| = i_2(H_0) = q^2$ (this explains why we write ``$t_1$ and $t_1'$" in the first row of Table \ref{tab:nonmr}). Note that $T$ has a unique class of involutions.

Now consider case (b). Here $p \geqs 5$ and $H_0$ contains an involutory graph automorphism $t = \gamma_1$ of $\text{L}_3^{\e}(q)$, as explained in Remark \ref{r:nmr}. We claim that $t$ is in the $T$-class $t_2$. To see this, let $V=\mathcal{L}(\bar{G})$ and observe that  
\[
V{\downarrow}A_2 = \mathcal{L}(A_2) \oplus V(4\l_1+\l_2) \oplus V(\l_1+4\l_2). 
\] 
Since $t$ interchanges the final two summands, we deduce that $t$ has Jordan form $[-I_{40},I_{38}]$ on $V$ and the claim follows. Similarly, in (c) we choose $t = (1,\gamma_1) \in H_0$, where $\gamma_1$ induces a graph automorphism on the $\text{L}_3^{\e}(q)$ factor of ${\rm soc}(H_0)$ (again, see Remark \ref{r:nmr} for the existence of $t$). To determine the $T$-class of $t$, we appeal to \cite[9.8]{Tho}, which shows that $C_{\bar{G}}(t) = A_1E_7$ and thus $t$ is a $t_8$-involution. The result follows.
\end{proof}

\section{Proof of Theorem \ref{t:main2}, Part II}\label{s:subfield}

In this section we turn to the subgroups arising in part (II) of Theorem \ref{t:types}; our main result is Theorem \ref{t:subfield} below. We continue to assume that $T$ is one of the groups in \eqref{e:T}. The possibilities for $H$ are as follows: 
\begin{itemize}\addtolength{\itemsep}{0.2\baselineskip}
\item[{\rm (a)}] $H$ is a subfield subgroup of $G$ (defined over a subfield $\mathbb{F}_{q_0}$ of $\mathbb{F}_{q}$ of prime index);
\item[{\rm (b)}] $T = E_6(q)$, $H_0 = {}^2E_6(q_0)$ and $q=q_0^2$;
\item[{\rm (c)}] $T = F_4(q)$, $H_0 = {}^2F_4(q)$ and $q=2^{2m+1}$ (with $m \geqs 0$);\item[{\rm (d)}] $T = G_2(q)$, $H_0 = {}^2G_2(q)$ and $q=3^{2m+1}$ (with $m \geqs 0$).
\end{itemize}

\begin{thm}\label{t:subfield}
Let $G$ be an almost simple primitive permutation group of degree $n$, with socle $T$ and point stabilizer $H$. Assume $T$ is one of the groups in \eqref{e:T} and $H$ is of type {\rm (II)} in Theorem \ref{t:types}. Then either ${\rm ifix}(T)>n^{1/2}$, or 
$n^{\a}< {\rm ifix}(T)<n^{4/9}$ and 
$(T,H_0,\a)$ is one of the cases in Table $\ref{tabb:sub}$.
\end{thm}

{\small
\renewcommand{\arraystretch}{1.2}
\begin{table}
\begin{center}
$$\begin{array}{llcl}\hline
T & H_0 & \a & \mbox{Conditions} \\ \hline   
G_2(q) & G_2(q_0) & 3/7 & \mbox{$q$ odd} \\ 
& {}^2G_2(q) & 2/5 & q=3^{2m+1}, \, m \geqs 0 \\
\hline \end{array}$$
\caption{Type (II) actions with ${\rm ifix}(T) \leqs n^{1/2}$}
\label{tabb:sub}
\end{center}
\end{table}
\renewcommand{\arraystretch}{1}}

\begin{proof}
First assume $H$ is a subfield subgroup of $G$. A similar argument applies in each case, so for brevity we illustrate the general approach with an example. Suppose $T = E_8(q)$ and $H_0 = E_8(q_0)$, where $q=q_0^k$ with $k$ prime. First assume $p \ne 2$ and let $t \in H_0$ be an involution, so $t$ is of type $t_1$ or $t_8$ as an element of $H_0$ (see Table \ref{tabinv}), and we note that the $T$-class of $t$ has the same label (that is, $t^T \cap H_0 = t^{H_0}$). In particular, if $t$ is a $t_8$-involution, then the expression for ${\rm fix}(t)$ in \eqref{e:fix} implies that ${\rm fix}(t)>n^{1/2}$ if and only if
\[
\frac{|{\rm SL}_{2}(q)||E_7(q)|}{|E_8(q)|^{1/2}} > \frac{|{\rm SL}_{2}(q_0)||E_7(q_0)|}{|E_8(q_0)|^{1/2}}.
\]
One checks that the expression on the left hand side is increasing as a function of $q$ and the inequality is easily verified since $q \geqs q_0^2$. Similarly, if $p=2$ and $t \in H_0$ is an $A_1$-involution, then $|C_T(t)| = q^{57}|E_7(q)|$, $|C_{H_0}(t)| = q_0^{57}|E_7(q_0)|$ and once again it is straightforward to check that ${\rm fix}(t)>n^{1/2}$. 

To complete the proof, it remains to consider the cases labelled (b), (c) and (d) above. The argument in (b) is almost identical to the subfield case $E_6(q^{1/2})<E_6(q)$ and we omit the details. In (c), if we take an involution $t \in H_0$ in the class $(\tilde{A}_1)_2$, then the $T$-class of $t$ has the same label (see \cite[Table II]{Shin}) and the result quickly follows. Finally, in (d) we note that both $T$ and $H_0$ have a unique class of involutions and it is simple to check that 
${\rm ifix}(T) = q(q^2-1)$ and $n = q^3(q^3-1)(q+1)$. The result follows.
\end{proof}

\begin{rem}
It is straightforward to improve the bound ${\rm ifix}(T)>n^{1/2}$ in Theorem \ref{t:subfield}, if needed. For example, suppose $T = E_8(q)$, $H_0 = E_8(q_0)$ and $q = q_0^k$ is even. If we take an $A_1$-involution $t \in H_0$ as above, then  
\[
{\rm fix}(t) = \frac{q^{57}|E_7(q)|}{q_0^{57}|E_7(q_0)|}, \;\; n  = \frac{|E_8(q)|}{|E_8(q_0)|}
\]
and one checks that ${\rm ifix}(T)>n^{95/124}$. Similarly, if $q$ is odd then ${\rm ifix}(T)>n^{17/31}$.
\end{rem}

\section{Proof of Theorem \ref{t:main2}, Part III}\label{s:exotic}

Next we handle the exotic local subgroups arising in part (III) of Theorem \ref{t:types}. The relevant subgroups are listed in Table \ref{tab:exotic} (see \cite[Table 1]{CLSS}). Note that the case $T = E_7(q)$ with $H_0 = (2^2 \times {\rm P\O}_{8}^{+}(q).2).{\rm Sym}_{3}$ has already been handled in Section \ref{s:algebraic}, so it is omitted in Table \ref{tab:exotic}.

{\small
\renewcommand{\arraystretch}{1.2}
\begin{table}
\begin{center}
$$\begin{array}{llll}\hline
& T & H_0 & \mbox{Conditions} \\ \hline
{\rm (a)} & G_2(q) & 2^3.{\rm SL}_{3}(2) & q=p \geqs 3 \\
{\rm (b)} & F_4(q) & 3^3.{\rm SL}_{3}(3) & q=p \geqs 5 \\
{\rm (c)} &E_6^{\e}(q) & 3^{3+3}.{\rm SL}_{3}(3) & q=p \geqs 5,\; q \equiv \e \imod{3} \\
{\rm (d)} &E_8(q) & 2^{5+10}.{\rm SL}_{5}(2) & q=p \geqs 3 \\
{\rm (e)} &E_8(q) & 5^3.{\rm SL}_{3}(5) & p\ne 2,5,\, q = \left\{\begin{array}{ll} p & p \equiv \pm 1 \imod{5} \\
p^2 & p \equiv \pm 2 \imod{5}
\end{array}\right. \\
{\rm (f)} &{}^2E_6(2) & {\rm U}_{3}(2) \times G_2(2) & \\
{\rm (g)} &E_7(3) & {\rm L}_{2}(3) \times F_4(3) & \\ \hline
\end{array}$$
\caption{Exotic local maximal subgroups}
\label{tab:exotic}
\end{center}
\end{table}
\renewcommand{\arraystretch}{1}}

\begin{thm}\label{t:exotic}
Let $G$ be an almost simple primitive permutation group of degree $n$, with socle $T$ and point stabilizer $H$. Assume $T$ is one of the groups in \eqref{e:T} and $H$ is of type {\rm (III)} in Theorem \ref{t:types}. Then ${\rm ifix}(T)>n^{\a}$ and either $\a \geqs 1/2$, or 
$(T,H_0,\a)$ is one of the cases in Table $\ref{tabb:loc}$. In particular, ${\rm ifix}(T)>n^{4/9}$ unless $T = G_2(q)$, $H_0 = 2^3.{\rm SL}_{3}(2)$ and $q=p \geqs 11$.
\end{thm}

{\small
\renewcommand{\arraystretch}{1.2}
\begin{table}
\begin{center}
$$\begin{array}{llcl}\hline
T & H_0 & \a & \mbox{Conditions} \\ \hline  
G_2(q) & 2^3.{\rm SL}_{3}(2) & 3/7 & q=p \geqs 3 \\ 
F_4(q) & 3^3.{\rm SL}_{3}(3) & 6/13 & q=p \geqs 5 \\
E_6^{\e}(q) & 3^{3+3}.{\rm SL}_{3}(3) & 19/39 & q=p \geqs 5,\, q \equiv \e \imod{3} \\
E_8(q) & 5^3.{\rm SL}_{3}(5) & 15/31 & p \ne 2,5,\, q = \left\{\begin{array}{ll} p & p \equiv \pm 1 \imod{5} \\
p^2 & p \equiv \pm 2 \imod{5} 
\end{array}\right. \\ 
\hline \end{array}$$
\caption{Type (III) actions with ${\rm ifix}(T) \leqs n^{1/2}$}
\label{tabb:loc}
\end{center}
\end{table}
\renewcommand{\arraystretch}{1}}

\begin{proof}
First consider case (a) in Table \ref{tab:exotic}. Here $H_0 = {\rm AGL}_{3}(2)$ and \eqref{e:i2} holds. Since $i_2(H_0) = 91$ it follows that $n^{3/7}<{\rm ifix}(T)<n^{1/2}$ and one checks that ${\rm ifix}(T)>n^{4/9}$ if $q \in \{3,5,7\}$.

Next let us turn to (b). Here $H_0 = {\rm ASL}_{3}(3)$ has a unique class of involutions, say $t^{H_0}$, with $|t^{H_0}| = 1053$. By a theorem of Borovik \cite{Bor} (also see \cite[Theorem 1]{LSJGT}), $H_0$ acts irreducibly on the Lie algebra $V = \mathcal{L}(\bar{G})$. By inspecting the character table of $H_0$ (with the aid of {\sc Magma}, for example), we deduce that $t$ has trace $-4$ on $V$, so $t$ has Jordan form $[-I_{28},I_{24}]$ and thus $t$ is a $t_1$-involution in $T$. It follows that $n^{6/13}<{\rm ifix}(T)<n^{1/2}$.

In (c), first note that $H_0 = 3^{3+3}.{\rm SL}_{3}(3)$ is a maximal parabolic subgroup of $\O_7(3)$ (see the final paragraph in the proof of \cite[Lemma 2.11]{CLSS}). One checks that $H_0$ has a unique class of involutions $t^{H_0}$ (of size $9477$) and by inspecting the character table of $H_0$ (constructed via {\sc Magma}), using the fact that $H_0$ acts irreducibly on $\mathcal{L}(\bar{G})$, we deduce that $t$ is a $t_2$-involution. Therefore, $n^{19/39}<{\rm ifix}(T) < n^{1/2}$. Case (e) is very similar. Here $H_0 = {\rm ASL}_{3}(5)$ has a unique class of involutions $t^{H_0}$ of size $19375$. Again, $H_0$ acts irreducibly on $\mathcal{L}(\bar{G})$ and from the character table of $H_0$ we deduce that $t$ is a $t_1$-involution. This gives $n^{15/31}<{\rm ifix}(T) < n^{1/2}$.

Now let us turn to (d). Here $H_0 = N_T(E)$, where $E$ is elementary abelian of order $2^5$. By inspecting the proof of \cite[Lemma 2.17]{CLSS} we see that there is an involution $t \in H_0 \setminus E$ which is contained in the $t_8$-class of $T$ (in the notation of \cite{CLSS}, we can take $t = e_1$, a generator of the centre of a fundamental ${\rm SL}_{2}$ subgroup of $\bar{G} = E_8$). The bound ${\rm ifix}(T)>n^{1/2}$ quickly follows.

Finally, let us consider cases (f) and (g) in Table \ref{tab:exotic}. In (f), set $t = (1,A_1) \in H_0$ (that is, we take $t$ to be an $A_1$-involution in the $G_2(2)$ factor) and note that 
$|t^{H_0}| = 63$. Then $t$ is an $A_1$-involution in $T$ and we quickly deduce that ${\rm ifix}(T)>n^{1/2}$. Now let us turn to case (g). Let $t \in H_0$ be a $t_1$-involution in the $F_4(3)$ factor (that is, $t= (1,t_1) \in H_0$). We claim that $t$ is a $t_1$-involution in $T$. To see this, let $V = \mathcal{L}(\G)$ be the Lie algebra of $\G$ and let $\bar{H} = A_1F_4$ be a  $\s$-stable subgroup of $\bar{G}$ such that $H_0 = N_T(\bar{H}_{\s})$. Then 
\[
V{\downarrow}A_1F_4 = \mathcal{L}(A_1F_4) \oplus (W(2) \otimes W(\lambda_4)),
\]
where $W(\mu)$ is the Weyl module of highest weight $\mu$. Now $\dim C_{\bar{H}}(t) = 27$, so $t$ has Jordan form $[-I_{28},I_{27}]$ on $\mathcal{L}(A_1F_4)$. Similarly, by applying \cite[Proposition 1.2]{LS99}, we deduce that $t$ acts as $[-I_{12},I_{14}]$ on $W(\lambda_4)$, so $[-I_{36},I_{42}]$ is the Jordan form on $W(2) \otimes W(\l_4)$. Therefore, $t$ has Jordan form $[-I_{64},I_{69}]$ on $V$, so $\dim C_{\bar{G}}(t) = 69$ and the claim follows. One now checks that ${\rm ifix}(T)>n^{1/2}$, as required.
\end{proof}

\section{Proof of Theorem \ref{t:main2}, Part IV}\label{s:as}

In this section we study the subgroups in part (IV) of Theorem \ref{t:types}, which allows us to complete the proof of Theorem \ref{t:main2} when $T$ is one of the groups in \eqref{e:T}. Here $H$ is almost simple and there are two cases to consider, according to the nature of the socle $S$ of $H$. Recall that $T$ is defined over $\mathbb{F}_{q}$, where $q=p^f$ and $p$ is a prime. Let ${\rm rk}(T)$ be the (untwisted) Lie rank of $T$ (and similarly ${\rm rk}(S)$ if $S$ is a group of Lie type). Let ${\rm Lie}(p)$ be the set of finite simple groups of Lie type defined over fields of characteristic $p$. The following result is part of \cite[Theorem 8]{LS03} (the value of $u(E_8(q))$ in part (ii)(c) is taken from \cite{Laww}).

\begin{thm}\label{t:simples}
Let $G$ be an almost simple group with socle $T$, where $T$ is one of the groups in \eqref{e:T}. Let $H$ be a maximal almost simple subgroup of $G$ as in part (IV) of Theorem \ref{t:types}, with socle $S$.  Then one of the following holds:
\begin{itemize}\addtolength{\itemsep}{0.2\baselineskip}
\item[{\rm (i)}] $S \not\in {\rm Lie}(p)$ and the possibilities for $S$ are described in \cite{LS99};
\item[{\rm (ii)}] $S  = H(q_0) \in {\rm Lie}(p)$, ${\rm rk}(S) \leqs \frac{1}{2}{\rm rk}(T)$ and one of the following holds: 
\begin{itemize}\addtolength{\itemsep}{0.2\baselineskip}
\item[{\rm (a)}] $q_0 \leqs 9$;
\item[{\rm (b)}] $S = {\rm L}_{3}^{\e}(16)$;
\item[{\rm (c)}] $S = {\rm L}_{2}(q_0)$, ${}^2B_2(q_0)$ or ${}^2G_2(q_0)$, where $q_0 \leqs (2,q-1)u(T)$ and $u(T)$ is defined as follows: 
{\small
\renewcommand{\arraystretch}{1.2}
\[
\begin{array}{llllll} \hline
T & G_2(q) & F_4(q) & E_6^{\e}(q) & E_7(q) & E_8(q) \\ \hline
u(T) & 12 & 68 & 124 & 388 & 1312 \\ \hline
\end{array}
\]
\renewcommand{\arraystretch}{1}}
\end{itemize}
\end{itemize}
\end{thm}

Our main result is the following.

\begin{thm}\label{t:mainas}
Let $G$ be an almost simple primitive permutation group of degree $n$, with socle $T$ and point stabilizer $H$. Assume $T$ is one of the groups in \eqref{e:T} and $H$ is of type {\rm (IV)} in Theorem \ref{t:types}. Then ${\rm ifix}(T)>n^{\a}$ and either $\a \geqs 4/9$, or 
$\a=3/7$ and $(T,H_0)$ is one of the cases in Table $\ref{tabb:as}$. 
\end{thm}

{\small
\renewcommand{\arraystretch}{1.2}
\begin{table}
\begin{center}
\[\begin{array}{lll}\hline
T & H_0 & \mbox{Conditions} \\ \hline  
G_2(q) & {\rm L}_{2}(13)  & q \geqs 17, \, p \ne 13, \, \mathbb{F}_{q} = \mathbb{F}_{p}[\sqrt{13}] \\ 
& {\rm L}_{2}(8)  & q \geqs 23, \, p \geqs 5, \, \mathbb{F}_{q} = \mathbb{F}_{p}[\omega],\, \omega^3-3\omega+1=0 \\ 
& {\rm U}_{3}(3){:}2  & q=p \geqs 11 \\ \hline 
\end{array}\]
\caption{Type (IV) actions with ${\rm ifix}(T) \leqs n^{4/9}$}
\label{tabb:as}
\end{center}
\end{table}
\renewcommand{\arraystretch}{1}}

\subsection{The case $S \not\in {\rm Lie}(p)$}

Let us start by considering the case where $S$ is not a simple group of Lie type over a field of characteristic $p$ (in the literature, $H$ is said to be \emph{non-generic}). The possibilities for $T$ and $S$ (up to isomorphism) are determined in \cite{LS99}, with the relevant cases recorded in a sequence of tables (see \cite[Section 10]{LS99}). 

\begin{lem}\label{l:g2as}
Theorem \ref{t:mainas} holds if $T = G_2(q)$ and $S \not\in {\rm Lie}(p)$.
\end{lem}

\begin{proof}
The maximal subgroups of $G$ are determined up to conjugacy in \cite{Coop} ($q$ even) and \cite{K88} ($q$ odd) and the relevant cases are listed in Table \ref{t:g2t} (as usual, $H_0 = H \cap T$). The result for $q \leqs 4$ can be checked directly, using {\sc Magma}, so let us assume $q \geqs 5$. Then \eqref{e:i2} holds (since $q$ is odd) and the result follows by computing $i_2(H_0)$. 
\end{proof}

{\small
\renewcommand{\arraystretch}{1.2}
\begin{table}
\begin{center}
\[\begin{array}{ll} \hline
H_0 & \mbox{Conditions} \\ \hline
{\rm L}_{2}(13) & p \ne 13, \, \mathbb{F}_{q} = \mathbb{F}_{p}[\sqrt{13}] \\
{\rm L}_{2}(8) & p \geqs 5, \,  \mathbb{F}_{q} = \mathbb{F}_{p}[\omega],\, \omega^3-3\omega+1 = 0 \\
{\rm U}_{3}(3){:}2 & q= p \geqs 5 \\ 
{\rm J}_{1} & q=11 \\
{\rm J}_{2} & q=4 \\ \hline
\end{array}\]
\caption{Maximal almost simple subgroups, $T=G_2(q)$, $S \not\in {\rm Lie}(p)$}
\label{t:g2t}
\end{center}
\end{table}
\renewcommand{\arraystretch}{1}}

For the remaining groups, we adopt the following approach. First we inspect \cite{LS99} to determine the (finite) list of possibilities for $S$. Given such a subgroup $S$, set 
\begin{align}
\begin{split}\label{e:abcd0}
a & = \max\{|t^S| \,:\, \mbox{$t \in S$ is an involution}\} \\
b & = \max\{|t^T| \,:\, \mbox{$t \in T$ is an involution}\} \\
c & = |{\rm Aut}(S)| \\
d & = |T|
\end{split}
\end{align}
If we choose an involution $t \in S$ with $|t^S|=a$, then $|t^T| \leqs b$ and thus \eqref{e:fix} implies that ${\rm ifix}(T) \geqs {\rm fix}(t) \geqs na/b$. Therefore,
\begin{equation}\label{e:useful}
\frac{\log {\rm ifix}(T)}{\log n} \geqs 1 - \frac{\log b - \log a}{\log n}  \geqs 1 - \frac{\log b - \log a}{\log d - \log c}
\end{equation}
since $n \geqs d/c$. It is easy to compute this lower bound and this quickly reduces the problem to a handful of cases that require closer attention.

\begin{prop}\label{p:nlie}
Theorem \ref{t:mainas} holds if $S \not\in {\rm Lie}(p)$.
\end{prop}

\begin{proof}
By the previous lemma, we may assume $T \ne G_2(q)$. To illustrate the approach outlined above, suppose $T = F_4(q)$ and note that  
\begin{equation}\label{e:f4c}
b = \left\{\begin{array}{ll}
q^{14}(q^4+q^2+1)(q^4+1)(q^6+1) & \mbox{$q$ odd} \\
q^{4}(q^4+q^2+1)(q^8-1)(q^{12}-1) & \mbox{$q$ even}
\end{array}\right.
\end{equation}
(see Table \ref{tabinv}). From \cite{LS99}, we deduce that
$$S \in \{{\rm Alt}_{\ell}, {\rm L}_{2}(r), {\rm L}_{3}^{\e}(3), {\rm L}_{4}(3), {}^3D_4(2), {\rm M}_{11}, {\rm J}_1, {\rm J}_2 \}$$
with $\ell \in \{5, \ldots, 10\}$ and $r \in \{7,8,13,17,25,27\}$. We now check that the lower bound in \eqref{e:useful} yields ${\rm ifix}(T) > n^{4/9}$, unless $q=2$ and $S$ is one of the following:
\[
{\rm Alt}_{6}, \, {\rm Alt}_{8}, \, {\rm Alt}_{9}, \, {\rm L}_{2}(7), \, {\rm L}_{2}(27), \, {\rm L}_{3}^{\e}(3), \, {\rm L}_{4}(3).
\]
In \cite{NW}, Norton and Wilson determine the maximal subgroups of $F_4(2)$ and its automorphism group, and we see that $H_0 = {\rm L}_{4}(3){:}2$ is the only option. There are two classes of involutions in $S$ and we choose $t \in S$ in the smaller class, so $|t^S| = 2106$. Now $S$ has a $26$-dimensional irreducible module $V$ over $\mathbb{F}_{2}$, which we can identify with the minimal module for $T$ (as noted in \cite[Corollary 2]{LSJGT}, $S$ acts irreducibly on $V$). With the aid of {\sc Magma}, we calculate that $t$ has  Jordan form $[J_2^{10}, J_1^6]$ on $V$ and by inspecting \cite[Table 3]{Lawunip} we deduce that $t$ belongs to the $T$-class labelled $\tilde{A}_1$ or $(\tilde{A}_1)_2$ in Table \ref{tabinv}. In particular, $|t^T| \leqs (q^4+1)(q^6-1)(q^{12}-1)$ and this implies that ${\rm ifix}(T) > n^{3/5}$.

Very similar reasoning applies when $T \in \{E_6^{\e}(q), E_7(q), E_8(q)\}$ and we omit the details. (Note that for $T=E_6(2)$, ${}^2E_6(2)$ and $E_7(2)$, the maximal subgroups of $G$ are given in \cite{KW}, \cite{ATLAS} and \cite{BBR}, respectively.)
\end{proof}

\subsection{The case $S \in {\rm Lie}(p)$}

To complete the proof of Theorem \ref{t:mainas}, we may assume that $S$ is a simple group of Lie type over a field $\mathbb{F}_{q_0}$ of characteristic $p$ ($H$ is said to be \emph{generic}). We will use the restrictions on both $q_0$ and ${\rm rk}(S)$ given in part (ii) of Theorem \ref{t:simples}. 

As in the previous section, for $T = G_2(q)$ we can inspect \cite{Coop, K88} to determine the possibilities for $S$. In fact, one checks there are no maximal subgroups that satisfy the conditions in Theorem \ref{t:simples}(ii) which are not of type (I) and (II) in Theorem \ref{t:types}. (Note that the maximal subgroup $H_0 = {\rm PGL}_{2}(q) < G_2(q)$ in \cite{K88} (with $p \geqs 7$) is of the form $N_G(\bar{H}_{\sigma})$ as in Theorem \ref{t:types}(I), with $\bar{H} = A_1$, so it has been handled in Section \ref{s:algebraic}.) 

The following lemma will be useful in the proof of Theorem \ref{t:mainas} for $T = F_4(q)$ (we thank David Craven for suggesting this approach).

\begin{lem}\label{l:f4ex}
Let $G$ be an almost simple group with socle $T=F_4(q)$ and let $H$ be a maximal almost simple subgroup of $G$ with socle $S = {\rm L}_{3}^{\e}(q_0)$. Then 
$$(q,q_0) \not\in \{(3,9), (4,16), (8,16)\}.$$
\end{lem}

\begin{proof}
Write $T = \bar{G}_{\s}$, where $\s$ is a Steinberg endomorphism of $\bar{G}=F_4$ and let $V$ be the Lie algebra of $\bar{G}$. Seeking a contradiction, let us assume $(q,q_0)$ is one of the three possibilities above. We proceed by arguing as in the proof of \cite[Corollary 3]{LS98}.

Fix a semisimple element $x \in S$ of order $n>u(T) = 68$ (for example, if $q=3$ we can take $n = 91$ if $\e=+$, and $n = 80$ if $\e=-$). By \cite[Proposition 2]{LS98}, there exists a positive dimensional closed subgroup $\bar{L}$ of $\bar{G}$ containing $x$ such that
\begin{itemize}\addtolength{\itemsep}{0.2\baselineskip}
\item[(a)] Some non-trivial power of $x$, say $x^k$, lies in $\bar{L}^{\circ}$; and
\item[(b)] Every $x$-invariant subspace of $V$ is also $\bar{L}$-invariant.
\end{itemize}
Set $\bar{J} = \la S, \bar{L} \ra^{\circ}$. Then $1 \ne x^k \in S \cap \bar{J}$, so $S \cap \bar{J}$ is a non-trivial normal subgroup of $S$ and thus $S<\bar{J}$ since $S$ is simple.

Let $\mathcal{M}$ be the collection of all $S$-invariant subspaces of all $K\bar{G}$-composition factors of $V$ and set
\[
\bar{Y} = \bigcap_{U \in \mathcal{M}}\bar{G}_U,
\]
where $\bar{G}_U$ denotes the setwise stabilizer of $U$ in $\bar{G}$. By 
\cite[Theorem 4]{LS98}, $\bar{Y}$ is a proper closed subgroup of $\bar{G}$ (it is also $\s$-stable by \cite[Proposition 1.12(iii)]{LS98}). Clearly, each $U \in \mathcal{M}$ is $x$-invariant, and thus $\bar{L}$-invariant by (b) above. In particular, each $U \in \mathcal{M}$ is $\bar{J}$-invariant and thus 
\[
S < \bar{J} \leqs \bar{Y} < \bar{G}.
\]

Finally, by arguing as in the proof of \cite[Theorem 6]{LS98}, we deduce that $H \leqs N_G(\bar{H}_{\s})<G$ for some maximal closed connected reductive $\s$-stable subgroup $\bar{H}$ of $\bar{G}$ with $\bar{Y}^{\circ} \leqs \bar{H}$ (see the remark following the statement of Theorem 6 in \cite{LS98}). Therefore, $H = N_G(\bar{H}_{\s})$ by the maximality of $H$. In terms of Theorem \ref{t:types}, this means that $H$ is a subgroup of type (I), but these subgroups are known and by inspection there are no examples with socle $S$ (see Table \ref{tabmaxfin}). This final contradiction completes the proof of the lemma.
\end{proof}

\begin{lem}\label{l:f4as}
Theorem \ref{t:mainas} holds if $T = F_4(q)$ and $S \in {\rm Lie}(p)$.
\end{lem}

\begin{proof}
According to Theorem \ref{t:simples}(ii), the possibilities for $S$ are recorded in Table \ref{t:f4t} (where $q_0$ is a power of $p$). To handle these cases, we start by considering the lower bound in \eqref{e:useful}. If we define $a,b,c,d$ as in \eqref{e:abcd0}, then ${\rm ifix}(T) > n^{4/9}$ if 
\begin{equation}\label{e:abcd}
\frac{d^{5/9}}{b} > \frac{c^{5/9}}{a}
\end{equation}
where $b$ is given in \eqref{e:f4c} and 
\[
d = q^{24}(q^2-1)(q^6-1)(q^8-1)(q^{12}-1).
\]

For example, suppose $S = {\rm L}_{2}(q_0)$ where $4 \leqs q_0 \leqs 136$ is odd. Note that $q^5 \geqs q_0$. In terms of \eqref{e:abcd0} we have 
\[
a \geqs \frac{1}{2}q_0(q_0-1), \;\; c \leqs 4q_0(q_0^2-1)
\]
and we deduce that the bound in \eqref{e:abcd} holds unless $(q,q_0) = (3,9)$ (it is helpful to observe that the expression on the left in \eqref{e:abcd} is increasing as a function of $q$).
Here $(a,c) = (45,1440)$ and \eqref{e:abcd} holds. Similarly, if $q_0$ is even then 
we can quickly reduce to the case $(q,q_0) = (2,4)$. However, we can rule this case out by inspecting the complete list of maximal subgroups in \cite{NW}.

{\small
\renewcommand{\arraystretch}{1.2}
\begin{table}
\begin{center}
\[\begin{array}{ll} \hline
S & \mbox{Conditions} \\ \hline
{\rm L}_{2}(q_0) & 4 \leqs q_0 \leqs 68(2,q-1) \\
{}^2B_2(q_0) & q_0 \in \{8,32\} \\
{}^2G_2(q_0) & q_0 = 27 \\
{\rm L}_{3}^{\e}(q_0)  & q_0 \leqs 16 \\
G_2(q_0), {\rm PSp}_{4}(q_0), \Omega_5(q_0) & q_0 \leqs 9 \\ \hline
\end{array}\]
\caption{Possible maximal almost simple subgroups, $T=F_4(q)$, $S \in {\rm Lie}(p)$}
\label{t:f4t}
\end{center}
\end{table}
\renewcommand{\arraystretch}{1}}

The cases $S \in \{{}^2B_2(q_0), {}^2G_2(q_0), G_2(q_0), \Omega_5(q_0)\}$ are all straightforward and we omit the details. Next assume $S = {\rm PSp}_{4}(q_0)$, so $q_0 \leqs 9$ and $q^3 \leqs q_0$. If $q=2$ then \cite{NW} implies that $H_0 = {\rm Sp}_{4}(4){:}2$, but this is one of the maximal rank subgroups we handled in Section \ref{s:mr}, so we may assume $q \geqs 3$. Now
\[
a \geqs \frac{1}{2}q_0^3(q_0-1)(q_0^2+1),\;\; c \leqs 6q_0^4(q_0^2-1)(q_0^4-1)
\]
and thus \eqref{e:abcd} holds unless $(q,q_0) = (4,2), (3,3), (4,4)$ or $(3,9)$. In fact, in each of these cases we can verify \eqref{e:abcd} by taking precise values for $a$ and $c$ (for example, if $q_0 = 3$ then $a=270$ and $c= 51840$).

To complete the proof of the lemma, we may assume $S = {\rm L}_{3}^{\e}(q_0)$ and either $q_0 \leqs 9$ or $q_0 = 16$ (see Theorem \ref{t:simples}(ii)), so $q^4 \geqs q_0$. Note that ${\rm U}_{3}(2)$ is not simple, so $q_0 \geqs 3$ if $\e=-$. In view of \cite{NW}, we also observe that $q \geqs 3$. Now
\[
a  = \left\{\begin{array}{ll}
q_0^2(q_0^2+\e q_0+1) & \mbox{$q_0$ odd} \\
(q_0+\e)(q_0^3-\e) & \mbox{$q_0$ even,}
\end{array}\right.\;\; c = 2\log_pq_0 \cdot q_0^3(q_0^2-1)(q_0^3-\e)
\]
and one checks that \eqref{e:abcd} holds unless $(q,q_0)$ is one of the following:
\begin{equation}\label{e:a2cases}
( 3, 3 ), (4,4), ( 4, 8 ), ( 8, 8 ), ( 3, 9 ), ( 4, 16 ), (8, 16 ),
\end{equation}
where $\e=-$ if $(q,q_0) = (3,3)$ or $(8,8)$. Note that we can immediately eliminate the cases $(q,q_0) = (3,9)$, $(4,16)$ and $(8,16)$ by applying Lemma \ref{l:f4ex}.

First assume $\e=+$. If $(q,q_0) = (4,4)$ and $H_0 < {\rm Aut}(S)$ then we can take $c = 2q_0^3(q_0^2-1)(q_0^3-1)$ and we deduce that \eqref{e:abcd} holds. On the other hand, if $H_0 = {\rm Aut}(S)$ then there is an involution $t \in H_0$ with $|t^{H_0}| = 1008$, so we can take $a=1008$ and once again we see that \eqref{e:abcd} holds. The case 
$(q,q_0)=(4,8)$ is similar. Indeed, if $t \in S$ is an involution then $|C_S(t)|$ is divisible by $7$, so Lagrange's theorem implies that $t$ is an involution of type $A_1$ or $\tilde{A}_1$ (as an element of $T$). Therefore, we can take $b = (q^4+1)(q^{12}-1)$ and one easily checks that \eqref{e:abcd} holds. 

Now assume $\e=-$. Note that $(q,q_0) \ne (4,8)$ by Lagrange's theorem. Suppose $(q,q_0) = (3,3)$. If $H_0 = {\rm Aut}(H_0)$ then there is an involution $t \in H_0$ with $|t^{H_0}| = 252$, so we can take $a=252$ and one checks that \eqref{e:abcd} holds. Similarly, if $H_0 = S$ then we can take $c=|S|$ and once again \eqref{e:abcd} holds. The cases $(q,q_0) = (4,4)$ and $(8,8)$ are entirely similar (if $H_0 \ne S$, we can take $a=1040$ and $32832$, respectively).
\end{proof}

\begin{prop}
Theorem \ref{t:mainas} holds if $S \in {\rm Lie}(p)$.
\end{prop}

\begin{proof}
We may assume that $T  = E_6^{\e}(q)$, $E_7(q)$ or $E_8(q)$. We proceed as in the proof of Lemma \ref{l:f4as}, with the aim of verifying the bound in \eqref{e:abcd}. As before, the possibilities for $S$ are restricted by the conditions in part (ii) of Theorem \ref{t:simples}. Also note that some of the cases that arise will have already been handled in Sections \ref{s:mr} and \ref{s:algebraic}. For example, see Section \ref{s:mr} for the case where $T = E_6^{\e}(q)$ and $S = {\rm L}_{3}^{\e}(q^3)$. In particular, by inspecting the complete list of maximal subgroups in 
\cite{BBR, ATLAS, KW}, we observe that the relevant almost simple maximal subgroups in the cases where $T = E_6^{\e}(2)$ or $E_7(2)$ have already been handled in Sections \ref{s:mr} and \ref{s:algebraic}.

Although there are a few more possibilities for $S$ to consider, the arguments are straightforward and no special complications arise (in particular, there is no need to establish the non-maximality of certain subgroups, as in Lemma \ref{l:f4ex}). We ask the reader to check the details. 
\end{proof}

\vs

This completes the proof of Theorem \ref{t:mainas}.

\section{Completion of proof of Theorem \ref{t:main2}}\label{s:small}

In this section we complete the proof of Theorem \ref{t:main2} by handling the groups with 
\begin{equation}\label{e:T2}
T \in \{{}^2B_2(q),{}^2G_2(q), {}^2F_4(q), {}^3D_4(q)\}.
\end{equation}
In each case, the maximal subgroups of $G$ have been determined (see Suzuki \cite{Suz} for ${}^2B_2(q)$, Malle \cite{Mal} for ${}^2F_4(q)$ and Kleidman \cite{K88, Kl3} for ${}^2G_2(q)$ and ${}^3D_4(q)$). In view of Proposition \ref{p:small}, we may assume that $T \ne {}^2G_2(3)', {}^2F_4(2)'$. Moreover, since the relevant maximal rank subgroups (including parabolic subgroups) were handled in Sections \ref{s:parab} and \ref{s:mr}, the remaining possibilities for $H_0 = H \cap T$ are as follows:
\[
\begin{array}{ll}
T = {}^2B_2(q): & {}^2B_2(q_0) \mbox{ with $q_0 \geqs 8$} \\
T = {}^2G_2(q): & {}^2G_2(q_0) \\
T = {}^2F_4(q): & {}^2F_4(q_0) \\
T = {}^3D_4(q): & {}^3D_4(q_0),\, G_2(q),\; {\rm PGL}_{3}^{\e}(q)
\end{array}
\]
where $q=q_0^k$ for a prime $k$.

\begin{prop}\label{p:subb}
Let $G$ be an almost simple primitive permutation group of degree $n$ with socle $T$ and point stabilizer $H$. Suppose $T$ is one of the groups in \eqref{e:T2} and $H$ is a subfield subgroup.
Then ${\rm ifix}(T) > n^{\a}$ and 
\[
\liminf_{q\to \infty} \frac{\log {\rm ifix}(T)}{\log n} = \b,
\]
where $\a,\b$ are given in Table $\ref{tab:alb}$. In particular,  ${\rm ifix}(T) > n^{4/9}$ if and only if $T = {}^2F_4(q)$ or ${}^3D_4(q)$ (with $q$ even).
\end{prop}

{\small 
\renewcommand{\arraystretch}{1.2}
\begin{table}
\begin{center}
\[
\begin{array}{llcc} \hline
T & H_0 & \a & \b \\ \hline
{}^2B_2(q) & {}^2B_2(q_0) & 0.397 & 2/5  \\
{}^2G_2(q) & {}^2G_2(q_0) & 0.426 & 3/7   \\
{}^2F_4(q) & {}^2F_4(q_0) & 15/26 & 15/26   \\
{}^3D_4(q) & {}^3D_4(q_0) & \left\{\begin{array}{ll} 3/7 & \mbox{$q$ odd} \\
0.637 & \mbox{$q$ even}
\end{array}\right. & 3/7 \\ \hline
\end{array}
\]
\caption{The constants $\a,\b$ in Proposition \ref{p:subb}}
\label{tab:alb}
\end{center}
\end{table}
\renewcommand{\arraystretch}{1}}

\begin{proof}
If $T  = {}^2B_2(q)$ or ${}^2G_2(q)$ then both $T$ and $H_0$ have a unique class of involutions, so \eqref{e:i2} holds and the result quickly follows since
\[
i_2({}^2B_2(q)) = (q^2+1)(q-1),\;\; i_2({}^2G_2(q)) = q^2(q^2-q+1).
\]
Similarly, if $T = {}^2F_4(q)$ and $t \in H_0$ is an $(\tilde{A}_1)_2$-involution, then 
\[
{\rm ifix}(T) = {\rm ifix}(t) = \frac{|t^{H_0}|}{|t^T|}\cdot n = \frac{(q_0^3+1)(q_0^2-1)(q_0^6+1)}{(q^3+1)(q^2-1)(q^6+1)}\cdot n
\]
and it is easy to verify the desired bounds. The case $T = {}^3D_4(q)$ is entirely similar and we omit the details (note that $T$ and $H_0$ have a unique class of involutions when $q$ is odd).
\end{proof}

Finally, we handle the two remaining cases that arise when $T = {}^3D_4(q)$.

\begin{prop}\label{p:fin}
Let $G$ be an almost simple primitive permutation group of degree $n$, with socle $T = {}^3D_4(q)$ and point stabilizer $H$. Set $H_0 = H \cap T$ and assume that $H_0  = G_2(q)$ or ${\rm PGL}_{3}^{\e}(q)$. Then ${\rm ifix}(T) > n^{\a}$ and 
\[
\liminf_{q\to \infty} \frac{\log {\rm ifix}(T)}{\log n} = \b,
\]
where $\a,\b$ are given in Table $\ref{tab:fin}$. In particular, ${\rm ifix}(T)>n^{4/9}$ if and only if $q$ is even.  
\end{prop}

{\small 
\renewcommand{\arraystretch}{1.2}
\begin{table}
\begin{center}
\[
\begin{array}{llc} \hline
H_0 & \multicolumn{1}{c}{\a} & \b  \\ \hline
G_2(q) & \left\{\begin{array}{ll} 
3/7 & \mbox{$q$ odd} \\
5/7 & \mbox{$q$ even} 
\end{array}\right. & 3/7 \\
{\rm PGL}_{3}(q) & \left\{\begin{array}{ll}  
2/5 & \mbox{$q$ odd}  \\
7/10 &  \mbox{$q$ even} 
\end{array}\right. & 2/5 \\
{\rm PGU}_{3}(q) & \left\{\begin{array}{ll} 
0.387 & \mbox{$q$ odd} \\
0.672 & \mbox{$q$ even} 
 \end{array}\right. & 2/5 \\ \hline
\end{array}
\]
\caption{The constants $\a,\b$ in Proposition \ref{p:fin}}
\label{tab:fin}
\end{center}
\end{table}
\renewcommand{\arraystretch}{1}}

\begin{proof}
First assume $H_0 = G_2(q)$. If $q$ is odd then both $T$ and $H_0$ have a unique class of involutions, hence \eqref{e:i2} holds and the result follows. If $q$ is even and $t \in H_0$ is an $A_1$-involution, then $t$ is in the $A_1$-class of $T$ (see the proof of \cite[Lemma 6.3]{LLS2}, for example) and we deduce that  
\[
{\rm ifix}(T) = \frac{|t^{H_0}|}{|t^T|}\cdot n = \frac{q^6-1}{(q^2-1)(q^8+q^4+1)}\cdot n,
\]
which gives the desired result. A very similar argument applies when $H_0 = {\rm PGL}_{3}^{\e}(q)$. Indeed, we note that $H_0$ has a unique class of involutions and 
\[
i_2(H_0) = \left\{\begin{array}{ll}
(q+\e)(q^3-\e) & \mbox{ if $q$ even} \\
q^2(q^2+\e q+1) & \mbox{ if $q$ odd.}
\end{array}\right.
\]
Moreover, if $q$ is even then the involutions in $H_0$ are contained in the $A_1$-class of $T$. The result quickly follows.
\end{proof}

\vs

This completes the proof of Theorem \ref{t:main2}.

\section{Algebraic groups}\label{s:alg}

In this final section we prove Theorem \ref{t:main4}. Let $\bar{G}$ be a simple exceptional algebraic group over an algebraically closed field $K$ of characteristic $p \geqs 0$ and let $\bar{H}$ be a maximal positive dimensional closed subgroup of $\bar{G}$.
Recall that there is a natural action of $\bar{G}$ on the coset variety $\O = \bar{G}/\bar{H}$ and for each $t \in \bar{G}$, the set of fixed points $C_{\O}(t)$ is a subvariety. Set
\[
{\rm ifix}(\bar{G}) = \max\{\dim C_{\O}(t) \,:\, \mbox{$t \in \bar{G}$ is an involution}\}.
\]
The possibilities for $\bar{H}$ are determined up to conjugacy by Liebeck and Seitz in \cite[Theorem 1]{LS04} and the subgroups that arise are as follows:
\begin{itemize}\addtolength{\itemsep}{0.2\baselineskip}
\item[(a)] Parabolic subgroups (one class for each node of the Dynkin diagram of $\bar{G}$):
\item[(b)] Reductive subgroups of maximal rank (see \cite[Table 10.3]{LS04});
\item[(c)] Reductive subgroups of smaller rank (see \cite[Table 10.1]{LS04}).
\end{itemize}  
Note that these are precisely the infinite analogues of the subgroups arising in the statement of Theorem \ref{t:main3}.

\begin{proof}[Proof of Theorem \ref{t:main4}]
Set $\O = \bar{G}/\bar{H}$ and let $t \in \bar{H}$ be an involution. Then 
\[
\dim C_{\O}(t) = \dim \O - \dim t^{\bar{G}} + \dim (t^{\bar{G}}\cap \bar{H})
\]
by \cite[Proposition 1.14]{LLS}. In particular, the dimension of $C_{\O}(t)$ (and also $\dim \O = \dim \bar{G} - \dim \bar{H}$) is independent of the choice of algebraically closed field of fixed characteristic $p$. In addition, it is well known that these dimensions in characteristic zero are equal to the dimensions over any algebraically closed field of prime characteristic $r$, for all large primes $r$. Therefore, we may assume $p>0$ and $K = \bar{\mathbb{F}}_p$.

Let $\s$ be a Steinberg endomorphism of $\bar{G}$ so that $\bar{G}(q):=\bar{G}_{\s}$ is defined over $\mathbb{F}_q$ for some $p$-power $q$. We may assume that $\bar{H}$ is also $\s$-stable, so we can consider $\bar{H}(q):=\bar{H}_{\s}$ and the action of $\bar{G}(q)$ on the set of cosets 
$\Omega(q):=\bar{G}(q)/\bar{H}(q)$. Let $t \in \bar{H}(q)$ be an involution and let $e$ be the number of irreducible components of $C_{\Omega}(t)$ of maximal dimension. By \cite[Lemma 7.1]{LSh05}, which follows from the Lang-Weil estimates in \cite{LW}, we deduce that 
\[
|C_{\Omega(q)}(t)| = (e+o(1))q^{\dim C_{\Omega}(t)},\;\; |\Omega(q)| = (1+o(1))q^{\dim \Omega}
\]
for infinitely many values of $q$. Therefore, 
\[
\frac{\log |C_{\Omega(q)}(t)|}{\log |\Omega(q)|}  = \frac{\dim C_{\O}(t)}{\dim \O} + o(1)
\]
and the result now follows from Theorem \ref{t:main3}.
\end{proof}


\begin{thebibliography}{99}
\bibitem{Asch} M. Aschbacher, \emph{On the maximal subgroups of the finite classical groups}, Invent. Math. \textbf{76} (1984), 469--514.

\bibitem{AS}
M. Aschbacher and G.M. Seitz, \emph{Involutions in {C}hevalley groups over fields of even order}, Nagoya Math. J. \textbf{63} (1976), 1--91.

\bibitem{Babai}
L. Babai, \emph{On the order of uniprimitive permutation groups}, Annals of Math. \textbf{113} (1981), 553--568.

\bibitem{BBR}
J. Ballantyne, C. Bates and P. Rowley, \emph{The maximal subgroups of $E_7(2)$}, LMS J. Comput. Math. \textbf{18} (2015), 323--371.

\bibitem{BPP}
J. Bamberg, T. Popiel and C.E. Praeger, \emph{Simple groups, product actions, and generalised quadrangles}, preprint (arxiv:1702.07308)

\bibitem{Bender}
H. Bender, \emph{Transitive Gruppen gerader Ordnung, in denen jede Involution genau einen Punkt festl\"{a}sst}, J. Algebra \textbf{17} (1971), 527--554.

\bibitem{Boch} A. Bochert, \emph{\"{U}ber die Zahl der verschiedenen Werte, die eine funktion gegebener Buchstaben durch Vertauschung derselben erlangen kann}, Math. Ann. \textbf{33} (1889), 584--590.

\bibitem{Bor} 
A.V. Borovik, \emph{Jordan subgroups and orthogonal decompositions}, Algebra i Logika \textbf{28} (1989), 382--392; English transl. Algebra and Logic \textbf{28} (1989), 248--254 (1990).

\bibitem{Magma} 
W. Bosma, J. Cannon and C. Playoust, \emph{The {\sc Magma} algebra system I: The user language}, J. Symbolic Comput. \textbf{24} (1997), 235--265.

\bibitem{Bou} N. Bourbaki, \emph{Groupes et algebr\`{e}s de Lie (Chapitres 4,5 et 6)},
Hermann, Paris, 1968.

\bibitem{Bur1}
T.C. Burness, \emph{Fixed point spaces in primitive actions of simple algebraic groups}, J. Algebra \textbf{265} (2003), 744--771.

\bibitem{Bur2} T.C. Burness, \emph{Fixed point ratios in actions of finite classical groups, II}, J. Algebra \textbf{309} (2007), 80--138.

\bibitem{BG} 
T.C. Burness and M. Giudici, \emph{Classical groups, derangements and primes}, Aust. Math. Soc. Lecture Series, vol. 25, Cambridge University Press, 2016.

\bibitem{BGS}
T.C. Burness, R.M. Guralnick and J. Saxl, \emph{On base sizes for algebraic groups}, J. Eur. Math. Soc. (JEMS) \textbf{19} (2017), 2269--2341.

\bibitem{BLS} 
T.C. Burness, M.W. Liebeck and A. Shalev, \emph{Base sizes for simple groups and a conjecture of Cameron}, Proc. London Math. Soc. \textbf{98} (2009), 116--162.

\bibitem{Carter72} R.W. Carter, \emph{Simple groups of {L}ie type}, Pure and Applied Mathematics, vol. 28, John Wiley and Sons, London-New York-Sydney, 1972. 

\bibitem{Carter78} R.W. Carter, \emph{Centralizers of semisimple elements in finite groups of Lie type}, Proc. London Math. Soc. \textbf{37} (1978), 491--507.

\bibitem{Carter} R.W. Carter, \emph{Finite groups of Lie type: Conjugacy classes and complex characters}, John Wiley and Sons, New York, 1985.

\bibitem{CLSS}
A.M. Cohen, M.W. Liebeck, J. Saxl and G.M. Seitz, \emph{The local maximal subgroups of exceptional groups of {L}ie type}, Proc. London Math. Soc. \textbf{64} (1992), 21--48.

\bibitem{ATLAS}
J.H. Conway, R.T. Curtis, S.P. Norton, R.A. Parker and R.A. Wilson, \emph{Atlas of
  finite groups}, Oxford University Press, 1985.

\bibitem{Coop}
B.N. Cooperstein, \emph{Maximal subgroups of $G_2(2^n)$}, J. Algebra \textbf{70} (1981), 23--36.

\bibitem{Cov}
E. Covato, \emph{The involution fixity of simple groups}, PhD thesis, University of Bristol, 2018.

\bibitem{GLS}
D. Gorenstein, R. Lyons  and R. Solomon, \emph{The classification of the
  finite simple groups, Number 3}, Mathematical Surveys and Monographs,
  vol. 40, Amer. Math. Soc., 1998.

\bibitem{GM}
R.M. Guralnick and K. Magaard, \emph{On the minimal degree of a primitive permutation group}, J. Algebra \textbf{207} (1998), 127--145.
 
\bibitem{Jo} 
C. Jordan, \emph{Th\'{e}or\`{e}mes sure les groupes primitifs}, J. Math. Pures Appl. \textbf{16} (1871), 383--408.

\bibitem{K88}
P.B. Kleidman, \emph{The maximal subgroups of the Chevalley groups $G_2(q)$ with $q$ odd, the Ree groups ${}^2G_2(q)$, and their automorphism groups}, J. Algebra \textbf{117} (1988), 30--71.

\bibitem{Kl3}
P.B. Kleidman, \emph{The maximal subgroups of the {S}teinberg triality groups {$^{3}D_{4}(q)$} and of their automorphism groups}, J. Algebra \textbf{115} (1988), 182--199.

\bibitem{KL}
P.B. Kleidman and M.W. Liebeck, \emph{The subgroup structure of the
  finite classical groups}, London Math. Soc. Lecture Note Series, vol.
  129, Cambridge University Press, 1990.

\bibitem{KW}
P.B. Kleidman and R.A. Wilson, \emph{The maximal subgroups of {$E_{6}(2)$} and {{\rm Aut}$(E_{6}(2))$}}, Proc. London Math. Soc. \textbf{60} (1990), 266--294.
  
\bibitem{LW} S. Lang and A. Weil, \emph{Number of points of varieties over finite fields}, Amer. J. Math. \textbf{76} (1954), 819--827.
 
\bibitem{Law09}
R. Lawther, \emph{Unipotent classes in maximal subgroups of exceptional algebraic groups}, J. Algebra \textbf{322} (2009), 270--293.

\bibitem{Lawunip}
R. Lawther, \emph{Jordan block sizes of unipotent elements in exceptional algebraic groups}, Comm. Algebra \textbf{23} (1995), 4125--4156.

\bibitem{Laww}
R. Lawther, \emph{Sublattices generated by root differences}, J. Algebra \textbf{412} (2014), 255--263.

\bibitem{LLS}
R. Lawther, M.W. Liebeck and G.M. Seitz, \emph{Fixed point spaces in actions of exceptional algebraic groups}, Pacific J. Math. \textbf{205} (2002), 339--391.

\bibitem{LLS2} 
R. Lawther, M.W. Liebeck and G.M. Seitz, \emph{Fixed point ratios in actions of finite exceptional groups of Lie type}, Pacific J. Math. \textbf{205} (2002), 393--464.

\bibitem{LS91} 
M.W. Liebeck and J. Saxl, \emph{Minimal degrees of primitive permutation groups, with an application to monodromy groups of Riemann surfaces}, Proc. London Math. Soc. \textbf{63} (1991), 266--314.

\bibitem{LS90}
M.W. Liebeck and G.M. Seitz, \emph{Maximal subgroups of exceptional groups of Lie type, finite and algebraic}, Geom. Dedicata \textbf{36} (1990), 353--387.

\bibitem{LSS}
M.W. Liebeck, J. Saxl and G.M. Seitz, \emph{Subgroups of maximal rank in finite exceptional groups of Lie type}, Proc. London Math. Soc. \textbf{65} (1992), 297--325.

\bibitem{LS98}
M.W. Liebeck and G.M. Seitz, \emph{On the subgroup structure of exceptional groups of {L}ie type}, Trans. Amer. Math. Soc. \textbf{350} (1998), 3409--3482.

\bibitem{LS99} 
M.W. Liebeck and G.M. Seitz, \emph{On finite subgroups of exceptional algebraic groups}, J. reine angew. Math. \textbf{515} (1999), 25--72.

\bibitem{LS03} 
M.W. Liebeck and G.M. Seitz, \emph{A survey of of maximal subgroups of exceptional groups of Lie type}, in Groups, combinatorics \& geometry (Durham, 2001), 139--146, World Sci. Publ., River Edge, NJ, 2003.

\bibitem{LS04}
M.W. Liebeck and G.M. Seitz, \emph{The maximal subgroups of positive dimension in exceptional algebraic groups}, Mem. Amer. Math. Soc. \textbf{802}, 2004.

\bibitem{LSJGT}
M.W. Liebeck and G.M. Seitz, \emph{Subgroups of exceptional algebraic groups which are irreducible on an adjoint or minimal module}, J. Group Theory \textbf{7} (2004), 347--372.

\bibitem{LS_book} M.W. Liebeck and G.M. Seitz, \emph{Unipotent and nilpotent classes in simple algebraic groups and Lie algebras}, Mathematical Surveys and Monographs, vol.180, Amer. Math. Soc., 2012.

\bibitem{LSh05}
M.W. Liebeck and A. Shalev, \emph{Fuchsian groups, finite simple groups and representation varieties}, Invent. Math. \textbf{159} (2005), 317--367.

\bibitem{LSh11}
M.W. Liebeck and A. Shalev, \emph{Fixed points of elements of linear groups}, Bull. London Math. Soc. \textbf{43} (2011), 897--900.

\bibitem{LSh}
M.W. Liebeck and A. Shalev, \emph{On fixed points of elements in primitive permutation groups}, J. Algebra \textbf{421} (2015), 438--459.

\bibitem{Lub} F. L\"{u}beck, \emph{Generic computations in finite groups of Lie type}, book in preparation.

\bibitem{Lus} G. Lusztig, \emph{Character sheaves V}, Adv. Math. \textbf{61} (1986), 103--155.

\bibitem{Mal0}
G. Malle, \emph{Die unipotenten charaktere von {$\,^{2}F_{4}(q^{2})$}}, Comm. Algebra  \textbf{18} (1990), 2361--2381.

\bibitem{Mal}
G. Malle, \emph{The maximal subgroups of {$\,^{2}F_{4}(q^{2})$}}, J. Algebra \textbf{139} (1991), 52--69.

\bibitem{MT}
G. Malle and D. Testerman, \emph{Linear algebraic groups and finite
  groups of {L}ie type}, Cambridge Studies in Advanced Mathematics, vol.133, Cambridge University Press, 2011.

\bibitem{NW} S.P. Norton and R.A. Wilson, \emph{The maximal subgroups of $F_4(2)$ and its automorphism group}, Comm. Algebra \textbf{17} (1989), 2809--2824.

\bibitem{RicSpr} R.W. Richardson and T.A. Springer, \emph{The Bruhat order on symmetric varieties}, Geom. Dedicata \textbf{35} (1990), 389--436.

\bibitem{Ronse} C. Ronse, \emph{On permutation groups of prime power order}, Math. Z. \textbf{173} (1980), 211--215.

\bibitem{SS} J. Saxl and A. Shalev, \emph{The fixity of permutation groups}, J. Algebra \textbf{174} (1995), 1122--1140.

\bibitem{Shin}
K. Shinoda, \emph{The conjugacy classes of the finite {R}ee groups of type
  {$(F_{4})$}}, J. Fac. Sci. Univ. Tokyo \textbf{22} (1975), 1--15.

\bibitem{Spal}
N. Spaltenstein, \emph{Carat\`{e}res unipotents de ${}^3D_4(\mathbb{F}_{q})$}, Comment. Math. Helv. \textbf{57} (1982), 676--691.

\bibitem{Suz}
M. Suzuki, \emph{On a class of doubly transitive groups}, Annals of Math. \textbf{75} (1962),  105--145.

\bibitem{Tho}
A.R. Thomas, \emph{The irreducible subgroups of the exceptional algebraic groups}, Mem. Amer. Math. Soc., to appear.
  
\bibitem{WebAt}
R.A. Wilson et al., \emph{A {W}orld-{W}ide-{W}eb {A}tlas of finite group representations},  {\texttt{http://brauer.maths.qmul.ac.uk/Atlas/v3/}}

\end{thebibliography}
\end{document}